\documentclass[a4paper,11pt]{article}

\usepackage{amsmath}
\usepackage{amssymb}
\usepackage{amsthm}
\usepackage{amscd}
\usepackage{epic, eepic}
\usepackage{url}
\usepackage[dvipsnames]{xcolor}
\usepackage[utf8]{inputenc}
\usepackage{comment}
\usepackage{enumerate}
\usepackage{graphicx}
\usepackage{capt-of}
\usepackage{epstopdf}
\usepackage{enumitem}
\usepackage{todonotes}
\usepackage{soul}
\usepackage{thmtools}
\usepackage{thm-restate}
\declaretheorem[name=Theorem,numberwithin=section]{theorem}

\usepackage[top=22mm, bottom=23mm, left=21mm, right=21mm]{geometry}
\usepackage[hidelinks]{hyperref}
\usepackage[capitalize]{cleveref}
\usepackage[normalem]{ulem}
\newcommand{\stkout}[1]{\ifmmode\text{\sout{\ensuremath{#1}}}\else\sout{#1}\fi}
\usepackage{tikz}
\usepackage[allcommands]{overarrows}
\usetikzlibrary{math,calc,intersections}
\usetikzlibrary{positioning,arrows,shapes,decorations.markings,decorations.pathreplacing, decorations.pathmorphing,matrix,patterns}
\tikzstyle{vertex}=[circle,draw=black,fill=black,inner sep=0,minimum size=5pt,text=white,font=\footnotesize]

\makeatletter
\renewenvironment{proof}[1][\proofname] {\par\pushQED{\qed}\normalfont\topsep6\p@\@plus6\p@\relax\trivlist\item[\hskip\labelsep\textit{#1}\@addpunct{.}]\ignorespaces}{\popQED\endtrivlist\@endpefalse}
\makeatother

\newtheorem{lemma}[theorem]{\bf Lemma}
\newtheorem{claim}[theorem]{\bf Claim}
\newtheorem{corollary}[theorem]{\bf Corollary}
\newtheorem{proposition}[theorem]{\bf Proposition}
\newtheorem{conjecture}[theorem]{\bf Conjecture}

\newcommand{\claimproofstart}[1][Proof]{\begin{proof}[#1]
\renewcommand{\qedsymbol}{$\boxdot$}}
\newcommand\claimproofend{
\end{proof}
\renewcommand{\qedsymbol}{$\square$}}
\theoremstyle{definition}

\newtheorem{definition}[theorem]{\bf Definition}

\newcounter{propcounter}

\def\eps{\varepsilon}
\def\ep{\varepsilon}

\def\cA{\mathcal{A}}
\def\cC{\mathcal{C}}

\def\cH{\mathcal{H}}
\def\cP{\mathcal{P}}

\def\oc{\mathrm{oc}}
\def\ec{\mathrm{ec}}

\def\de{\delta}
\newcommand{\N}{\mathbb{N}}

\title{\vspace{-0.9cm}Minimising the harmonic sum of cycle lengths}
\author{Aleksa Milojevi\'c\thanks{Department of Mathematics, ETH, Z\"urich, Switzerland. Research supported in part by SNSF grant 200021-228014. {Email}: {\tt \{aleksa.milojevic,benjamin.sudakov\}@math.ethz.ch}.
} \and  Richard Montgomery\thanks{Mathematics Institute, University of Warwick, Coventry, UK. Research supported by the European Research Council (ERC) under the European Union Horizon 2020 research and innovation programme (grant agreement No.\ 947978). Email: {\tt richard.montgomery@warwick.ac.uk}.}
 \and Alexey Pokrovskiy\thanks{Department of Mathematics, University College London, London, UK. {Email}: {\tt dralexeypokrovskiy@gmail.com.}}\and  Benny Sudakov\footnotemark[1]}
\date{}

\begin{document}

\maketitle

\begin{abstract}
A central theme in extremal graph theory is to understand the relationship between the density of a graph and the richness of its cycle length spectrum, which is the set of distinct cycle lengths occurring in the graph. In 1966, Erd\H{o}s and Hajnal suggested studying $s(G):=\sum_{\ell\in\cC(G)}1/\ell$ as a measure of the richness of the cycle length spectrum $\cC(G)$ of a graph $G$. 

Through a series of increasingly strong conjectures, Erd\H{o}s suggested that the complete bipartite graphs minimise $s(G)$ among all graphs $G$ with the same average degree. The sharpest such conjecture, from 1981, states that the graph $K_{k, n-k}$ minimises $s(G)$ among all $n$-vertex graphs with at least $k(n-k)$ edges (where $k\leq n/2$). We prove this conjecture for all sufficiently large $k$, by showing the stronger statement that any $n$-vertex graph $G$ with $e(G)>(k-1)(n-k+1)$ and $n\geq 2k$ satisfies
$s(G)\geq\sum_{\ell=2}^{k}1/(2\ell)$. Moreover, we show that the complete bipartite graph $K_{k,n-k}$ is the unique graph with at least $k(n-k)$ edges that achieves equality here.
\end{abstract}

\section{Introduction}\label{sec:intro}
The \emph{cycle length spectrum} of a graph $G$ is the set of its cycle lengths, denoted
\[
    \cC(G):=\{\ell\in\mathbb N:
    G\text{ contains a cycle of length }\ell\}.
\]
The study of the richness of the cycle length spectrum in relation to the density of a graph is a classical theme in extremal graph theory.
For example, given the average degree $d(G)$ of a graph $G$, must $\mathcal{C}(G)$ contain representatives of prescribed congruence classes~\cite{Bollobas77}, contain long sequences of consecutive even lengths~\cite{Verstraete00,GHS02,LM18}, or meet prescribed sparse sequences of integers~\cite{Verstraete05,SV08,liu2023solution}?
As $d(G)$ may be arbitrarily large yet $G$ contain only even cycles, corresponding questions have been asked for the \emph{odd} cycle lengths using a chromatic number condition. 
This type of question is very different from asking how many cycles are contained in a dense graph, since a graph may contain a very large number of cycles while realizing only a small collection of cycle lengths. Thus,  $\cC(G)$ captures the distribution and diversity of cycle lengths, rather than the sheer number of cycles. We refer the reader to the survey of Verstra\"ete~\cite{VerstraeteSurvey16} for a broader account of extremal questions concerning cycle lengths.

A natural quantitative measure of the richness of the cycle length spectrum
was proposed by Erd\H{o}s and Hajnal in 1966~\cite{EH66}. They considered the
harmonic sum of the cycle lengths, that is,
\begin{equation}\label{eq:harmonic-cycle-sum}
    s(G):=\sum_{\ell\in\cC(G)}\frac1\ell.
\end{equation}
The measure $s(G)$ treats different multiplicative scales similarly: the total harmonic mass of integers in the interval $[a, 2a]$ is approximately $\log 2$. So, for the sum $s(G)$ to be large, $G$ must contain many cycle lengths on many different scales, with shorter cycles receiving larger weight. In this sense, $s(G)$ measures not only how many distinct cycle lengths occur in $G$, but also how broadly they are distributed. The harmonic measure also appears naturally in other areas of combinatorics, for example in the Erd\H{o}s--Tur\'an conjecture on arithmetic progressions~\cite{ErdosTuran1936}.

Erd\H{o}s and Hajnal \cite{EH66} asked whether $s(G)$ must diverge as the chromatic number of $G$, $\chi(G)$, grows. By the discussion above, this is substantially stronger than simply finding many different cycle lengths: it asks to show that the graph contains \textit{many cycle lengths on many scales}. 
Erd\H{o}s~\cite{E75} later suggested that $s(G)$ diverges even if `chromatic number' is replaced by `average degree' in the assumption, and suggested moreover that any graph $G$ with average degree at least $d$ satisfies
\begin{equation}\label{eq:erdos-asymptotic-conjecture}
    s(G)\geq \left(\frac{1}{2}-o_d(1)\right)\log d.
\end{equation}
The complete balanced bipartite graph with $d\geq 2$ vertices on each side shows that the constant $\frac12$ would be optimal here, as it has cycles of every even length from $4$ up to $2d$, and has the corresponding harmonic sum $(\frac12+o_d(1))\log d$.

In the 1980s, Gy\'arf\'as, Koml\'os and Szemer\'edi~\cite{GKS84} confirmed the original conjecture of Erd\H{o}s and Hajnal under the average degree condition. More specifically, they showed that there is some $c>0$ such that every graph $G$ with average degree $d$ has $s(G)\geq c\log d$. This confirmed that graphs with average degree $d$ cannot have their cycle lengths concentrated in fewer than $\Omega(\log d)$ intervals $[a,2a]$, and thus their cycle spectra must be rich at many different scales.

This progress led to an even stronger conjecture, which Erd\H{o}s proposed in his paper \textit{``On the combinatorial problems which I would most like to see solved''}. This article, specially prepared for the inaugural issue of Combinatorica in 1981, aimed to present a selection of problems Erd\H{o}s believed were particularly important for the development of combinatorics and contained many highly influential problems. 

The observation behind strengthening the conjecture was the following: while the complete bipartite graph $K_{d,d}$ demonstrates that the constant $\frac{1}{2}$ in \eqref{eq:erdos-asymptotic-conjecture} would be tight, the average degree of $K_{d,d}$ can be increased without adding more cycle lengths by adding vertices to one of its classes. That is, when $n\geq 2d$, $K_{d,n-d}$ has the harmonic sum of cycle lengths $s(K_{d,n-d})=\sum_{\ell=2}^d\frac{1}{2\ell}\approx \frac{1}{2}\log d$ and average degree $d(K_{d,n-d})$ which tends to $2d$ when $d$ is fixed and $n$ grows. The natural conjecture is then that here $K_{k,n-k}$ minimises $s(G)$  among all the $n$-vertex graphs with at least $k(n-k)$ edges, as follows.

\begin{conjecture}[Erd\H{o}s~\cite{E81}]\label{conj:main}
Let $k,n\in\mathbb N$ satisfy $1\leq k\leq n/2$.  Among all graphs $G$ with
$n$ vertices and at least $k(n-k)$ edges, the complete bipartite graph
$K_{k,n-k}$ minimises $s(G)$.  Equivalently, every such graph satisfies
\begin{equation}\label{eq:erdos-exact-conj}
    \sum_{\ell\in\cC(G)}\frac1\ell
       \geq \sum_{\ell=2}^{k}\frac1{2\ell}.
\end{equation}
\end{conjecture}

Despite consistent attention to this problem over the intervening decades~\cite{GPSV85,SV08,Verstraete05}, it took more than 40 years for further progress. Finally, in 2023, Liu and
Montgomery~\cite{liu2023solution} confirmed the asymptotic prediction of Erd\H{o}s in \eqref{eq:erdos-asymptotic-conjecture}, in work which also solved the long-standing Erd\H{o}s--Hajnal odd-cycle
problem. To do so, they developed methods for constructing paths and cycles of prescribed lengths using sublinear expansion. This is a tool introduced in the 1990s by Koml\'os and Szemer\'edi~\cite{K-Sz-1,K-Sz-2} which in recent years has developed into a central framework in extremal graph theory (see the surveys by Letzter~\cite{Letzter2024SublinearExpanders} and Montgomery~\cite{Montgomery2026GraphTheoryExpansion} for more details). 

With the confirmation of the asymptotic lower bound in \eqref{eq:erdos-asymptotic-conjecture}, it is worth reflecting that the weighting given to the cycle lengths in $s(G)$ belies how close this is to Conjecture~\ref{conj:main}. Indeed, to nudge the lower bound on $s(G)$ up by even a constant effectively may require us to find a positive density of cycle lengths on a new scale. 

In this paper, we resolve Conjecture~\ref{conj:main} for all sufficiently large $k$.
In fact, we establish that $K_{k,n-k}$ minimises $s(G)$ even over graphs with a weaker edge condition than in Conjecture~\ref{conj:main}, and moreover show that $K_{k,n-k}$ is the only extremal example for the original conjecture, as follows.

\begin{theorem}\label{thm:main}
Let $k$ be a sufficiently large integer, and let $G$ be a graph on $n$ vertices,
where $n\geq2k$, with more than $(k-1)(n-k+1)$ edges.  Then
\begin{equation}\label{eq:suminmainthm}
    \sum_{\ell\in\cC(G)}\frac1\ell
       \geq \sum_{\ell=2}^{k}\frac1{2\ell}.
\end{equation}
Moreover, if $e(G)\geq k(n-k)$, then equality holds if and only if $G$ is a
complete bipartite graph with vertex classes of sizes $k$ and $n-k$.
\end{theorem}

The weaker edge condition in Theorem~\ref{thm:main} (compared to Conjecture~\ref{conj:main}) is best possible. That is, there are $n$-vertex graphs with $(k-1)(n-k+1)$ edges for which \eqref{eq:suminmainthm} does not hold. Namely, $K_{k-1,n-k+1}$ 
 has $(k-1)(n-k+1)$ edges and no $2k$-cycle, so that $\mathcal{C}(K_{k-1,n-k+1})$ is a strict subset of $\mathcal{C}(K_{k,n-k})=\{4,6,\dots,2k\}$ and, hence, $s(K_{k-1,n-k+1})<s(K_{k,n-k})$. 

The motivating idea behind Conjecture~\ref{conj:main} is that every substantial departure from a complete bipartite graph must be paid for by an increased harmonic weight of the cycle lengths, and the proof of Theorem~\ref{thm:main} confirms this intuition precisely. The proof shows that, for any $\eps>0$ there is a $\delta>0$ such that, for large $k$, any $n$-vertex graph $G$ with more than $(k-1)(n-k+1)$ edges which satisfies $s(G)<s(K_{k,n-k})+\delta$ must contain an almost-complete bipartite graph which has itself the cycle lengths $4,6,\dots,(1-\eps)2k$. In order to confirm \eqref{eq:erdos-exact-conj}, we then show that if $G$ is not close to $K_{k,n-k}$ there is either a scale at which $G$ has an abundance of odd cycles, or a new scale  (beyond $(1-\eps)2k$) at which $G$ has an abundance of cycles.

In this form, the proof falls under the conceptual framework of the `stability method' which has seen great success in the study of dense graphs~\cite{SimonovitsStability,FurediStability}. That is, many exact extremal results in graph theory have been shown by arguing that any graph close to meeting some extremal condition must be structurally close to a precise extremal example for the problem. This proximity allows us to obtain additional structural properties, which makes it easier to show the exact extremal condition. 

For sparse graphs, this framework arises much more rarely and inexactly. For example, graphs falling under Theorem~\ref{thm:main} with $s(G)\leq (1+\eps)s(K_{k,n-k})$ (for any fixed $\eps>0$) vary broadly in their composition and the structural properties they share are much weaker. If we consider such graphs with some basic properties (which we will show hold in a minimal counterexample to Theorem~\ref{thm:main}), then what they do share is suggested by the above discussions: they contain almost-complete bipartite graphs (which form the bedrock of the structures that we will call \emph{clusters}). This can be roughly recovered from the work of Liu and Montgomery~\cite{liu2023solution}, and with only a little modification (confirmed in Appendix~\ref{appendix}) we can use this work as a `black box'. However, much more work needs to be done to understand the structural properties of near-extremal examples and use these to find more cycle lengths. We discuss our strategy and the new ideas required to implement it in Section~\ref{sec:outline}.

\medskip\noindent
\textbf{Notation.} For a graph $G$, we write $|G|$ for the number of vertices in $G$, and $e(G)$ for the number of edges. For a vertex $v\in V(G)$, we write $N_G(v)$ ($d_G(v)$) for the set (number) of neighbours of $v$ in $G$, and $N_G(v, U)$ ($d_G(v,U)$) for the set (number) of neighbours of $v$ in some set $U\subseteq V(G)$. The average degree of a graph $G$ is denoted by $d(G)$. The neighbourhood of a set $U\subseteq V(G)$ is defined as $N(U)=\{v\in V(G)\backslash U:N(v)\cap U\neq \varnothing\}$. Finally, when stating our results, we use the standard parameter hierarchy notation, where $\eps\ll \delta\ll1$ should be read as: `the statement holds for every sufficiently small parameter $\delta$, and every $\eps$ which is sufficiently small as a function of $\delta$'. We typically do not specify how small we need the parameters to be, but the requirements could in principle be computed explicitly. We also omit integer parts where they are not essential to the arguments. 

\medskip\noindent
\textbf{Paper organisation.} In Section~\ref{sec:outline}, we outline our proof, in which we will take a minimal counterexample to Theorem~\ref{thm:main} and show it contains structures we call \emph{clusters}. In the same section, we recall the tools we will use in our proof. In Section~\ref{sec:properties_of_minimal_counterexample} we discuss the properties of a minimal counterexample, such as its minimum degree, number of vertices and connectivity. Properties of clusters will be shown in Section~\ref{sec:clusters}. Then, in Section~\ref{sec:main_prop_proof}, we will prove the important statement that most vertices of $G$ are contained in any maximal edge-disjoint collection of clusters. Finally, we prove Theorem~\ref{thm:main} in Section~\ref{sec:main_thm_proof}.


\section{Proof outline and preliminaries}\label{sec:outline}\label{sec:preliminaries}

\subsection{Sketch of the proof}

In order to prove Theorem~\ref{thm:main}, we assume to the contrary that it is false and study the properties of a minimal counterexample. That is, having chosen our large $k_0$ as in the statement of Theorem~\ref{thm:main}, we take the smallest possible $n$ for which there is some $k\geq k_0$ with $n\geq 2k$ for which there is some $n$-vertex graph $G$ with more than $(k-1)(n-k+1)$ edges for which either \textup{\textbf{a)}} $\sum_{\ell\in \cC(G)}\frac{1}{\ell}<\sum_{\ell=2}^k\frac{1}{2\ell}$ or \textup{\textbf{b)}} $\sum_{\ell\in \cC(G)}\frac{1}{\ell}= \sum_{\ell=2}^k\frac{1}{2\ell}$, $e(G)\geq k(n-k)$ and $G$ is not a copy of $K_{k,n-k}$. 
As we show in Section~\ref{sec:conmindeg}, $G$ must then have minimum degree at least $k$, be 2-connected and satisfy a further connectivity condition (see Lemma~\ref{lemma:minimum degree}). More importantly, however, is that we then show that $G$ is not too small, as, in fact, we will have $n>k^{1.02}$ (see Lemma~\ref{lem:few_vertices}). Indeed, if $n$ is smaller, then we can combine results of Liu and Ma (Theorem~\ref{thm:AP}) and Gould, Haxell and Scott (Theorem~\ref{thm:GHS}) to find enough cycles with length in $[4,2k]$ (with some different structure according to different cases) to reach a contradiction.

Assuming, then, that $n>k^{1.02}$, as $e(G)> (k-1)(n-k+1)$, we get that the average degree $d(G)$ of $G$ is at least $(1-o(1))2k$ (in contrast to when $n$ is small when it may be close to $k$). Therefore, we can find a subgraph $\Gamma\subset G$ with similarly good minimum degree, average degree at least only a little below $2k$ and which is a sublinear expander. Here, a \emph{sublinear expander} is a graph satisfying some weak expansion condition (see Section~\ref{sec:sublin}). The sublinear expander $\Gamma$ may be small, and in passing to it we may have lost the precise minimum degree condition needed for our argument for when $G$ is small, but the added sublinear expansion condition will allow us to show that $\Gamma$ will contain an \emph{almost-complete pair} $(A,B)$ (or enough different cycle lengths to yield a contradiction).

\begin{definition}\label{def:almost}
Let $\eps>0$. A pair $(A,B)$ of disjoint sets $A,B\subseteq V(G)$ is \textit{$(\eps,k)$-almost-complete (in $G$)} if $(1-\eps)k\leq |A|< k$, $|B|\geq 2k$ and $\delta(G[A, B])\geq (1-2\eps) k$.
\end{definition}

If $\Gamma$ still has many vertices ($|\Gamma|>k^{1.01})$ then work of Liu and Montgomery (in the slightly modified form of Theorem~\ref{mainthm-new}) applies; with only a little additional work (see Section~\ref{sec:largeexpandercase}), in this case we also get that $\Gamma$ must contain an almost-complete pair.

Overall, these elements combine to show that our minimal counterexample $G$ must contain an almost-complete pair $(A,B)$. Because of the large density of edges between $A$ and $B$, it is easy to find many cycles of different lengths between $A$ and $B$ (see Lemma~\ref{lemma:path lengths in almost complete pairs}), and indeed all even cycle lengths between $4$ and $2|A|\geq (1-\eps)2k$. In comparison to $K_{k,n-k}$, then, we will need to find some other cycle lengths whose harmonic sum is at least the harmonic sum of the even numbers from $2|A|+2$ to $2k$. Recalling the discussion from the introduction, it would suffice to either $\textbf{a)}$ find a new scale at which there is a positive density of even cycle lengths, \textbf{b)} find cycles with each even integer length from  $2|A|+2$ to $2k+2$, or \textbf{c)} find a scale at which there is a positive density of odd cycle lengths (where this scale may be small as we have not yet claimed any odd cycle lengths).

If one were able to find a path of length $\ell$, where $\eps k\ll \ell\leq O(k)$, with both endpoints in $A$ and with interior disjoint from $A\cup B$, then the argument for constructing cycles between almost-complete pairs is robust enough to find similar cycles with this path inserted into it, and thus find the cycle lengths as required by \textbf{a)} or  \textbf{b)} if these lengths are even or, if they are odd, enough for \textbf{c)}. Similarly, if  $G$ contained two edge-disjoint almost-complete pairs, $(A_1, B_1)$ and $(A_2, B_2)$ intersecting in at least four vertices and (say) at most $k/2$ vertices, then we can similarly find the lengths for \textbf{a)} or \textbf{b)}.

This now gives us a feeling for what $G$ might look like: $G$ may consist of a collection of almost-complete pairs which are essentially vertex disjoint with edges added between them (to ensure minimum degree at least $k$) so that there is no path $\eps k\ll \ell\leq O(k)$ with both endpoints in the smaller side of any one of the almost-complete pairs. Where this occurs, we might hope to find a cycle which passes between many different almost-complete pairs often enough, that by varying its length using the almost-complete pairs we can find many different cycle lengths at a new scale (depending on whether they are odd or even cycles then we might satisfy \textbf{a)} or \textbf{c)} above).

This brief sketch is essentially what we do, but in order to make it work we extend the almost-complete pairs to structures we call \emph{clusters} (see Definition~\ref{def:cluster} for a precise definition). At the core of a cluster will be an almost complete pair $(A,B)$, and the cluster can then be formed by iteratively adding vertices with degree at least 2 into the cluster vertices so far.

We will then take an edge-disjoint collection of clusters $\cH=\{H_1, \dots, H_t\}$ which maximises the number of covered edges of $G$. Moreover, subject to this constraint, we will choose $\cH$ which minimises the number of clusters $t$. 
It turns out that this collection of clusters will cover almost all of the vertices of $G$. In fact, it will cover all but fewer than $O(|\cH|)$ vertices (see Lemma~\ref{lem:few uncovered vertices}). In particular, this will show that $\cH$ contains at least one cluster. In Section~\ref{sec:main_thm_proof}, we will complete the proof of Theorem~\ref{thm:main} in two cases, when there is only one cluster (encapsulated in Lemma~\ref{lemma:t=1} and proved in Section~\ref{sec:one_cluster}) and when there are at least two clusters (encapsulated in Lemma~\ref{lem:not_many_clusters} and proved in Section~\ref{sec:case:t>1}). Before we prove each of these cases, we sketch our approach in more detail.

In the rest of this section, we define clusters precisely and introduce our notation for them, before recalling some results from the literature on cycles and then proving a simple property of cycles in almost-complete pairs. In Section~\ref{sec:properties_of_minimal_counterexample}, we prove some properties of a minimal counterexample to Theorem~\ref{thm:main}. In Section~\ref{sec:clusters}, we prove some properties of clusters in a minimal counterexample. In Section~\ref{sec:main_prop_proof}, we prove our key lemma, Lemma~\ref{lem:few uncovered vertices}, which shows that a maximal collection of clusters covers most of the vertices in the graph. Finally, in Section~\ref{sec:main_thm_proof}, we complete the proof of Theorem~\ref{thm:main} as outlined above.

\subsection{Definition and notation of clusters}
We use the following definition for \emph{clusters}, a key part of our proof as outlined in the above proof sketch.
\begin{definition}\label{def:cluster}
 A subgraph $H\subseteq G$ is a \textit{$(\eps,k)$-cluster} if there is some $\ell\geq 0$ for which there is a strictly increasing chain $X^{(0)}\subsetneq X^{(1)}\subsetneq \cdots \subsetneq X^{(\ell)}=V(H)$ satisfying the following two properties.
\begin{enumerate}[label = \textbf{\roman{enumi})}]
    \item  $X^{(0)}=A\cup B$, where $(A, B)$ is $(\eps,k)$-almost-complete in $H$.
    \item For every $i\in [\ell]$ and $v\in X^{(i)}\backslash X^{(i-1)}$, we have $|N_H(v)\cap X^{(i-1)}|\geq 2$, and $|N_H(v)\cap X^{(i-2)}|\leq 1$.
\end{enumerate}
\end{definition}

\noindent
\begin{minipage}[t]{0.58\textwidth}
\vspace{0pt}
\setlength{\parindent}{17pt}
\indent For this definition, we take $X^{(-1)}=\emptyset$ where relevant. Note that adding edges to a cluster maintains it being a cluster. Additionally, adding a vertex which has two edges into a cluster maintains it being a cluster.  Indeed, retain the same almost-complete pair $X^{(0)}$ and define the subsequent sets iteratively, at each step adding all remaining vertices having at least two neighbours in the current set. Every original vertex is eventually added, and so is the new vertex. Moreover, a vertex added at a given step had at most one neighbour two steps earlier, by the choice of the step at which it first enters.
\end{minipage}\hfill
\begin{minipage}[t]{0.37\textwidth}
\vspace{0pt}
\centering
\begin{tikzpicture}[
    x=1cm,
    y=1cm,
    layer/.style={
        draw=black,
        rounded corners=7pt,
        line width=0.6pt
    },
    core/.style={
        draw=black,
        fill=white,
        line width=0.6pt
    },
    point/.style={
        circle,
        fill=black,
        inner sep=1.7pt
    },
    edge/.style={
        draw=black,
        line width=0.65pt
    }
]
    \draw[layer,fill=black!3] (-2.8,-1.7) rectangle (1.8,1.8);

    \node[anchor=south west] at (-2.8,1.2) {$X^{(2)}$};
     \draw[layer,fill=white] (-1.9,-1.55) rectangle (1.4,1.3);
    \node[anchor=south west] at ($(-1.9,0.65)$) {$X^{(1)}$};


    \draw [black!50,fill=black!50] ($(-0.5,-0.3)+(0,0.6)$) -- ($(0.5,-0.3)+(0,1)$) -- ($(0.5,-0.3)+(0,-1)$) -- ($(-0.5,-0.3)-(0,0.6)$) -- cycle;


    \node[anchor=south west] at ($(-0.5,-0.3)+(0,0.6)+(-0.85,-0.1)$) {$X^{(0)}$};
    
    \draw[core] (-0.5,-0.3) ellipse [x radius=0.3, y radius=0.6];
    \draw[core] (0.5,-0.3) ellipse [x radius=0.3, y radius=1];
    \node at (-0.5,-0.6) {$A$};
    \node at (0.5,-0.6) {$B$};

    \node[point] (a1) at (-0.5, 0.15) {};
    \node[point] (b1) at (0.4, 0.3) {};
    \node[point] (b2) at (0.6, -0.1) {};

    \node[point] (u1) at (-0.25,1) {};
    \node[point] (u2) at (0.4,1.5) {};

    \draw[edge] (u1) -- (a1);
    \draw[edge] (u1) -- (b1);
    \draw[edge] (u2) -- (u1);
    \draw[edge] (u2) -- (b2);
\end{tikzpicture}

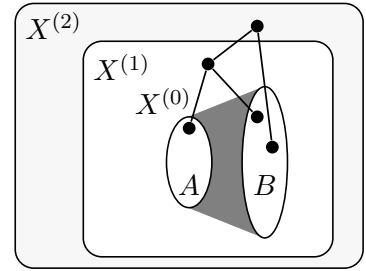
\captionof{figure}{A cluster of depth two, where $X^{(0)}=A\cup B$.}
\label{fig:cluster}
\end{minipage}

When it is not completely clear which cluster we consider, we will write $X^{(i)}(H)$, $A(H)$, $B(H)$ to denote the corresponding vertex subsets of the cluster $H$.

\subsection{Results on cycle lengths}
We will use the following result of Liu and Ma~\cite{LM18} on arithmetic progressions of cycle lengths in bipartite graphs.

\begin{theorem}\label{thm:AP}
Let $k\in \N$ and let $G$ be a bipartite graph with $\delta(G)\geq k$. Then, $\cC(G)$ contains an arithmetic progression of length $k-1$ with common difference two.
\end{theorem}

We will also use the following result of Gould, Haxell and Scott~\cite[Theorem~1]{GHS02}.

\begin{theorem}\label{thm:GHS}
There exists $C>0$ such that the following holds for each $c\in (0,1)$ with $K=C\cdot c^{-5}$. Let $G$ be a graph with $n\geq 45K/c^4$ vertices and minimum degree at least $cn$. Then $G$ contains a cycle of length $\ell$ for
every even integer $\ell\in[4, \ec(G)-K]$ and every odd integer $\ell\in [K, \oc(G)-K]$, where $\ec(G)$ and $\oc(G)$ are the lengths of the longest even and odd cycle in $G$, respectively.
\end{theorem}

We will use the following result of Brandt~\cite{Brandt1997ShortCycles} on cycles in non-bipartite dense graphs.

\begin{theorem}\label{thm:brandt}
Any non-bipartite $n$-vertex graph $G$ with more than $\frac{1}{4}(n-1)^2+1$ edges contains cycles of all possible lengths from 3 up to the length of a longest cycle in $G$.
\end{theorem}

Finally, we will use the following classical result of Erd\H{o}s and Gallai~\cite{ErdosGallai1959MaximalPathsCircuits}.

\begin{theorem}\label{thm:EG} Any graph $G$ with $d(G)\geq 2$ has a cycle with length at least $d(G)$.
\end{theorem}

\subsection{Paths in almost-complete pairs}
Here we confirm the robust existence of paths in almost-complete pairs (see Definition~\ref{def:almost}).
\begin{lemma}\label{lemma:path lengths in almost complete pairs} Let $1/k\ll \eps\ll 1$. Let $G$ be a graph containing an $(\eps,k)$-almost-complete pair $(A,B)$.
Let $F\subset A\cup B$ satisfy $|F|\leq k/10$.
Then, for any distinct vertices $u, v\in A\backslash F$ and any integer $2\leq \ell\leq |A\backslash F|-1$, $G[A,B]-F$ contains a $u,v$-path of length $2\ell$.
\end{lemma}
\begin{proof}
Set $a:=|A\setminus F|\ge (9/10-\eps)k$. 
Create an auxiliary graph $L$ with vertex set $A\setminus F$ and, for each distinct $x,y\in V(L)$, an edge $xy\in E(L)$ when $|N_G(x,B)\cap N_G(y,B)|\geq k/2$. Then, using the definition of an $(\eps,k)$-almost-complete pair, for each $x\in V(L)$ there are at most $2\eps k \cdot |N_G(x,B)|$ non-edges between $N_G(x,B)$ and $A\setminus F$ in $G$. Thus, at most $k/10$ vertices in $A\setminus F$ can have more than $20\eps |N_G(x,B)|$ non-neighbours in $N_G(x,B)$. As $|N_G(x,B)|\geq (1-2\eps)k$, we therefore have $d_L(x)\geq |A\setminus (F\cup \{x\})|-k/10\ge a-1-k/10\geq 7a/8$. Thus, $\delta(L)\geq 7|V(L)|/8$.

Now, let $u, v\in A\backslash F$ be distinct and let $\ell\in\N$ satisfy $2\leq \ell\leq |A\backslash F|-1$. We now find a path, $P$ say, in $L$ with length $\ell$. If $\ell\leq 3a/4$, then this can be done greedily from the minimum degree condition, building a path of length $\ell-2$ from $u$ and then finding a common neighbour to attach it to $v$. If $\ell\geq 3a/4$, then this can be done by removing $|A\setminus F|-\ell-1\leq a/4$ vertices of $V(L)\setminus \{u,v\}$ from $L$ and noting that the resulting graph, $L'$ say, has minimum degree at least $(|L'|+1)/2$ and hence has a Hamilton path from $u$ to $v$ by a standard variant of Dirac's theorem (see~\cite{Dirac1952AbstractGraphs,Ore1963HamiltonConnected}).

Then, create a bipartite auxiliary graph $L^*$ with vertex classes $E(P)$ and $B\setminus F$ where, for each $xy\in E(P)$ and $b\in B\setminus F$ we put an edge $(xy)b$ in $L^*$ if $b\in N_G(x)\cap N_G(y)$. As $P\subset L$, each $xy\in E(P)$ satisfies $d_{L^*}(xy)\ge k/2-|B\cap F|\geq 2k/5$. Furthermore, for each $b\in B\setminus F$, we have that $eb\in E(L^*)$ for all but at most $4\eps k$ edges $e\in E(P)$ as $b$ has at most $2\eps k$ non-neighbours in $A\setminus F$.
We claim that 
  Hall's matching condition holds from $E(P)$ into $B\setminus F$ in $L^*$. Indeed, consider some $X\subseteq E(P)$. If $|X|\le 2k/5$, then picking any $xy\in X$ we have $|N(X)|\ge d_{L^*}(xy)\ge 2k/5$. If $|X|>2k/5$, then for every   $b\in B\setminus F$, the non-neighbourhood of $b$ is too small to contain $X$ --- and so $|N_{L^*}(X)|=|B\setminus F|\ge 2k-k/10\ge  |X|$.
So by Hall's Theorem,  we can find distinct vertices $b_e\in B\setminus F$, $e\in E(P)$, such that $eb_e\in E(L^*)$ for each $e\in E(P)$. Replacing each edge $xy\in E(P)$ on the path $P$ by the edges $xb_{xy}$ and $b_{xy}y$ (which are in $G$ as $(xy)b_{xy}\in E(L^*)$), we get a $u,v$-path of length $2\ell$ in $G$, as required.
\end{proof}
For convenience, we will note the following corollary on the cycle lengths in almost complete pairs.

\begin{corollary}\label{cor:cycles in almost complete pairs} Let $1/k\ll \eps\ll 1$. Let $G$ be a graph containing an $(\eps,k)$-almost-complete pair $(A,B)$. Then, $G[A,B]$  contains a cycle of any even length between $4$ and $2|A|$.
\end{corollary}
\begin{proof}
Firstly, note that any two vertices of $B$ have at least $(1-4\eps)k\geq 2$ common neighbours in $A$, and therefore $G$ contains a cycle of length $4$ as $|B|\geq 2k\geq 2$. Then, let $\ell$ satisfy $2\leq \ell\leq |A|-1$. Pick $b\in B$ and distinct vertices $u, v\in N(b)\cap A$. By Lemma~\ref{lemma:path lengths in almost complete pairs} there exists a $u,v$-path $P_{uv}$ in $G[A,B]-b$ with length $2\ell$. Using the edges $ub$ and $bv$ to close $P_{uv}$ into a cycle now gives a cycle of length $2\ell+2$. Thus, $G[A,B]$  contains a cycle of any even length between $4$ and $2|A|$, as required.
\end{proof}



\section{Properties of a minimal counterexample}\label{sec:properties_of_minimal_counterexample}

In this section, we prove some useful properties of a minimum counterexample to Theorem~\ref{thm:main}. In Section~\ref{sec:conmindeg} we study its connectivity and its minimum degree, while in Section~\ref{sec:counterxamplenotsmall} we show that it is quite large, that is, $n>k^{1.02}$. 

\subsection{Connectedness and minimum degree of a minimum counterexample}\label{sec:conmindeg}

The properties we show here for a minimum counterexample (i.e., those in Lemma~\ref{lemma:minimum degree}) will follow rather straight-forwardly as long as we can show that the number of vertices in the counterexample, $n$, is larger than $2k$. If $n=2k$, then, as $e(G)\geq k^2$ with equality only when $G$ is a copy of $K_{k,k}$, we have that $G$ is not bipartite. This allows us to combine a result of Brandt (Theorem~\ref{thm:brandt}) with the classical Erd\H{o}s-Gallai theorem (Theorem~\ref{thm:EG}) to show that $G$ contains cycles of all lengths between 3 and $k$, and thus harmonic sum in excess of $\sum_{\ell=2}^k\frac{1}{2\ell}$, as shown below, a contradiction.

\begin{lemma}\label{lemma:minimum degree} Let $1/k_0\ll 1$.
Let $n$ be the least integer for which there is an integer $k_0\leq k\leq n/2$ and an $n$-vertex graph $G$ with more than $(k-1)(n-k+1)$ edges for which either \textup{\textbf{a)}} $\sum_{\ell\in \cC(G)}\frac{1}{\ell}<\sum_{\ell=2}^k\frac{1}{2\ell}$ or \textup{\textbf{b)}} $\sum_{\ell\in \cC(G)}\frac{1}{\ell}= \sum_{\ell=2}^k\frac{1}{2\ell}$, $e(G)\geq k(n-k)$ and $G$ is not a copy of $K_{k,n-k}$.

Then, $n\geq 2k+1$ and any such graph $G$ satisfies the following. \textup{\textbf{i)}} $\delta(G)\geq k$. \textup{\textbf{ii)}} $G$ is 2-connected. \textup{\textbf{iii)}} There is no partition $V(G)=A\cup S\cup B$ with $|S|\leq k-1$ and $|A\cup S|, |B\cup S|\geq 2k+2$ in which $S$ separates $A$ and $B$.
\end{lemma}
\begin{proof} Let $G$ be a graph demonstrating the minimality of $n$. We will first show that $n\geq 2k+1$. Since by assumption we have $n\ge 2k$, suppose to the contrary that $n=2k$. Then, we have $e(G)\geq (k-1)(k+1)+1=k^2$. As neither \textbf{a)} nor \textbf{b)} holds when $G$ is a copy of $K_{k,k}$ (the only $2k$-vertex bipartite graph with $k^2$ edges), $G$ is thus not bipartite. Therefore, by Theorem~\ref{thm:brandt}, as $e(G)\geq k^2>\frac{1}{4}(2k-1)^2+1$, $\mathcal{C}(G)$ contains each integer $\ell$ with $3\leq \ell\leq \max \mathcal{C}(G)$. As $d(G)\geq 2e(G)/(2k)\geq k$, by Theorem~\ref{thm:EG}, we have that $\max (\mathcal{C}(G))\geq k$, and hence $\mathcal{C}(G)$ contains each integer $\ell$ with $3\leq \ell \leq k$. Thus,
\begin{equation}\label{eq:sumtocontradict}
\sum_{\ell\in \cC(G)}\frac{1}{\ell}-\sum_{\ell=2}^k\frac{1}{2\ell}\geq \sum_{\ell=3}^{k}\frac1{\ell}-\sum_{\ell=2}^k\frac{1}{2\ell}=\Big(\frac{1}{3}+\frac{1}{4}-\frac{1}{4}-\frac{1}{6}-\frac{1}{8}\Big)+\sum_{\ell=5}^k\Big(\frac{1}{\ell}-\frac{1}{2\ell}\Big)>0.
\end{equation}
This contradicts $\sum_{\ell\in \cC(G)}\frac{1}{\ell}\leq \sum_{\ell=2}^k\frac{1}{2\ell}$, and thus $n\geq 2k+1$.

Using this, we now show \textbf{i)} that $\delta(G)\geq k$. Indeed, if some vertex $v\in V(G)$ has degree at most $k-1$, then $G':=G-v$ has $n-1$ vertices and $e(G')\geq e(G)-(k-1)>(k-1)(n-k)$. As $n-1\geq 2k$, by the minimality of $n$, we must then have that neither \textbf{a)} nor \textbf{b)} holds with $(G',n-1)$ in place of $(G,n)$. As $\sum_{\ell\in \cC(G')}\frac{1}{\ell}\leq \sum_{\ell\in \cC(G)}\frac{1}{\ell}$, this is only possible if \textbf{b)} holds for $(G,n)$ but \textbf{b)} does not hold with $(G',n-1)$ in place of $(G,n)$. Thus, we have $e(G)\geq k(n-k)$, $\sum_{\ell\in \cC(G)}\frac{1}{\ell}\leq \sum_{\ell=2}^k\frac{1}{2\ell}$ and that $G$ is not a copy of $K_{k,n-k}$. As $e(G')\geq e(G)-(k-1)>k(n-1-k)$, and $\sum_{\ell\in \cC(G')}\frac{1}{\ell}\leq \sum_{\ell\in \cC(G)}\frac{1}{\ell}\leq \sum_{\ell=2}^k\frac{1}{2\ell}$, we must have that $G'$ is a copy of $K_{k,n-1-k}$, which contradicts $e(G')>k(n-1-k)$. Thus, we have $\delta(G)\geq k$.

Next we show \textbf{ii)} that $G$ is $2$-connected. Suppose to the contrary that there is some minimal $U\subset V(G)$ with $|U|\leq 1$ for which $G-U$ is disconnected. Let $V(G-U)=A_1\cup A_2$ be a partition into non-empty sets so that there are no edges between $A_1$ and $A_2$ in $G$. For each $i\in [2]$, let $G_i=G[A_i\cup U]$. Then, $|G_1|+|G_2|-1\leq |G|$ and $e(G)=e(G_1)+e(G_2)$.

As $\delta(G)\geq k$, each vertex in $G_1-U$ has degree at least $k$ in $G_1$. Thus, $|G_1|\geq k+1$, and if $U=\emptyset$, then $d(G_1)\geq \delta(G_1)\geq k$. If $U\neq \emptyset$, then there is at least one edge from $U$ to $V(G_1)\setminus U$, whereupon $2e(G_1)\geq 1+(|G_1|-1)k$, so that $d(G_1)>\frac{(|G_1|-1)k}{|G_1|}\geq k-1$. Thus, in either case, we have $d(G_1)>k-1$. Similarly, we have $d(G_2)>k-1$.

We will show that we can assume (by relabelling $G_1$ and $G_2$ if necessary) that $|G_1|<2k$ and $e(G_1)>\frac{1}{4}(|G_1|-1)^2+1$.
First, suppose that \textbf{b)} holds in the statement of the lemma for $G$. Then, as
\[
k(|G_1|-(k+1)/2)+k(|G_2|-(k+1)/2)\leq k(n-k)\leq e(G)
\]
without loss of generality, we can assume that $e(G_1)\geq k(|G_1|-(k+1)/2)$. As $G_1$ is therefore not a copy of $K_{k,|G_1|-k}$, to avoid a contradiction we must have that $|G_1|<2k$. Thus, as $|G_1|\geq k+1$ and $|G_1|<2k$, we have
\[
e(G_1)\geq k\left(|G_1|-\frac{k+1}{2}\right)\geq \left(\frac{|G_1|-1}{2}\right)\left(\frac{|G_1|+1}{2}\right)> \frac14(|G_1|-1)^2+1.
\]
Thus, if  \textbf{b)} holds in the statement of the lemma for $G$, $|G_1|<2k$ and $e(G_1)>\frac{1}{4}(|G_1|-1)^2+1$.

Suppose, then, that \textbf{a)} holds in the statement of the lemma for $G$.
Note that, if $e(G_i)< (k-1)(|G_i|-k/2)$ for each $i\in [2]$, then
\[
e(G)< (k-1)(|G_1|+|G_2|-k)\leq (k-1)(n-k+1),
\]
a contradiction. Therefore, without loss of generality, we can assume that $e(G_1)\geq (k-1)(|G_1|-k/2)$.
As $\sum_{\ell\in \cC(G_1)}\frac{1}{\ell}\leq \sum_{\ell\in \cC(G)}\frac{1}{\ell}<\sum_{\ell=2}^k\frac{1}{2\ell}$ and $e(G_1)\geq (k-1)(|G_1|-k/2)>(k-1)(|G_1|-k+1)$ (using $1/k\ll 1$), by the minimality of $n$ we must have that $|G_1|<2k$.

Let $n_1=|G_1|$, so that $n_1\leq 2k-1$ and $e(G_1)\geq (k-1)(n_1-k/2)$. Now, as $n_1<2k$ (and hence $\frac{1}2(n_1-1)\leq k-1$), we have
\begin{equation}
e(G_1)\geq (k-1)(n_1-k/2)\geq \frac{1}{4}(n_1-1)(2n_1-k)\geq \frac{1}{4}(n_1-1)\cdot n_1>\frac{1}{4}(n_1-1)^2+1,
\end{equation}
where we have used that $n_1\geq k$ and $1/k\ll 1$.

Thus, whichever of \textbf{a)} or \textbf{b)} holds for $G$, we can assume that $|G_1|<2k$ and $e(G_1)>\frac{1}{4}(|G_1|-1)^2+1$. Furthermore, then, as each vertex in $G_1-U$ has degree at least $k$ in $G_1$, and $|G_1|<2k$, any two vertices in $G_1-U$ have a common neighbour in $G_1$. Therefore, as no such pair can be opposite parts of a bipartition of $G_1$ (and $\delta(G_1-U)>0$), $G_1$ must be non-bipartite.

As $d(G_1)>k-1$, by Theorem~\ref{thm:EG}, $G_1$ has a cycle with length at least $k$. Thus, by Theorem~\ref{thm:brandt}, $\mathcal{C}(G_1)$ contains each integer $\ell$ with $3\leq \ell \leq k$. Similarly to \eqref{eq:sumtocontradict},
this implies that $\sum_{\ell\in \cC(G)}{1}/{\ell}\geq \sum_{\ell\in \cC(G_1)}{1}/{\ell}>\sum_{\ell=2}^k{1}/({2\ell})$, contradicting $\sum_{\ell\in \cC(G)}\frac{1}{\ell}\leq \sum_{\ell=2}^k\frac{1}{2\ell}$.
Therefore, we must have that $G$ is 2-connected.

Finally we show that  \textbf{iii)} holds. Suppose, for contradiction, that there does exist such a partition $V(G)=A\cup S\cup B$, and denote the cardinalities of these sets by $a, s$ and $b$ respectively. Then, we have $a+b+s=n$ and $2k+2\leq |A\cup S|, |B\cup S|< n$.

Suppose that the condition \textup{\textbf{a)}} holds. If $e(G[A\cup S])\leq (k-1)(a+s-k+1)$ and $e(G[B\cup S])\leq (k-1)(b+s-k+1)$, then
 \begin{align*}
    e(G)\leq e(G[A\cup S])\!+\!e(G[B\cup S])&\leq (k\!-\!1)(a\!+\!s\!-\!k\!+\!1)+(k\!-\!1)(b\!+\!s\!-\!k\!+\!1)\\
    &=(k\!-\!1)(a\!+\!b\!+\!s\!-\!k\!+\!1)+(k\!-\!1)(s\!-\!k\!+\!1)\le (k\!-\!1)(n\!-\!k\!+\!1),
\end{align*} contradicting the assumption on $e(G)$.
Thus, either $G[A\cup S]$ has more than $(k-1)(a+s-k+1)$ edges or $G[B\cup S]$ has more than $(k-1)(b+s-k+1)$ edges. Suppose, by relabelling if necessary, that  $e(G[A\cup S])> (k-1)(a+s-k+1)$. Then, by minimality of $n$ (using that $2k+2\leq a+s<n$), we have $\sum_{\ell\in \cC(G[A\cup S])}\frac{1}{\ell}\geq \sum_{\ell=2}^{k}\frac{1}{2\ell}$, so that $\sum_{\ell\in \cC(G)}\frac{1}{\ell}\geq \sum_{\ell\in \cC(G[A\cup S])}\frac{1}{\ell}\geq \sum_{\ell=2}^{k}\frac{1}{2\ell}$, which contradicts \textup{\textbf{a)}}.

Suppose, then, that the condition \textup{\textbf{b)}} holds. If $e(G[A\cup S])\leq k(a+s-k)$ and $e(G[B\cup S])\leq k(b+s-k)$, then
 \[e(G)\leq e(G[A\cup S])+e(G[B\cup S])\leq k(a+s-k)+k(b+s-k)=k(a+b+s-k)+k(s-k)<k(n-k),\]
 which contradicts the assumption \textup{\textbf{b)}}. Therefore, by relabelling if necessary, we can assume that $e(G[A\cup S])>k(a+s-k)$.
By the minimality of $n$, as $2(k+1)\leq a+s<n$, we know that there is no graph $G'$ with $e(G')>k(a+s-k)$ and $\sum_{\ell\in \cC(G')}\frac{1}{\ell}< \sum_{\ell=2}^{k+1}\frac{1}{2\ell}$. Therefore, we must have $\sum_{\ell\in \cC(G[A\cup S])}\frac{1}{\ell}\geq \sum_{\ell=2}^{k+1}\frac{1}{2\ell}>\sum_{\ell=2}^{k}\frac{1}{2\ell}$, and hence $\sum_{\ell\in \cC(G)}\frac{1}{\ell}\geq \sum_{\ell\in \cC(G[A\cup S])}\frac{1}{\ell}> \sum_{\ell=2}^{k}\frac{1}{2\ell}$, contradicting the assumption \textup{\textbf{b)}}. This achieves a contradiction in all cases, and thus \textup{\textbf{iii)}} holds.
\end{proof}


\subsection{A minimal counterexample is quite large}\label{sec:counterxamplenotsmall}
We now confirm that the minimal counterexample is quite large (with more than $k^{1.02}$ vertices). The main tools we use for this are the results of Liu and Ma (Theorem~\ref{thm:AP}), Gould, Haxell and Scott (Theorem~\ref{thm:GHS}) and Erd\H{o}s and Gallai (Theorem~\ref{thm:EG}), which otherwise we can combine to find enough cycles with length in $[4,2k]$ to reach a contradiction.


\begin{lemma}\label{lem:few_vertices}
Let $1/k_0\ll 1$.
Let $n$ be the least integer for which there is an integer $k_0\leq k\leq n/2$ and an $n$-vertex graph $G$ with more than $(k-1)(n-k+1)$ edges for which either \textup{\textbf{a)}} $\sum_{\ell\in \cC(G)}\frac{1}{\ell}<\sum_{\ell=2}^k\frac{1}{2\ell}$ or \textup{\textbf{b)}} $\sum_{\ell\in \cC(G)}\frac{1}{\ell}= \sum_{\ell=2}^k\frac{1}{2\ell}$, $e(G)\geq k(n-k)$ and $G$ is not a copy of $K_{k,n-k}$. Then, $n>k^{1.02}$.
\end{lemma}
\begin{proof} Let $G$ be a graph demonstrating the minimality of $n$. Assume, for contradiction, that $n\leq k^{1.02}$. By Lemma~\ref{lemma:minimum degree}, we have that $\delta(G)\geq k$ and $G$ is 2-connected.
Using Theorem~\ref{thm:GHS}, let $C$ be such that the property in that theorem holds.
Let $c=k^{-0.02}/2$, so that $\delta(G)\geq k\geq cn$.  Then, setting $K=C\cdot c^{-5}$, and using that, as $1/k\ll 1$,  $K\leq k^{0.2}$ and $n\geq k\geq 45K/c^4$, by the property of $C$ applied with $G$, $\cC(G)$ contains all even cycle lengths between $4$ and $\ec(G)-K$ and all odd cycle lengths between $K$ and $\oc(G)-K$.

We now show that $G$ is bipartite. Suppose otherwise, for contradiction. Let $S$ be a longest odd cycle of $G$, and suppose that $|S|\leq k/3$. Then, $\delta(G-S)\geq 2k/3$ so that, by Theorem~\ref{thm:EG}, $G-S$ contains a cycle, $S'$ say, with length at least $2k/3$. As $G$ is 2-connected, by Menger's theorem, $G$ contains two vertex-disjoint paths from $V(S)$ to $V(S')$. Observe that these two paths, in conjunction with $S$ and $S'$, contain an odd cycle with length at least $|S'|/2+3\geq k/3$. Thus, $\oc(G)\geq k/3$.

As $G$ contains a bipartite subgraph with average degree at least $d(G)/2\geq k/2$, by Theorem~\ref{thm:EG} again we have $\ec(G)\geq d(G)/2\geq k/2$. Therefore, we have
\begin{align*}
\sum_{\ell\in \mathcal{C}(G)}\frac{1}{\ell}&= \sum_{\ell\in \mathcal{C}(G):\ell=0\,\mathrm{mod}\,2}\frac{1}{\ell}+\sum_{\ell\in \mathcal{C}(G):\ell=1\,\mathrm{mod}\,2}\frac{1}{\ell}\geq \sum_{\ell=2}^{\lfloor k/4-K \rfloor}\frac{1}{2\ell}+\sum_{\ell=\lceil K\rceil }^{\lfloor k/6-K\rfloor}\frac{1}{2\ell-1
}\\ &\geq
\frac12\sum_{\ell=2}^{\lfloor k/5\rfloor}\frac1\ell
+
\frac12\sum_{\ell=\lceil k^{0.2}\rceil}^{\lfloor k/7\rfloor}
\frac1\ell \geq
\frac12\left(\log\frac{k}{5}-1\right)
+
\frac12\log\frac{k}{14k^{0.2}} >\frac12 \log k\ge \sum_{\ell=2}^k\frac{1}{2\ell},
\end{align*}
a contradiction. Thus, $G$ is bipartite.

Therefore, by Theorem~\ref{thm:AP}, $\cC(G)$ contains an arithmetic progression of length $k-1$ with common difference two. As $\ec(G)\geq k$, and $\cC(G)$ contains every even number between $4$ and $\ec(G)-k^{0.2}$, we therefore have that $\cC(G)$ contains all even cycle lengths between $4$ and $2k$ (either $\ec(G)\geq 2k+k^{0.2}$ in which case ``between $4$ and $\ec(G)-k^{0.2}$'' gives what we want, or else $\ec(G)< 2k+k^{0.2}$ in which case the progression  starts at some   $a\le 2k+k^{0.2}-2(k-2)<k-k^{0.2}\le \ec(G)-k^{0.2}$ and ends at $a+2(k-2)\ge 4+2(k-2)=2k$. In the latter case we have all even numbers from $4$ to $\ec(G)-k^{0.2}$ and from $a$ to $a+2(k-2)$ which together contain all even numbers from $4$ to $2k$). 
Thus, $\sum_{\ell\in \mathcal{C}(G)}\frac{1}{\ell}\geq \sum_{\ell=2}^k\frac{1}{2\ell}$ and, furthermore, there must be no cycle in $G$ with length at least $2k+2$ as otherwise we have $\sum_{\ell\in \mathcal{C}(G)}\frac{1}{\ell}> \sum_{\ell=2}^k\frac{1}{2\ell}$.

Thus, we have $\sum_{\ell\in \mathcal{C}(G)}\frac{1}{\ell}\geq \sum_{\ell=2}^k\frac{1}{2\ell}$ and $G$ has no cycle with length at least $2k+2$.
 As \textbf{a)} in the statement of the lemma does not hold, we have that $e(G)\geq k(n-k)$ and $G$ is not a copy of $K_{k,n-k}$. If $G$ contains a subgraph with minimum degree $k+1$, then, by Theorem~\ref{thm:AP} applied to such a subgraph, $\cC(G)$ contains an arithmetic progression of length $k$ of common difference two, and hence a cycle with length at least $2k+2$, a contradiction. Therefore, we can iteratively remove vertices from $G$ until we reach a subgraph $G'\subset G$ with $2k$ vertices and at least $k(|G'|-k)=k^2$ edges, which therefore is a balanced, complete bipartite graph.

If $G-V(G')$ contains an edge $uv$, then, as $G$ is 2-connected, there is (through $uv$) a path, $P$ say, from $V(G')$ to $V(G')$ with length at least 3. Joining this to a path in $V(G')$ between its endpoints of length $\ge 2k-2$ we get a cycle of length at least $(2k-2)+3=2k+1$. Since this is an even cycle it has length $\ge 2k+2$, a contradiction. 
Letting $V(G')=A\cup B$ be the bipartition of $G'$ (so that $|A|=|B|=k$), as $G$ has no cycle of length $2k+2$, there is no pair of vertices $a,b\in V(G)\setminus V(G')$ with $d_G(a,B)\geq 2$ and $d_G(b,A)\geq 2$. Therefore, as $\delta(G)\geq k$, we must have that $G$ is a copy of $K_{k,n-k}$, a contradiction.
\end{proof}


\section{Properties of clusters in a minimal counterexample}\label{sec:clusters}

In this section, we will establish several properties of clusters, which we gather into Lemma~\ref{lem:cluster properties}. We will use the following definition. 

\begin{definition}\label{def:cluster cycle} A \textit{cluster cycle} is a collection of $s\geq 2$ edge-disjoint clusters $H_{1}, \dots, H_{s}$ and $s$ vertex-disjoint paths $P_1, \dots, P_s$, each of length at most $4$, such that, for each $j\in [s]$, $P_j$ starts in $H_{j}$ and ends in $H_{j+1}$, where the indices are considered modulo $s$ (i.e., $H_{s+1}=H_1$).
\end{definition}

\begin{lemma}\label{lem:cluster properties} Let $1/k_0\ll \eps \ll 1$.
Let $n$ be the least integer for which there is an integer $k_0\leq k\leq n/2$ and an $n$-vertex graph $G$ with more than $(k-1)(n-k+1)$ edges for which either \textup{\textbf{a)}} $\sum_{\ell\in \cC(G)}\frac{1}{\ell}<\sum_{\ell=2}^k\frac{1}{2\ell}$ or \textup{\textbf{b)}} $\sum_{\ell\in \cC(G)}\frac{1}{\ell}= \sum_{\ell=2}^k\frac{1}{2\ell}$, $e(G)\geq k(n-k)$ and $G$ is not a copy of $K_{k,n-k}$.
Let $G$ be such a graph.

Then, the following properties hold.
\begin{enumerate}[label = \textbf{\textup{\alph{enumi})}}]
\item If $H_1$ and $H_2$ are $(\eps,k)$-clusters in $G$, then we have that either $H_1\cup H_2$ is an $(\eps,k)$-cluster,
or $|(A(H_1)\cup B(H_1))\cap (A(H_2)\cup B(H_2))|\leq 3$.\label{prop:cluster:vxdisjoint}
\item For every $(\eps,k)$-cluster $H$ in $G$ and every distinct $u,v\in V(H)$, there are vertex-disjoint paths from $u$ to $A(H)\cup B(H)$ and from $v$ to $A(H)\cup B(H)$ of length at most $3\eps k$ each. In particular, the diameter of $H$ is at most $6\eps k+4$. \label{prop:cluster:depth}
\item $G$ does not contain a cluster cycle of edge-disjoint $(\eps,k)$-clusters in which no union of two clusters is itself a $(\eps,k)$-cluster.\label{prop:cluster:nocycle}
\end{enumerate}
\end{lemma}

We will now prove the statements of Lemma~\ref{lem:cluster properties}\,\ref{prop:cluster:vxdisjoint}--\ref{prop:cluster:nocycle} in Sections~\ref{sec:cluster1}--\ref{sec:cluster3}, respectively.


\subsection{Vertex-intersecting clusters: Proof of Lemma~\ref{lem:cluster properties}\,\ref{prop:cluster:vxdisjoint}}\label{sec:cluster1}

We will now prove Lemma~\ref{lem:cluster properties}\,\ref{prop:cluster:vxdisjoint} by showing that any two clusters whose union is not a cluster cannot overlap significantly. We do this by contradiction, showing in several different cases that vertex overlaps of edge-disjoint clusters lead to enough additional cycles to violate the conditions of a minimal counterexample.

\begin{proof}[Proof of Lemma~\ref{lem:cluster properties}\,\ref{prop:cluster:vxdisjoint}.]
Suppose $H_1$ and $H_2$ are $(\eps,k)$-clusters in $G$ such that $H_1\cup H_2$ is not an $(\eps,k)$-cluster. Let $A_i=A(H_i)$ and $B_i=B(H_i)$ for each $i\in [2]$.  We will show that $|V(H_1)\cap B_2|\leq 1$, $|V(H_2)\cap B_1|\leq 1$ and $|A_1\cap A_2|\leq 1$, and hence the required bound on $|(A_1\cup B_1)\cap (A_2\cup B_2)|$ will hold.
Let $H$ be an $(\eps,k)$-cluster of $G$ with $E(H_1)\subset E(H)\subset E(H_1\cup H_2)$ which maximises $e(H)$.  Note that, by this maximality, there is no vertex $v\in V(H_2)\setminus V(H)$ with at least two neighbours in $V(H)$ in $H_2$, for otherwise we could add $v$ and its neighbouring edges into $V(H)$  in $H_2$ to $H$ and get a contradiction.

Now, suppose that $|V(H)\cap B_2|\geq 2$, so that we may take distinct vertices $v_1,v_2\in V(H)\cap B_2$.
By the definition of the $(\eps,k)$-cluster $H_2$, we have $|N_{H_2}(v_i,A_2)|\geq (1-2\eps)k$ for each $i\in [2]$, and $|A_2|<k$. Then, for all but at most $4\eps k$ vertices $a\in A_2$ we have $v_1,v_2\in N_{H_2}(a)$. For each such $a\in A_2$, by the maximality of $H$ we have $a\in V(H)$. Therefore, as each  $b\in B_2$ has degree at least $(1-2\eps)k$ in $A_2$ in $H_2$, each $b\in B_2$ has at least $(1-2\eps)k-(k-4\eps k)\geq 2$ neighbours in $H_2$ in $A_2\cap V(H)$. Thus, by the maximality of $H$, we have $B_2\subset V(H)$. Then, every vertex $a\in A_2$ (i.e., not just those $a$ with $a\in N_{H_2}(v_1)\cap N_{H_2}(v_2)$) has at least 2 neighbours in $H_2$ in $V(H)$, so that, by the maximality again, $A_2\subset V(H)$.

Using the definition of an $(\eps,k)$-cluster in $G$, we can order the vertices in $V(H_2)\setminus (A_2\cup B_2)$ so that each vertex in this sequence has at least two neighbours in $A_2\cup B_2$ or among the previous vertices in the sequence. Therefore, by the maximality, as $A_2\cup B_2\subset V(H)$, there is no earliest vertex in the sequence which is not in $V(H)$, and, hence, $V(H_2)\subset V(H)$. As adding any set of edges within an $(\eps,k)$-cluster to it cannot spoil the $(\eps,k)$-cluster property, we thus have that $H=H_1\cup H_2$, contradicting that $H_1\cup H_2$ is not an $(\eps,k)$-cluster.

Therefore, we have that $|V(H)\cap B_2|\leq 1$, and, thus, $|V(H_1)\cap B_2|\leq 1$. By a symmetric argument, we also must have that $|V(H_2)\cap B_1|\leq 1$.
Suppose, then, $|A_1\cap A_2|\geq 2$, and take distinct vertices $a_1,a_2\in A_1\cap A_2$. If $|A_1\cup A_2|\geq k+1$, then for each $\ell_1,\ell_2\geq k/3$ with $\ell_1+\ell_2\leq k-1$, we can take disjoint sets $A_1^-,A_2^-\subset (A_1\cup A_2)\setminus \{a_1,a_2\}$ with $A_1^-\subset A_1$, $A_2^-\subset A_2$, $|A_1^-|=\ell_1$ and $|A_2^-|=\ell_2$.
Note that, for each $i\in [2]$, $H_i[A_i^-\cup \{a_1,a_2\},B_i\setminus V(H_{3-i})]$ is $(4\eps,\ell_i+3)$-almost-complete using that $|B_i\cap V(H_{3-i})|\leq 1$, and, thus by Lemma~\ref{lemma:path lengths in almost complete pairs}, there is an $a_1,a_2$-path with length $2|A_i^-|+2$. Combining these two paths gives a cycle of length $2(\ell_1+\ell_2+2)$. Thus, $\mathcal{C}(G)$ contains a cycle of every even length from $4\lceil k/3\rceil+4$ to $2(k+1)$. As, by Corollary~\ref{cor:cycles in almost complete pairs}, $H_1$ contains a cycle of every even length between $4$ and $2(1-\eps)k$, we thus have $\{4,6, \dots, 2k+2\}\subseteq \cC(G)$. Then, $\sum_{\ell\in \cC(G)}\frac{1}{\ell}\geq \sum_{\ell=2}^{k+1}\frac{1}{2\ell}>\sum_{\ell=2}^{k}\frac{1}{2\ell}$, contradicting the properties of $G$.

Thus, we must have $|A_1\cup A_2|\leq k$. Then, for each $v\in B_2$, as $v$ has at least $(1-2\eps)k$ neighbours in $H_2$ in $A_2$ and $|A_1|\geq (1-\eps)k$, $v$ has at least $(1-3\eps)k\geq 2$ neighbours in $H_2$ in $A_1$. Thus, as $A_1\subset V(H)$, by the maximality of $H$ we have $B_2\subset V(H)$. This contradicts $|V(H)\cap B_2|\leq 1$, and thus we must have $|A_1\cap A_2|\leq 1$.
\end{proof}


\subsection{Depth of clusters: Proof of Lemma~\ref{lem:cluster properties}\,\ref{prop:cluster:depth}}
To prove Lemma~\ref{lem:cluster properties}\,\ref{prop:cluster:depth}, and only in its proof, we will use the following definition of \emph{cluster depth} and \emph{depth-decreasing paths}.

\begin{definition} Given a cluster $H$ with strictly increasing chain $X^{(0)}\subsetneq \dots\subsetneq X^{(r)}=V(H)$, and $v\in V(H)$, we say the \textit{depth of the vertex $v$} in the cluster $H$ is the smallest $0\leq i\leq r$ for which $v\in X^{(i)}$.

We say that a path $P$ is a \textit{depth-decreasing path from $u$} if $u$ is an endpoint of $P$ and, if the edges of $P$ are oriented to be a directed path from $u$, then for each resulting  edge $\vec{xy}$ the depth of $y$ is smaller than the depth of $x$.
\end{definition}

We now prove Lemma~\ref{lem:cluster properties}\,\ref{prop:cluster:depth}. Again, we work by contradiction, showing that if the paths required do not exist then we have enough additional cycles to violate the conditions of a minimal counterexample.

\begin{proof}[Proof of Lemma~\ref{lem:cluster properties}\,\ref{prop:cluster:depth}.] Let $H$ be a cluster in $G$, and set $A=A(H)$ and $B=B(H)$.
We start with the following claim.

\begin{claim}\label{claim:decreasingpaths} Let $u,v\in V(H)$ and $0\leq i\leq j\leq r$ be such that $u$ has depth $i$ and $v$ has depth $j$ in $H$, where we may have $u=v$. Then, there are depth-decreasing paths $P_u$ and $P_v$ in $H$, from $u$ and $v$ respectively, of total length at least $j$, which both end in $A\cup B$ and are vertex-disjoint except possibly at $u=v$.
\end{claim}
\claimproofstart
We prove this statement by induction on $i+j$. If $i+j=0$, then $u, v\in A\cup B$, and we can take $P_u$ and $P_v$ to be the single-vertex paths containing $u$ and $v$, respectively.
Suppose, then, $d>0$ and that the claim holds whenever $i+j<d$. Let $u,v\in V(H)$ and $0\leq i\leq j\leq r$ such that $i+j=d$, $u$ has depth $i$ and $v$ has depth $j$ in $H$.
We now proceed differently according to whether $j\leq i+1$ or $j\geq i+2$.
In the first case, by the definition of the cluster, we have $|N(v)\cap X^{(j-1)}|\geq 2$. Hence, $v$ has a neighbour $v'\in X^{(j-1)}$ distinct from $u$. Since $v'$ is at depth $j'<j$, we can apply the inductive hypothesis to the pair of vertices $v', u$ at depths $j'\leq  i$, and thus we find depth-decreasing paths $P_{v'}, P_u$ connecting $v', u$ to $A\cup B$. Since $P_{v'}$ is depth-decreasing, it does not contain $v$, and therefore $v$ can be added to $P_{v'}$ to obtain a path $P_v$ connecting $v$ to $A\cup B$. The total length of the paths $P_u, P_{v'}$ is at least $i$, and therefore  the total length of $P_u, P_v$ is at least $i+1\geq j$, as needed.

On the other hand, if $j\geq i+2$, let $v'\in N(v)\cap (X^{(j-1)}\backslash X^{(j-2)})$ be a neighbour of $v$ at depth $j-1$, which exists since $|N(v)\cap X^{(j-1)}|\geq 2$ and $|N(v)\cap X^{(j-2)}|\leq 1$. Applying the inductive hypothesis to the vertices $u, v'$ gives two depth-decreasing paths $P_u, P_{v'}$ of total length at least $j-1$, which connect $u, v'$ to $A\cup B$. Appending $v$ to $P_{v'}$ in the same way as before gives a pair of paths which connect $u, v$ to $A\cup B$, of total length at least $j$, as needed.
\claimproofend

We now show that $r\leq 3\eps k$. Suppose otherwise, and let $j=\lceil 3\eps k\rceil$. Let $v$ be a vertex of depth $j$ in $H$. By Claim~\ref{claim:decreasingpaths}, there exist depth-decreasing paths $P_1,P_2$ starting at $v$ and ending in $A\cup B$, of total length at least $3\eps k$, which are vertex-disjoint except for at $v$.
As $P_1$ and $P_2$ are depth-decreasing paths from $v$, and $v$ has depth $j$ in $H$, both $P_1$ and $P_2$ have length at most $j$. Thus, letting $P= P_1\cup P_2$, $P$ is a path from $A\cup B$ to $A\cup B$ in $G$ with length between $3\eps k$ and $6\eps k+2$. Using that $A,B$ is $(\eps,k)$-almost-complete, extend $P$ by up to one edge at each end as required to get a path, $P'$ say, with length between $3\eps k$ and $6\eps k+4$, endvertices in $A$, no internal vertices in $A$ and at most 2 vertices in~$B$.

Let $t$ be the length of $P'$, let $x$ and $y$ be its endvertices, and let $F=B\cap V(P)$ so that $|F|\leq 2$. Let $2\leq \ell\leq |A|- 1$. Then, by Lemma~\ref{lemma:path lengths in almost complete pairs}, there exists an $x,y$-path with length $2\ell$ in $G[A,B]-F$. Therefore, combining this with $P'$, $G$ contains a cycle with length $t+2\ell$. As this is true for some $t$ with $3\eps k\leq t\leq 6\eps k+4$ and any $\ell$ with $2\leq \ell\leq |A|-1$, and $|A|\geq (1-\eps)k$, we have that, for each $(1-\eps)k\leq \ell\leq k+1$, $G$ contains a cycle of length $2\ell-1$ or $2\ell$. Furthermore, by Corollary~\ref{cor:cycles in almost complete pairs}, $G$ has a cycle of length $2\ell$ for each $2\leq \ell\leq (1-\eps)k$. Thus, for any $2\leq\ell\leq k+1$, $G$ contains a cycle of length $2\ell-1$ or $2\ell$.
Hence, in total, we have
\[
\sum_{\ell\in \mathcal{C}(G)} \frac{1}{\ell} \geq \sum_{\ell=2}^{k+1}\frac{1}{2\ell}> \sum_{\ell=2}^{k}\frac{1}{2\ell},
\]
a contradiction.
Thus, $r\leq 3\eps k$ and hence, applying Claim~\ref{claim:decreasingpaths}, the property in the lemma holds as, for each $u,v\in V(H)$ and $0\leq i\leq j\leq r$  such that $u$ has depth $i$ and $v$ has depth $j$ in $H$, the paths $P_u$ and $P_v$ given have lengths $\ell(P_u)\leq i\leq r\leq 3\eps k$ and $\ell(P_v)\leq j\leq r\leq 3\eps k$.

Since $H[A \cup B]$ has diameter at most $4$, $u$ and $v$ can be connected by a path of length at most $6\eps k+4$, which is obtained by concatenating the paths $P_u, P_v$ and the path of length $\le 4$ connecting the endpoints of $P_u, P_v$.
\end{proof}


\subsection{No cluster cycle: Proof of Lemma~\ref{lem:cluster properties}\,\ref{prop:cluster:nocycle}}\label{sec:cluster3}
Finally in this section, we prove Lemma~\ref{lem:cluster properties}\,\ref{prop:cluster:nocycle}. We work again by contradiction: where such a cluster cycle exists, we will take a certain minimal cluster cycle and then find cycles passing through these cluster cycles, varying the length by changing how we pass through the almost-complete bipartite graphs in each cluster. In total, this will show we have enough additional cycles to violate the conditions of a minimal counterexample.

\begin{proof}[Proof of Lemma~\ref{lem:cluster properties}\,\ref{prop:cluster:nocycle}.] Suppose, for contradiction, that $G$ contains a cluster cycle of edge-disjoint $(\eps,k)$-clusters, where the union of no two clusters is itself a $(\eps,k)$-cluster. Let $s$ be the minimal length of such a cluster cycle, and let $H_1, \dots, H_s$ be edge-disjoint clusters which form a cluster cycle in that order with connecting paths $P_1, \dots, P_s$, which moreover minimise $\sum_{i\in [s]}|P_i|$. Throughout this proof, we will work with indices $\mathrm{mod}\;s$ so that, for example, $P_{s+1}=P_1$.

For each $i\in [s]$, let $A_i=A(H_i)$ and $B_i=B(H_i)$. Label vertices so that, for each $i\in [s]$, $P_i$ is a $y_{i},x_{i+1}$-path with $y_i\in V(H_i)$ and $x_{i+1}\in V(H_{i+1})$. We start with the following claim.

\begin{claim} If $s\geq 3$, then, for each $i\in [s]$ the following hold. \textup{\textbf{a)}} For each $j\in [s]$, if $i\notin \{j-1,j,j+1\}$, $V(H_i)\cap V(H_j)=\emptyset$. \textup{\textbf{b)}} If $P_i$ is a single-vertex path, then $|V(H_i)\cap V(H_{i+1})|=1$.
\textup{\textbf{c)}}  $P_i$ contains no vertices in $H_j$ for each $j\in [s]$ with $j\neq i,i+1$.
\label{claim:s>2}
\textup{\textbf{d)}} If $|P_i|\geq 2$, then $V(H_i)\cap V(H_{i+1})=\emptyset$. 
\textup{\textbf{e)}} $|P_i\cap H_j|\le 1$ for all $i,j$.
\end{claim}
\claimproofstart[Proof of Claim~\ref{claim:s>2}]  Suppose, for contradiction, that there are distinct non-adjacent clusters $H_i, H_j$ and some vertex $v\in V(H_i)\cap V(H_j)$. As the connecting paths of a cluster cycle are vertex-disjoint, either $v\notin V(P_i)\cup V(P_{i+1})\cup\dots \cup V(P_{j-1})$ or $v\notin V(P_j)\cup V(P_{j+1})\cup \dots\cup V(P_{i-1})$. By relabelling (and reversing the direction of the cluster cycle) if necessary, we can assume that $v\notin V(P_i)\cup V(P_{i+1})\cup\dots \cup V(P_{j-1})$. Let $P_v$ be the path of length 0 with vertex set $\{v\}$. Then, $H_i, H_{i+1}, \dots, H_j$ is a $(\eps,k)$-cluster cycle with connecting paths $P_i, \dots, P_{j-1}, P_v$. As $H_i, H_j$ are non-adjacent clusters (and thus $j\neq i-1$), this cluster cycle has length less than $s$, a contradiction. Thus, for all distinct non-adjacent clusters $H_i, H_j$, we have $V(H_i)\cap V(H_j)=\emptyset$, so that \textbf{a)} holds.

Now, as $s> 2$, there is no cluster cycle of length two and thus $|V(H_i)\cap V(H_j)|\leq 1$ for every $i,j\in [s]$ with $i\neq j$. Thus, for each $i\in [s]$, if $P_i$ is a single vertex, then $y_{i}=x_{i+1}$ and so $|V(H_i)\cap V(H_{i+1})|= 1$. Hence, \textbf{b)} holds. 

Suppose that there is some $i\in [s]$ for which there is some $j\in [s]$ with $i\notin \{j,j+1\}$ and $V(P_j)\cap V(H_i)\neq \emptyset$. Then, $P_j$ contains a path, $P_j'$ say, of length at most 4 from $V(H_j)$ to $V(H_i)$. Then, $H_i,H_{i+1},\dots,H_j$ is a $(\eps,k)$-cluster cycle with connecting paths $P_i,P_{i+1},\dots,P_{j-1},P_j'$, contradicting the choice of $s$. Thus, \textbf{c)} holds.

Suppose that for some $i\in [s]$, we have a vertex  $v\in V(H_i)\cap V(H_{i+1})$ and $|P_i|\ge 2$. 
Note that we can't have $v\in P_j$ for any $j\ne i$. Indeed if this happened then we'd have $P_j\cap H_i\ne \emptyset$ and $P_j\cap H_{i+1}\ne \emptyset$, contradicting \textbf{c)}.
Thus, we can replace $P_i$ with the single-vertex path $v$, contradicting the minimality of $\sum_{i\in [s]}|P_i|$. Therefore, \textbf{d)} holds.

Note that \textbf {e)} just needs to be proved for $j\in \{i, i+1\}$, since it is otherwise implied by \textbf {c)}. For $j\in \{i, i+1\}$, if $|P_i\cap H_j|\ge 2$, then we can shorten $P_i$ to get another path from $H_{i}$ to $H_{i+1}$, contradicting the minimality of $\sum_{i\in[s]}|P_i|$.
\claimproofend

We now find paths between the $(\eps,k)$-almost-complete pairs $(A(H_i),B(H_i))$, $i\in [s]$, in a cycle, as in the following claim.

\begin{claim}\label{clm:betterpaths}
There are distinct vertices $u_i,v_i\in A_{i}$, $i\in [s]$, and vertex-disjoint paths $Q_i$, $i\in [s]$, such that the following hold. 
\stepcounter{propcounter}
  \begin{enumerate}[label = \textup{\textbf{\Alph{propcounter}\arabic{enumi}}}]
\item For each $i\in [s]$, $Q_i$ is a $v_{i},u_{i+1}$-path with length at most $10\eps k$.\label{prop:path:short}
\item For each $i\in [s]$, $|V(\cup_{j\in [s]}Q_j)\cap (A_i\cup B_i)|\leq 8$.\label{prop:path:intersection}
\end{enumerate}
\end{claim}
\claimproofstart[Proof of Claim~\ref{clm:betterpaths}]
 We will do this in two cases, \textbf{a)} $s=2$ and \textbf{b)} $s\geq 3$.
\smallskip

Case \textbf{a)}: $s=2$.
 By the minimality of
$|P_1|+|P_2|$, we may assume that each $P_i$ has no internal
vertex in $V(H_1)\cup V(H_2)$. Apply Lemma~\ref{lem:cluster properties}\textup{(b)} in $H_1$ to
$x_1,y_1$ to get vertex-disjoint paths from these vertices to $A_1\cup B_1$. Concatenate these paths
with $P_2,P_1$ to get two vertex-disjoint paths from  $A_1\cup B_1$ to $H_2$. Cutting each path at its first
vertex in $H_2$, we obtain two vertex-disjoint $A_1\cup B_1,H_2$-paths of
length at most $3\eps k+4$, whose internal vertices lie outside
$H_2$. In particular, their endpoints in $H_2$ are distinct. 
Applying Lemma~\ref{lem:cluster properties}\textup{(b)} in $H_2$
to these endpoints and concatenating, we obtain two vertex-disjoint
$A_1\cup B_1, A_2\cup B_2$-paths, each of length at most $6\eps k+4$. Replace each
by a minimal $A_1\cup B_1,A_2\cup B_2$-subpath. Thus, each of the two paths contains
at most one vertex of each $A_i\cup B_i$.

If an endpoint in $A_i\cup B_i$ lies in $B_i$, extend the path at that
endpoint by one edge to get paths ending in $A_i$. These extensions
can be chosen to preserve vertex-disjointness, since every vertex
of $B_i$ has at least $(1-2\eps)k$ neighbours in $A_i$. Relabelling the resulting
paths as $Q_1,Q_2$, we obtain the
required paths. Indeed, each has length at most
$6\eps k+6\leq 10\eps k$, 
and only the first/last two vertices in each path can be in 
$A_1\cup B_1$/$A_2\cup B_2$.
\smallskip

Case \textbf{b)}: $s\geq 3$. For each $i\in [s]$, by Lemma~\ref{lem:cluster properties}\,\ref{prop:cluster:depth} (and appending a vertex in $A_i$ if necessary), there are vertex-disjoint paths $R_i^-$ and $R_i^+$ in $H_i$ with length at most $3\eps k+1$ each, which connect $x_i$ and $y_i$ respectively to $A_i$ and have $|(V(R_i^-)\cup V(R_i^+))\cap B_i|\leq 2$.
For each $i\in [s]$, let $Q_i=R_{i}^+\cup P_i\cup R_{i+1}^-$ and let $u_i$ and $v_i$ be the endvertex of $R_i^-$ and $R_i^+$ in $A_i$, respectively. 
Then, we have that the vertices $u_i,v_i\in A_{i}$, $i\in [s]$, and paths $Q_i$, $i\in [s]$, satisfy \ref{prop:path:short}.

That the paths $Q_i$, $i\in [s]$, are vertex disjoint and $|V(\cup_{j\in [s]}Q_j)\cap (A_i\cup B_i)|\leq 8$ for each $i\in [s]$, follows from Claim~\ref{claim:s>2}. Finally, the length of $Q_i$ is at most $(3\eps k+1)+4+(3\eps k+1)\leq 10\eps k$, as required.
\claimproofend

Using Claim~\ref{clm:betterpaths}, let $u_i,v_i\in A_i$, $i\in [s]$, and $Q_i$, $i\in [s]$, be such that the properties in that claim hold. We now show that we can connect these paths in $G$ to find many different cycle lengths.

\begin{claim} For each $\ell\in \N$ with $sk/2\leq \ell\leq 3ks/4$, $\mathcal{C}(G)$ contains a cycle of length $2\ell-1$ or $2\ell$.\label{claim:nowcyclelengths}
\end{claim}
\claimproofstart[Proof of Claim~\ref{claim:nowcyclelengths}] Let $r$ be the total length of $Q_1,\ldots,Q_s$ so that, by Claim~\ref{clm:betterpaths}, we have $r\leq 10\eps sk$. Let $sk/2\leq \ell\leq 3ks/4$. Then, we can pick integers $4\leq \ell_1,\dots,\ell_s\leq \lceil 3k/4\rceil$ such that $\lceil r/2\rceil+\ell_1+\dots+\ell_s=\ell$. By Lemma~\ref{lem:cluster properties}\,\ref{prop:cluster:vxdisjoint} (using that the union of any pair of clusters in the cycle is not a cluster), \ref{prop:path:intersection},  and Lemma~\ref{lemma:path lengths in almost complete pairs}, for each $i\in [s]$,  there is a $u_i,v_i$-path $R_i$ in $G[A_i,B_i]-((V(\cup_{j\in [s]}Q_j)\cup A_{i-1}\cup B_{i-1}\cup A_{i+1}\cup B_{i+1})\setminus \{u_i,v_i\})$ with length $2\ell_i$. Then, the concatenation of $R_1,Q_1,R_2,Q_2,\dots,R_s,Q_s$ is a cycle of length $r+2\ell_1+\dots+2\ell_s\in \{2\ell-1,2\ell\}$.
\claimproofend

By Corollary~\ref{cor:cycles in almost complete pairs}, $G[A_1,B_1]$ has a cycle of length $2\ell$ for each $2\leq \ell\leq (1-\eps)k$.
Hence, using Claim~\ref{claim:nowcyclelengths}, we have
\begin{align*}
\sum_{\ell\in \mathcal{C}(G)} \frac{1}{\ell} &\geq\sum_{\ell=2}^{(1-\eps)k}\frac{1}{2\ell}+ \sum_{\ell=sk/2}^{3ks/4}\frac{1}{2\ell}
\geq\sum_{\ell=2}^{(1-\eps)k}\frac{1}{2\ell}+\left(\frac{ks}{4}-1\right)\cdot \left(2\cdot \frac{3ks}{4}\right)^{-1}\\
&=\sum_{\ell=2}^{(1-\eps)k}\frac{1}{2\ell}+\frac{1}{6}-\frac{2}{3sk}>\sum_{\ell=2}^{(1-\eps)k}\frac{1}{2\ell}+(\eps k+1)\cdot \frac{1}{(1-\eps)2k}
>\sum_{\ell=2}^{k}\frac{1}{2\ell},
\end{align*}
a contradiction to the properties of $G$.\end{proof}


\section{Clusters cover most of a counterexample}\label{sec:main_prop_proof}

We will now prove the following lemma, Lemma~\ref{lem:few uncovered vertices}, which shows that, in a minimal counterexample to Theorem~\ref{thm:main},  a suitably maximal edge-disjoint collection of clusters covers all but few vertices in the graph (i.e., that \eqref{eq:mostly covers V(G)} holds).

\begin{lemma}\label{lem:few uncovered vertices} Let $1/k_0\ll \eps \ll 1$.  Let $n$ be the least integer for which there is an integer $k_0\leq k\leq n/2$ and an $n$-vertex graph $G$ with more than $(k-1)(n-k+1)$ edges for which either \textup{\textbf{a)}} $\sum_{\ell\in \cC(G)}\frac{1}{\ell}<\sum_{\ell=2}^k\frac{1}{2\ell}$ or \textup{\textbf{b)}} $\sum_{\ell\in \cC(G)}\frac{1}{\ell}= \sum_{\ell=2}^k\frac{1}{2\ell}$, $e(G)\geq k(n-k)$ and $G$ is not a copy of $K_{k,n-k}$.
Let $G$ be such a graph.

Let $\mathcal{H}$ be a collection of edge-disjoint $(\eps,k)$-clusters in $G$ which maximises $\sum_{H\in \cH}e(H)$ and, subject to this, minimises $|\cH|$. Then,
\begin{equation}\label{eq:mostly covers V(G)}
\bigg|V(G)\setminus V\bigg(\bigcup_{H\in \cH} H\bigg)\bigg|\leq \frac{2000|\mathcal{H}|}{\eps}.
\end{equation}
\end{lemma}

 We split this proof into three lemmas. The first two of these show that any $n$-vertex graph without a certain almost-complete pair but with a good average and minimum degree condition has strictly larger harmonic sum of cycle lengths than $K_{k,n-k}$, dealing respectively with the case $n\leq k^{1.01}$ (Lemma~\ref{lem:smalldensenoalmostcomplete}) and $n>k^{1.01}$ (Lemma~\ref{lemma:dense patches in expanders}). The last of these three lemmas (Lemma~\ref{lemma:many edges outside clusters}) then shows that in our minimal counterexample a maximal edge-disjoint collection of clusters, which does not cover some vertices, does not cover some dense subgraph.

\begin{lemma}\label{lem:smalldensenoalmostcomplete}
Let $1/k\ll \eps,\eps_1 \ll 1$ and $n\leq k^{1.01}$. Let $\Gamma$ be an $n$-vertex $(\eps_1,k/2)$-expander with no $(\eps,k)$-almost-complete pair in $\Gamma$ and such that $d(\Gamma)\geq (1-0.1\eps)2k$ and $\delta(\Gamma)\geq (1-0.1\eps)k$. Then, $\sum_{\ell\in \cC(\Gamma)} \frac{1}{\ell}> \sum_{\ell=2}^k\frac{1}{2\ell}$.
\end{lemma}

\begin{lemma}\label{lemma:dense patches in expanders}
Let $1/k\ll \eps,\eps_1 \ll 1$ and $n> k^{1.01}$. Let $\Gamma$ be an $n$-vertex $(\eps_1,k/2)$-expander with no $(\eps,k)$-almost-complete pair in $\Gamma$ and such that $d(\Gamma)\geq (1-0.1\eps)2k$ and $\delta(\Gamma)\geq (1-0.1\eps)k$. Then, $\sum_{\ell\in \cC(\Gamma)} \frac{1}{\ell}> \sum_{\ell=2}^k\frac{1}{2\ell}$.

\end{lemma}

\begin{lemma}\label{lemma:many edges outside clusters}
Let $1/k_0\ll \eps \ll 1$. Let $n$ be the least integer for which there is an integer $k_0\leq k\leq n/2$ and an $n$-vertex graph $G$ with more than $(k-1)(n-k+1)$ edges for which either \textup{\textbf{a)}} $\sum_{\ell\in \cC(G)}\frac{1}{\ell}<\sum_{\ell=2}^k\frac{1}{2\ell}$ or \textup{\textbf{b)}} $\sum_{\ell\in \cC(G)}\frac{1}{\ell}= \sum_{\ell=2}^k\frac{1}{2\ell}$, $e(G)\geq k(n-k)$ and $G$ is not a copy of $K_{k,n-k}$.
Let $G$ be such a graph.

Let $\mathcal{H}$ be a collection of edge-disjoint $(\eps,k)$-clusters in $G$ which maximises $\sum_{H\in \cH}e(H)$.
Suppose
\begin{equation}\label{eq:doesnt mostly cover V(G)}
\bigg|V(G)\setminus V\bigg(\bigcup_{H\in \cH} H\bigg)\bigg|>\frac{2000|\mathcal{H}|}{\eps}.
\end{equation}
 Then, there is a subgraph $\Gamma\subset G-\bigcup_{H\in \cH}E(H)$ with $d(\Gamma)\geq (1-0.01\eps)2k$.
\end{lemma}

Lemma~\ref{lem:few uncovered vertices} follows shortly from these lemmas, as follows.

\begin{proof}[Proof of Lemma~\ref{lem:few uncovered vertices} from Lemmas~\ref{lem:smalldensenoalmostcomplete},~\ref{lemma:dense patches in expanders} and~\ref{lemma:many edges outside clusters}.] Let
$\eps_1$ be such that $1/k\ll \eps_1\ll \eps$.
Suppose that \eqref{eq:mostly covers V(G)} does not hold. Then, as \eqref{eq:doesnt mostly cover V(G)} holds,
by Lemma~\ref{lemma:many edges outside clusters} there is a subgraph $\Gamma\subset G-\cup_{H\in \cH}E(H)$ with $d(\Gamma)\geq (1-0.01\eps)2k$. By Theorem~\ref{thm-expander} with $C=61, \eps_2=1/3$, there is a subgraph $\Gamma'\subset \Gamma$ with $d(\Gamma')\geq (1-C\eps_1)(1-0.01\eps)2k\ge (1-0.02\eps)2k$ and $\delta(\Gamma')\geq d(\Gamma')/2\ge (1-0.02\eps)k$ which is an $(\eps_1,\eps_2 d(\Gamma))$-expander. As $\eps_2 d(\Gamma)\geq k/2$ and $\delta(\Gamma')\geq (1-0.02\eps)k$, we have that $\Gamma'$ is an $(\eps_1,k/2)$-expander (expansion for sets of order $x\le \eps_2 d/2$ follows from minimum degree, while for larger sets it follows from $(\eps_1,\eps_2 d(\Gamma))$-expansion). Furthermore, by the maximality of $\sum_{H\in \cH}e(H)$, as $\Gamma'\subset \Gamma$ is edge-disjoint from each $H\in \cH$, $\Gamma'$ contains no $(\eps,k)$-cluster and, hence, no $(\eps,k)$-almost-complete pair.
 If $|\Gamma'|\leq k^{1.01}$, then by Lemma~\ref{lem:smalldensenoalmostcomplete} we get a contradiction. If $|\Gamma'|> k^{1.01}$, then by Lemma~\ref{lemma:dense patches in expanders} we get a contradiction.
\end{proof}

In the remainder of this section, we will prove Lemmas~\ref{lem:smalldensenoalmostcomplete},~\ref{lemma:dense patches in expanders} and~\ref{lemma:many edges outside clusters}, in their respective subsections.


\subsection{Sublinear expansion}\label{sec:sublin}

Following Koml\'os and Szemer\'edi~\cite{K-Sz-1,K-Sz-2}, we use the following definition of a (sublinear) expander. For more on sublinear expansion, see the recent survey of Letzter~\cite{Letzter2024SublinearExpanders}, as well as the less detailed overview by Montgomery~\cite{Montgomery2026GraphTheoryExpansion}.

\begin{definition}
For each $\eps_1>0$ and $k>0$, a graph $G$ is an \emph{$(\eps_1,k)$-expander} if
$$|N(X)|\geq \rho(|X|,\eps_1,k)\cdot |X|$$
for all $X\subseteq V(G)$ with $k/2\leq |X|\leq |G|/2$, where
\begin{eqnarray}\label{epsilon}
\rho(x,\eps_1,k):=\left\{\begin{tabular}{ l l }
$0$ & $\mbox{ if } x<k/5$, \\
$\eps_1/\log^2(15x/k)$ & $\mbox{ if } x\ge k/5$. \\
\end{tabular}
\right.
\end{eqnarray}
\end{definition}

As Koml\'os and Szemer\'edi~\cite{K-Sz-2} showed, every graph $G$ contains an expander with comparable average degree to $G$. We will use the following version of this, by Haslegrave, Kim and Liu (see \cite[Lemma~3.2]{HaslegraveKimLiu2022ExtremalDensity}).

\begin{theorem}\label{thm-expander}
Let $C>60$, $0<\eps_1\leq 1/20C$, $0<\eps_2<1/2$, and let $\rho$ be as in \eqref{epsilon}. Then, every graph $G$ has an $(\eps_1,\eps_2 d(G))$-expander subgraph $H$ with $d(H)\geq (1-C\eps_1)d(G)$ and $\delta(H)\geq d(H)/2$.
\end{theorem}

We will use the following path connection result, due to Koml\'os and Szemer\'edi~\cite{K-Sz-2}.

\begin{lemma}\label{lem-diameter} Let $\eps_1,k>0$.
	If $G$ is an $n$-vertex $(\eps_1,k)$-expander, then any two vertex sets, each of size at least
	$x\ge k$, are at distance at most $\frac{2}{\eps_1}\log^3(15n/k)$ from each other. This remains true even after deleting $x\cdot \rho(x,\eps_1,k)/4$ arbitrary vertices from $G$, where $\rho$ is as defined in \eqref{epsilon}.
\end{lemma}


\subsection{Proof of Lemma~\ref{lem:smalldensenoalmostcomplete}}
We now prove Lemma~\ref{lem:smalldensenoalmostcomplete}, which shows that if $\Gamma$ is a small sublinear expander with no almost complete pair and average degree at least $(1-o(1))2k$, then the harmonic sum of its cycle lengths exceeds that of $K_{k,k}$. We do so by contradiction, and show first that $\Gamma$ is bipartite (see Claim~\ref{claim:Gammabipartite}) and then that, when $S$ is a longest cycle in $\Gamma$, then $\Gamma-V(S)$ contains no subgraph with large average degree (see Claim~\ref{claim:longest_cycle}).

\begin{proof}[Proof of Lemma~\ref{lem:smalldensenoalmostcomplete}] We have $1/k\ll \eps,\eps_1 \ll 1$ and $n\leq k^{1.01}$, while $\Gamma$ is an $n$-vertex $(\eps_1,k/2)$-expander with no pair which is $(\eps,k)$-almost-complete in $\Gamma$ and such that $d(\Gamma)\geq (1-0.1\eps)2k$ and $\delta(\Gamma)\geq (1-0.1\eps)k$. Assume, for contradiction, that $\sum_{\ell\in \mathcal{C}(\Gamma)}\frac{1}{\ell}\leq \sum_{\ell=2}^k\frac{1}{2\ell}$.

Using Theorem~\ref{thm:GHS}, let $C$ be such that the property in that theorem holds.
Let $c=k^{-0.01}/2$, so that $\delta(\Gamma)\geq (1-0.1\eps)k\geq cn$.  Then, setting $K=C\cdot c^{-5}$, and using that, as $1/k\ll 1$,  $K\leq k^{0.1}$ and $n\geq k\geq 45K/c^4$, by the property of $C$ applied with $\Gamma$, $\cC(\Gamma)$ contains all even cycle lengths between $4$ and $\ec(\Gamma)-k^{0.1}$ and all odd cycle lengths between $k^{0.1}$ and $\oc(\Gamma)-k^{0.1}$.
Furthermore, by the Erd\H{o}s-Gallai theorem (Theorem~\ref{thm:EG}), we have that $\max\{\ec(\Gamma),\oc(\Gamma)\}\geq d(\Gamma)\geq (1-0.1\eps)2k$. As is well-known, $\Gamma$ contains a bipartite subgraph with average degree at least $d(\Gamma)/2$, and thus, again  by Theorem~\ref{thm:EG}, $\ec(\Gamma)\geq d(\Gamma)/2\geq (1-0.1\eps)k$.
Therefore, if $\oc(\Gamma)\geq (1-0.1\eps)2k$, we have
\[
\sum_{\ell\in \mathcal{C}(\Gamma)}\frac{1}{\ell}= \sum_{\ell\in \mathcal{C}(\Gamma):\ell=0\,\mathrm{mod}\,2}\frac{1}{\ell}+\sum_{\ell\in \mathcal{C}(\Gamma):\ell=1\,\mathrm{mod}\,2}\frac{1}{\ell}\geq \sum_{\ell=2}^{k/3}\frac{1}{2\ell}+\sum_{\ell=K}^{(1-0.1\eps)k-K}\frac{1}{2\ell-1}>\sum_{\ell=2}^k\frac{1}{2\ell}.
\]
Thus, we can assume that $\oc(\Gamma)< (1-0.1\eps)2k$ and, hence, $\ec(\Gamma)\geq (1-0.1\eps)2k$. Note that, as $\sum_{\ell\in \mathcal{C}(\Gamma)}\frac{1}{\ell}\leq \sum_{\ell=2}^k\frac{1}{2\ell}$, we have $\ec(\Gamma)\leq 2k+K\leq 2k+k^{0.1}$.

 We now show the following claim.

\begin{claim}\label{claim:Gammabipartite}
$\Gamma$ is bipartite.
\end{claim}
\claimproofstart
Suppose, for contradiction, that $\Gamma$ contains an odd cycle. Taking a shortest odd cycle $S$ in $\Gamma$, we then have that $|S|\leq 2k^{0.1}+1$ (if $\oc(\Gamma)\le 2k^{0.1}+1$ this is immediate, otherwise $\mathcal{C}(\Gamma)$ contains every odd number between $k^{0.1}$ and $\oc(\Gamma)-k^{0.1}\ge k^{0.1}+1$ giving at least one number $\le 2k^{0.1}+1$). Let $v\in V(S)$. Greedily, using $\delta(\Gamma)\geq (1-0.1\eps)k$, let $P$ be a path from $v$ in $\Gamma-(V(S)\setminus \{v\})$ with length $\sqrt{k}$. Letting $u$ be the other endpoint of $P$ and applying Lemma~\ref{lem-diameter} (removing the vertices in $V(P)\setminus \{u\}$, and applying the lemma with sets $\{u\}\cup (N_\Gamma(u)\setminus V(P))$ and $(V(S)\cup N_\Gamma(V(S)\setminus \{v\}))\setminus V(P)$),
 find a path $Q$ from $u$ to $V(S)\setminus \{v\}$ in $\Gamma-(V(P)\setminus \{u\})$. Note that $S\cup P\cup Q$ has an odd cycle containing $P$, and thus $\oc(\Gamma)\geq \sqrt{k}$. Then,
\[
\sum_{\ell\in \mathcal{C}(\Gamma)}\frac{1}{\ell}= \sum_{\ell\in \mathcal{C}(\Gamma):\ell=0\,\mathrm{mod}\,2}\frac{1}{\ell}+\sum_{\ell\in \mathcal{C}(\Gamma):\ell=1\,\mathrm{mod}\,2}\frac{1}{\ell}
\geq \sum_{i=2}^{(1-0.1\eps)k-K}\frac{1}{2i}+\sum_{i=K}^{\sqrt{k}/2-K}\frac{1}{2i+1}>\sum_{\ell=2}^k\frac{1}{2\ell},
\]
a contradiction.
\claimproofend

Let $S$ be  a longest cycle in $\Gamma$, and note that $|S|\leq 2k+k^{0.1}$.

\begin{claim}\label{claim:longest_cycle}
Every subgraph $\Gamma'\subset \Gamma - V(S)$ satisfies $d(\Gamma')\leq 0.01\eps k$.
\end{claim}
\claimproofstart
By the Erd\H{o}s-Gallai theorem (Theorem~\ref{thm:EG}), it suffices to show that $\Gamma-V(S)$ does not contain a cycle $S'$ of length at least $0.01\eps k$, since one can find such a cycle in any $\Gamma'$ with $d(\Gamma')\geq 0.01\eps k$.

Suppose, for contradiction, that $\Gamma-V(S)$ contains a cycle $S'$ with length at least $0.01\eps k$. Let
\[
r=k^{0.1}\quad \text{ and } \quad \ell_0=\frac{2}{\eps_1}\log^3(30n/k)\leq k^{0.1},
\]
where we have used that $n\leq k^{1.01}$ and $1/k\ll \eps_1$. We claim there exists a collection of $r$ vertex-disjoint paths between $V(S)$ and $V(S')$ of length at most $\ell_0+2$.

One can construct it greedily: if $\cP$ is the collection of paths found so far, delete $V(\cP)$ from the graph and apply Lemma~\ref{lem-diameter} to connect the cycles $S$ and $S'$ in the remaining graph. Since Lemma~\ref{lem-diameter} applies to sets of size at least $k/2$, we enlarge the sets $V(S)\setminus V(\cP)$ and $V(S')\setminus V(\cP)$ by adding their neighbours in $\Gamma-V(\cP)$ to them (which are sufficiently large due to the minimum degree condition $\delta(\Gamma)\geq (1-0.1\eps)k$). Lemma~\ref{lem-diameter} then guarantees a path of length at most $\ell_0$ between these two enlarged sets, which can be turned into a path of length at most $\ell_0+2$ between $V(S)$ and $V(S')$ avoiding $V(\cP)$. Lemma~\ref{lem-diameter} applies since the number of vertices removed from $\Gamma$ is at most $r(\ell_0+3)\ll \eps_1k$.

We claim that there exists a cycle longer than $S$ in $\Gamma$. To construct it, consider two of these paths with vertices closest together on $S$, which are within distance at most $\frac{2k+k^{0.1}}{r}\leq 0.005\eps k$ of each other on $S$. If, instead of taking the shorter path on $S$ between these two endpoints, one walks along these two paths and the longer arc of $S'$, one obtains a cycle longer than $S$, which is a contradiction (see Figure~\ref{fig:extending_the_cycle} \textbf{a)} for an illustration).
\claimproofend

Let $A'\cup B'$ be a bipartition of $S$ and let $A''\cup B''$ be the bipartition of $V(\Gamma)\setminus (A'\cup B')$, such that $A'\cup A''$, $B'\cup B''$ is a bipartition of $\Gamma$. Then, $(1-0.1\eps)k\leq |A'|=|B'|<k+k^{0.1}$, due to the constraints we have on the length of $S$.

\begin{claim}\label{claim:finding_B}
There are at least $7k$ vertices in $B''$ with at least $(1-0.2\eps)k$ neighbours in $A'$.
\end{claim}
\claimproofstart
Note that, if there are $u\in B''$ and $v\in A''$ with $d_\Gamma(u,A'),d_\Gamma(v,B')\geq 4k/5$, then we can find consecutive vertices $v_1,v_2,v_3,v_4$ on $S$ with $v_1,v_3\in N_\Gamma(u)$ and $v_2,v_4\in N_\Gamma(v)$, whereupon $S-v_1v_2-v_3v_4+uv_1+uv_3+vv_2+vv_4$ gives a longer cycle than $S$ in $\Gamma$, a contradiction (see Figure~\ref{fig:extending_the_cycle} \textbf{b)} for an illustration). Thus, we can assume, without loss of generality, that no vertex in $A''$ has more than $4k/5$ neighbours in $B'$.

Thus, each vertex in $A''$ must have at least $(1-0.1\eps)k-4k/5\geq k/10$ neighbours in $B''$. As $d(\Gamma-V(S))<0.01\eps k$, we therefore have
\[
2|A''|\cdot \frac{k}{10}\leq 2e(\Gamma-V(S))< 0.01\eps k (|A''|+|B''|),
\]
so that $|A''|\leq 0.1\eps |B''|$. 

Now, if  $|B'\cup B''|\leq 10k$, we have that $|A'\cup A''|\leq k+k^{0.1}+\eps k\leq (1+2\eps)k$ and, thus, as $|A'\cup A''|\ge |A'|\geq (1-0.1\eps)k\geq |B'\cup B''|/20$ and $\eps\ll 1$,
\[
d(\Gamma)\leq \frac{2|A'\cup A''|\cdot |B'\cup B''|}{|A'\cup A''|+|B'\cup B''|} \leq \frac{2(1+2\eps)k\cdot |B'\cup B''|}{21|B'\cup B''|/20}<(1-0.1\eps)2k,
\]
a contradiction. Thus, $|B'\cup B''|\geq 10k$, and $|B''|\geq 8k$.

As $d(\Gamma')<0.01\eps k$ for every $\Gamma'\subset \Gamma-V(S)$, at most $|A''|$ vertices in $B''$ have at least $0.01\eps k$ neighbours in $A''$ in $\Gamma$. Thus, as $|A''|\leq 0.1\eps |B''|$ and $\delta(\Gamma)\geq (1-0.1\eps)k$, all but at most $0.1\eps |B''|$ vertices of $B''$ have at least $(1-0.2\eps)k$ neighbours in $A'$. Since $|B''|\geq 8k$, this means at least $7k$ vertices of $B''$ have degree at least $(1-0.2\eps)k$ to $A'$.
\claimproofend

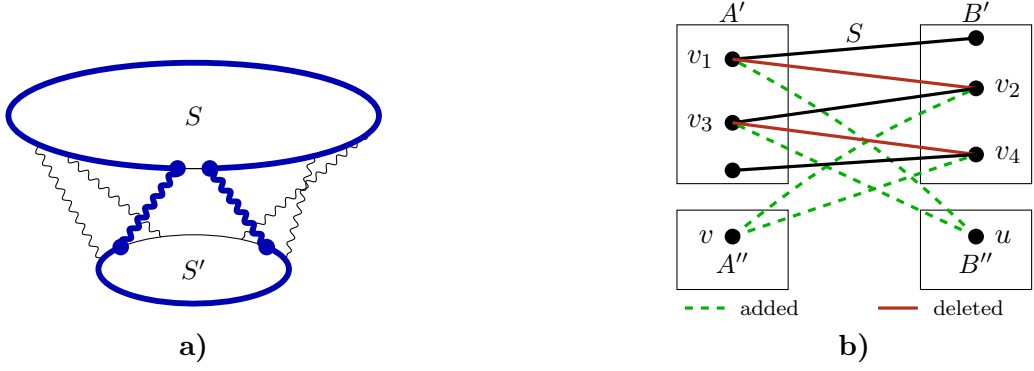
\begin{figure}
\centering
\def\xcenterS{0}
\def\ycenterS{1.4}
\def\xradS{3.5}
\def\yradS{1}

\def\xcenterSp{0}
\def\ycenterSp{-1.5}
\def\xradSp{1.8}
\def\yradSp{0.65}

\def\angleXS{265}
\def\angleYS{275}
\def\angleXSp{140}
\def\angleYSp{40}

\def\auxiliaryangles{210/165,225/110,310/15,340/60}

\begin{minipage}[t]{0.48\textwidth}
\centering
\begin{tikzpicture}[
    newcycle/.style={draw=blue!70!black, line width=2.2pt},
    wiggly path/.style={decorate, decoration={snake, amplitude=1.2pt, segment length=5pt}},
    vertex/.style={circle, fill=black, inner sep=1.8pt},
    selected vertex/.style={circle, fill=blue!70!black, inner sep=2.2pt},
    scale=0.7
]

\draw (\xcenterS,\ycenterS) ellipse [x radius=\xradS, y radius=\yradS];

\draw (\xcenterSp,\ycenterSp) ellipse [x radius=\xradSp, y radius=\yradSp];

\node[font=\small] at (\xcenterS,\ycenterS) {$S$};
\node[font=\small] at (\xcenterSp,\ycenterSp) {$S'$};

\coordinate (x) at ({\xcenterS+\xradS*cos(\angleXS)},  {\ycenterS+\yradS*sin(\angleXS)});

\coordinate (y) at ({\xcenterS+\xradS*cos(\angleYS)},  {\ycenterS+\yradS*sin(\angleYS)});

\coordinate (xp) at ({\xcenterSp+\xradSp*cos(\angleXSp)}, {\ycenterSp+\yradSp*sin(\angleXSp)});

\coordinate (yp) at ({\xcenterSp+\xradSp*cos(\angleYSp)},  {\ycenterSp+\yradSp*sin(\angleYSp)});

\foreach \angleS/\angleSp in \auxiliaryangles {
    \draw[wiggly path] ({\xcenterS+\xradS*cos(\angleS)}, {\ycenterS+\yradS*sin(\angleS)}) -- ({\xcenterSp+\xradSp*cos(\angleSp)}, {\ycenterSp+\yradSp*sin(\angleSp)});
}

\draw (x) arc[start angle=\angleXS,
    end angle=\angleYS,
    x radius=\xradS,
    y radius=\yradS
];

\draw[newcycle]
(y) arc[start angle=\angleYS,
    end angle={\angleXS+360},
    x radius=\xradS,
    y radius=\yradS
];

\draw[newcycle, wiggly path] (x) -- node[left, font=\small] {} (xp);
\draw[newcycle, wiggly path] (y) -- node[right, font=\small] {} (yp);

\draw[newcycle]
(xp) arc[
    start angle=\angleXSp,
    end angle={\angleYSp+360},
    x radius=\xradSp,
    y radius=\yradSp
];

\node[selected vertex] at (x) {};
\node[selected vertex] at (y) {};
\node[selected vertex] at (xp) {};
\node[selected vertex] at (yp) {};

\end{tikzpicture}
\par\smallskip
\textbf{a)}
\end{minipage}
\hfill
\begin{minipage}[t]{0.48\textwidth}
\centering
\begin{tikzpicture}[
    vertex/.style={circle, fill=black, inner sep=2.1pt},
    added edge/.style={draw=green!70!black, line width=1.2pt, dashed},
    deleted edge/.style={draw=BrickRed, line width=1.2pt},
    scale=0.7
]
\def\xleft{-2.3}
\def\xright{2.3}
\def\boxhalfwidth{1.05}
\def\boxtop{2.5}
\def\boxsplit{-0.75}
\def\boxbottom{-2.5}
\path[use as bounding box] (-3.7,-2.9) rectangle (3.7,3);
\draw ({\xleft-\boxhalfwidth},\boxbottom) rectangle ({\xleft+\boxhalfwidth},\boxsplit-0.25);
\draw ({\xright-\boxhalfwidth},\boxbottom) rectangle ({\xright+\boxhalfwidth},\boxsplit-0.25);

\draw ({\xleft-\boxhalfwidth},\boxsplit+0.25) rectangle ({\xleft+\boxhalfwidth},\boxtop);
\draw ({\xright-\boxhalfwidth},\boxsplit+0.25) rectangle ({\xright+\boxhalfwidth},\boxtop);

\node[font=\small] at (\xleft,2.75) {$A'$}; \node[font=\small] at (\xright,2.75) {$B'$};
\node[font=\small] at (\xleft,-2) {$A''$}; \node[font=\small] at (\xright,-2) {$B''$};
\coordinate (s0) at (\xright,2.25); \coordinate (v1) at (\xleft,1.85); \coordinate (v2) at (\xright,1.3);
\coordinate (v3) at (\xleft,0.65); \coordinate (v4) at (\xright,0.05); \coordinate (s5) at (\xleft,-0.25);

\node[vertex] at (s0) {}; \node[vertex,label=left:$v_1$] at (v1) {}; 
\node[vertex,label=right:$v_2$] at (v2) {};
\node[vertex,label=left:$v_3$] at (v3) {}; 
\node[vertex,label=right:$v_4$] at (v4) {}; 
\node[vertex] at (s5) {};
\node[font=\small] at (0,2.35) {$S$};
\coordinate (v) at (\xleft,-1.5); \coordinate (u) at (\xright,-1.5);
\draw[added edge] (u) to[bend right=8] (v1); \draw[added edge] (u)--(v3);
\draw[added edge] (v) to[bend left=8] (v2); \draw[added edge] (v)--(v4);
\node[vertex,label=left:$v$] at (v) {}; \node[vertex,label=right:$u$] at (u) {};

\draw[line width=1.2pt] (s0)--(v1); 
\draw[deleted edge] (v1)--node[pos=.5,sloped,font=\small] {}(v2); 
\draw[line width=1.2pt] (v2)--(v3);
\draw[deleted edge] (v3)--node[pos=.5,sloped,font=\small] {}(v4); 
\draw[line width=1.2pt] (v4)--(s5);

\draw[added edge] (-3.2,-2.85)--(-2.45,-2.85); \node[anchor=west,font=\scriptsize] at (-2.35,-2.85) {added};
\draw[deleted edge] (0.45,-2.85)--(1.2,-2.85); \node[anchor=west,font=\scriptsize] at (1.3,-2.85) {deleted};
\end{tikzpicture}
\par\smallskip
\textbf{b)}
\end{minipage}
\caption{\textbf{a)} Proof of Claim~\ref{claim:longest_cycle}, the paths between $S$ and $S'$ are represented by wiggly lines, and the final cycle used to obtain the contradiction is highlighted in blue. \textbf{b)} Proof of Claim~\ref{claim:finding_B}, the red edges of $S$ are deleted and the green dashed edges are added to obtain a longer cycle.}
\label{fig:extending_the_cycle}
\end{figure}

Let $B$ be the set of $7k$ vertices with at least $(1-0.2\eps)k$ neighbours in $A'$ produced by the previous claim, and let $A$ be the set of vertices in $A'$ with at least $k$ neighbours in $B$. By counting the nonedges between $A'\backslash A$ and $B$, we find that every vertex in $A'\backslash A$ has at least $6k$ nonedges, and so
\[
|A'\setminus A|\cdot 6k \leq |A'||B|-e(\Gamma[A',B])\leq 7k\cdot (0.2\eps k+k^{0.1}).
\]
The last inequality follows as any vertex of $B$ has at most $|A'|-(1-0.2\eps)k\leq0.2\eps k+k^{0.1}$ nonneighbours in $A'$. We conclude that $|A'\setminus A|\leq 0.5\eps k$. Thus, each vertex in $B$ has at least $(1-0.2\eps)k-0.5\eps k\geq (1-\eps)k+k^{0.1}$ neighbours in $A$. Removing up to $k^{0.1}$ vertices from $A$ to guarantee that $|A|< k$, we then get that $(A,B)$ is an $(\eps,k)$-almost-complete pair, a contradiction.
\end{proof}


\subsection{Proof of Lemma~\ref{lemma:dense patches in expanders}}\label{sec:largeexpandercase}

To prove Lemma~\ref{lemma:dense patches in expanders}, we will need a result from~\cite{liu2023solution} (Theorem~\ref{mainthm-old} below) in a modified form, Theorem~\ref{mainthm-new}. We will make this modification carefully in Appendix~\ref{appendix}, but here we will briefly discuss the changes needed to the proof. To state the result that we alter, we need the following definition.

\begin{definition} For any connected bipartite graph $H$ and $u,v\in V(H)$, let\label{pidefn}
$$
\pi(u,v,H)=\left\{\begin{array}{ll}
0 & \text{ if }u=v, \\
1 & \text{ if $u$ and $v$ are in different vertex classes in the (unique) bipartition of $H$},\\
2 & \text{ if $u$ and $v$ are in the same vertex class and $u\neq v$}.
\end{array}
\right.
$$
\end{definition}

\begin{theorem}[Theorem 2.7 in \cite{liu2023solution}]\label{mainthm-old} There exists $\eps_1>0$ such that, for each $0<\eps_2<1/5$, there exists $d_0=d_0(\eps_1,\eps_2)$ such that the following holds for each $n\geq d\geq d_0$. Suppose that $H$ is a bipartite $n$-vertex $(\eps_1,\eps_2 d)$-expander with $\de(H)\ge d$. Then, one of the following holds.
\stepcounter{propcounter}
  \begin{enumerate}[label = \textup{\textbf{\Alph{propcounter}\arabic{enumi}}}]
  \item $H$ contains a subdivision of the complete $\lfloor d/2\rfloor$-vertex graph in which each edge has been subdivided once (so that it becomes a path of length 2).
  \label{prop:oldEHthm3}
\item For each distinct $x,y\in V(H)$, and each $\ell\in [\log^{7}n,n/\log^{12}n]$ with $\pi(x,y,H)\equiv \ell\pmod 2$, $H$ contains an $x,y$-path with length $\ell$.
\label{prop:oldEHthm2}
  \end{enumerate}
\end{theorem}

In our Lemma~\ref{lemma:dense patches in expanders}, we have an $n$-vertex $(\eps_1,k/2)$-expander with no $(\eps,k)$-almost-complete pair in $\Gamma$ and such that $d(\Gamma)\geq (1-0.1\eps)2k$ and $\delta(\Gamma)\geq (1-0.1\eps)k$.
We wish to show that $\sum_{\ell\in \cC(\Gamma)} \frac{1}{\ell}> \sum_{\ell=2}^k\frac{1}{2\ell}$.
As we will have $n>k^{1.01}$, if \ref{prop:oldEHthm2} holds, then applying this with any edge $xy\in E(H)$ will give enough different cycle lengths to ensure that $\sum_{\ell\in \cC(\Gamma)} \frac{1}{\ell}> \sum_{\ell=2}^k\frac{1}{2\ell}$.
However, the subdivision in \ref{prop:oldEHthm3} does not guarantee enough cycle lengths for this unless $\lfloor d/2\rfloor\geq k+1$. However, our graph $\Gamma$ is not necessarily bipartite and we have only $\delta(\Gamma)\geq (1-0.1\eps)k$.

Working with bipartite graphs in \cite{liu2023solution} is essentially only a convenience, and we will be able to remove this condition without much modification. As the function $\pi(x,y,H)$ used in \ref{prop:oldEHthm2} requires $H$ to be bipartite, this requires a little change (see \ref{prop:newEHthm2}), but, applied to any edge $xy\in E(\Gamma)$, will still produce enough different cycles. Still, even when the bounds in the proof in \cite{liu2023solution} are tightened, when applied to $\Gamma$, this would only allow us to find a subdivision of the complete graph with up to $(1-0.1\eps)k$ vertices in which each edge has been subdivided once. Such a subgraph only contains even cycles of length up to $(1-0.1\eps)2k$, which is not enough to show  $\sum_{\ell\in \cC(\Gamma)} \frac{1}{\ell}> \sum_{\ell=2}^k\frac{1}{2\ell}$. Therefore, we will record the intermediate structure used in \cite{liu2023solution} to find the subdivision for \ref{prop:oldEHthm3} (see~\ref{prop:newEHthm3}), before using this further in our proof in conjunction with our assumption that $\Gamma$ contains no  $(\eps,k)$-almost-complete pair.

Our modified Theorem~\ref{mainthm-old} is the following, as proved in Appendix~\ref{appendix}.

\begin{restatable}{theorem}{mainthm}
\label{mainthm-new} For every sufficiently small $\eps_1>0$ and each $0<\eps_2<1/5$, there exists $d_0=d_0(\eps_1,\eps_2)$ such that the following holds for each $n\geq d\geq d_0$. If $\Gamma$ is an $n$-vertex $(\eps_1,\eps_2 d)$-expander with $d(\Gamma)\ge 2d$ and $\delta(\Gamma)\geq d$, then one of the following holds.

\stepcounter{propcounter}
  \begin{enumerate}[label = \textup{\textbf{C\arabic{enumi}}}]
  \item There are disjoint sets $U,V\subset V(\Gamma)$ with $|U|\leq \log^{20}n$, $|V|\geq \log^{80}n$, and $d_{\Gamma}(v,U)\geq (1-2\eps_2)d$ for each $v\in V$.\label{prop:newEHthm3}
\item For each distinct $x,y\in V(\Gamma)$, and each $\ell\in [\log^{7}n,n/(4\log^{12}n)]$, $\Gamma$ contains an $x,y$-path which has length $2\ell$ or $2\ell+1$.
\label{prop:newEHthm2}
  \end{enumerate}
\end{restatable}

Using Theorem~\ref{mainthm-new}, we can now prove Lemma~\ref{lemma:dense patches in expanders}.

\begin{proof}[Proof of Lemma~\ref{lemma:dense patches in expanders}]  We have $1/k\ll \eps,\eps_1 \ll 1$ and $n> k^{1.01}$,
while $\Gamma$ is an $n$-vertex $(\eps_1,k/2)$-expander with no pair which is $(\eps,k)$-almost-complete in $\Gamma$ and such that $d(\Gamma)\geq (1-0.1\eps)2k$ and $\delta(\Gamma)\geq (1-0.1\eps)k$. Assume, for contradiction, that $\sum_{\ell\in \mathcal{C}(\Gamma)}\frac{1}{\ell}\leq \sum_{\ell=2}^k\frac{1}{2\ell}$.

Let $d=(1-0.1\eps)k$ and $\eps_2=\eps/60$, so that $\eps_2d\leq k/2$. As $\delta(\Gamma)\geq (1-0.1\eps)k$, we have that $\Gamma$ is an $(\eps_1,\eps_2d)$-expander. By Theorem~\ref{mainthm-new} applied to $\Gamma$, we have that either \ref{prop:newEHthm3} or \ref{prop:newEHthm2} holds. 

If \ref{prop:newEHthm2} holds, then, let $xy$ be any adjacent pair in $\Gamma$. Setting $\ell_0=\log^{7}n+1$ and $\ell_1=n/(4\log^{12}n)+1$, we have that $\Gamma$ contains a cycle of length $2\ell-1$ or $2\ell$ for each $\ell_0\leq \ell \leq \ell_1$. As $n> k^{1.01}$ and $1/k\ll 1$ implies $\ell_1/\ell_0\geq k^{1.005}$, we thus have
\[
\sum_{\ell\in \cC(\Gamma)} \frac{1}{\ell}\geq \sum_{\ell=\ell_0}^{\ell_1}\frac{1}{2\ell}
\ge \frac{1}{2} \log\left(\frac{\ell_1}{\ell_0}\right)\ge \frac{1}{2} \log k^{1.005}> \sum_{\ell=2}^k\frac{1}{2\ell},
\]
a contradiction.

Thus, we must have that \ref{prop:newEHthm3} holds.  Take $U,V$ as in \ref{prop:newEHthm3}. Note we have $d_{\Gamma}(u,U)\ge (1-0.1\eps)d\ge (1-0.2\eps)k$ for all $u\in V$. Pick a collection of disjoint sets $V_{xy}$ in $V$, one for each pair $\{x, y\}\in U^{(2)}$, such that $|V_{xy}|\leq 1$ and $V_{xy}\subset N_\Gamma(x)\cap N_{\Gamma}(y)$, which maximise $\sum_{\{x, y\}\in U^{(2)}}|V_{xy}|$. In other words, for each pair $\{x, y\}\in U^{(2)}$ we try to pick a different common neighbour, so that we get as many common neighbours as we can. As $|\bigcup_{\{x, y\}\in U^{(2)}}V_{xy}|\leq \binom{|U|}{2}<|V|/2$, setting $V'=V\setminus(\bigcup_{\{x, y\}\in U^{(2)}}V_{xy})$ we have $|V'|\geq |V|/2$. 

\begin{claim}\label{claim:prettydisjointorallin} For every pair of distinct vertices $v,v'\in V'$ with $|N_{\Gamma}(v,U)\cap N_{\Gamma}(v',U)|\geq 2$, we have $|N_{\Gamma}(v,U)\cup N_{\Gamma}(v',U)|\leq k$.
\end{claim}
\claimproofstart[Proof of Claim~\ref{claim:prettydisjointorallin}] 

Suppose otherwise, and let $v,v'\in V'$ be distinct and such that $|N_{\Gamma}(v,U)\cap N_{\Gamma}(v',U)|\geq 2$ and $|N_{\Gamma}(v,U)\cup N_{\Gamma}(v',U)|\geq k+1$.

We will show that $\Gamma$ contains a cycle of length $2\ell$, for every $2\leq \ell\leq k+1$. As $|N_{\Gamma}(v,U)\cap N_{\Gamma}(v',U)|\geq 2$, it is clear that a cycle of length $4$ exists in $\Gamma$. If $\ell\geq 3$, take distinct vertices $b_1,\ldots,b_\ell$ such that, setting $b_{\ell+1}=b_1$, for each $i\in [\ell]$, we have $b_i,b_{i+1}\in N_\Gamma(v)$ or $b_i,b_{i+1}\in N_\Gamma(v')$. This is possible since $|N_{\Gamma}(v,U)\cap N_{\Gamma}(v',U)|\geq 2$ and $|N_{\Gamma}(v,U)\cup N_{\Gamma}(v',U)|\geq k+1$.

By the maximality of $\sum_{\{x, y\}\in U^{(2)}}|V_{xy}|$, for each $i\in[\ell]$ we have that $V_{b_ib_{i+1}}\neq \varnothing$ (for otherwise we could add $v$ or $v'$ to $V_{b_ib_{i+1}}$). Thus, we can choose distinct vertices $a_i\in V_{b_ib_{i+1}}$, $i\in [\ell]$. Using that $a_i\in N_\Gamma(b_i)\cap N_{\Gamma}(b_{i+1})$ for each $i\in [\ell]$, we have that $b_1a_1\dots b_\ell a_\ell b_1$ is a cycle of length $2\ell$ in $\Gamma$. See Figure~\ref{fig:prettydisjointorallin} for an illustration.

Therefore, $\sum_{\ell\in \cC(\Gamma)} \frac{1}{\ell}\geq \sum_{\ell=2}^{k+1}\frac{1}{2\ell}> \sum_{\ell=2}^k\frac{1}{2\ell}$, a contradiction.
\claimproofend

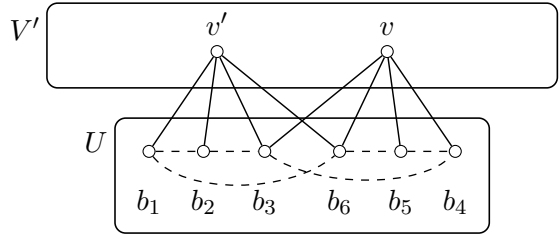
\begin{figure}
  \centering
  \begin{tikzpicture}[
    x=0.9cm,
    y=0.8cm,
    point/.style={
        circle,
        draw=black,
        fill=white,
        inner sep=0,
        minimum size=4.5pt
    },
    every edge/.style={
        draw=black,
        line width=0.55pt
    }
]
    \draw[rounded corners=4pt,line width=0.6pt]
        (-3.75,1.05) rectangle (3.75,2.45);
    \node[anchor=north west] at (-4.45,2.38) {$V'$};

    \draw[rounded corners=4pt,line width=0.6pt]
        (-2.75,-1.35) rectangle (2.75,0.55);
    \node[anchor=north west] at (-3.35,0.48) {$U$};

    \node[point,label=above:$v'$] (vp) at (-1.25,1.65) {};
    \node[point,label=above:$v$]  (v)  at ( 1.25,1.65) {};

    \node[point, label={[label distance=8pt]below:$b_1$}] (u1) at (-2.25,0) {};
    \node[point, label={[label distance=8pt]below:$b_2$}] (u2) at (-1.45,0) {};
    \node[point, label={[label distance=8pt]below:$b_3$}] (u3) at (-0.55,0) {};
    \node[point, label={[label distance=8pt]below:$b_6$}] (u4) at ( 0.55,0) {};
    \node[point, label={[label distance=8pt]below:$b_5$}] (u5) at ( 1.45,0) {};
    \node[point, label={[label distance=8pt]below:$b_4$}] (u6) at ( 2.25,0) {};

    \path
        (vp) edge (u1)
             edge (u2)
             edge (u3)
             edge (u4);

    \path
        (v) edge (u3)
            edge (u4)
            edge (u5)
            edge (u6);

    \draw[dashed,line width=0.5pt] (u1) -- (u2) -- (u3);
    \draw[dashed,line width=0.5pt] (u4) -- (u5) -- (u6);

    \draw[dashed,line width=0.5pt]
        (u1)
        .. controls (-1.75,-0.72) and (-0.20,-0.72) ..
        (u4);

    \draw[dashed,line width=0.5pt]
        (u3)
        .. controls (0.20,-0.61) and (1.75,-0.61) ..
        (u6);
\end{tikzpicture}
\caption{Proof of Claim~\ref{claim:prettydisjointorallin}. Each pair connected with a dashed line has a different common neighbour, which allows for constructing a cycle through all of them.}
\label{fig:prettydisjointorallin}
\end{figure}

By the property of $U,V$ from \ref{prop:newEHthm3}, we certainly have that, for each $v\in V'$ there are distinct vertices $u,u'\in N_\Gamma(v,U)$. Therefore, we can take distinct vertices $u,u'\in U$ for which there is a set $B\subset V'$ with $|B|\geq 2|V'|/|U|^2$ and $u,u'\in N_\Gamma(v)$ for each $v\in B$. Note that the property from \ref{prop:newEHthm3} implies that  $(1-2\eps_2)\cdot (1-0.1\eps)k\leq |U|\leq \log^{20}n$, and, hence, $k/2\leq \log^{20}n$. Thus $|B|\geq 2|V'|/|U|^2\geq \log^{40}n\geq (k/2)^2\geq 7k$.

For any $b\in B$, we have $d_\Gamma(b, U)\geq (1-2\eps_2)\cdot (1-0.1\eps)k\geq (1-0.2\eps)k$. So, if we fix some $b\in B$ and set $A'=N_\Gamma(b,U)$, by Claim~\ref{claim:prettydisjointorallin}, we have $(1-0.2\eps)k\leq |A'|\leq k$. For any other vertex $b'\in B$, we have 
\[|N_{\Gamma}(b,U)\cap N_{\Gamma}(b',U)|\geq |N_{\Gamma}(b',U)|+|N_\Gamma(b,U)|-|N_{\Gamma}(b',U)\cup N_\Gamma(b,U)|\geq 2(1-0.2\eps)k-k\geq (1-0.4\eps)k,\]
and so $d_\Gamma(b', A')\geq (1-0.4\eps)k$. Let $A$ be the set of vertices in $A'$ with at least $k$ neighbours in $B$. Then, by counting the nonedges between $A'$ and $B$, similarly as in the proof of Lemma~\ref{lem:smalldensenoalmostcomplete}
\[
|A'\setminus A|\cdot (|B|-k) \leq |A'||B|-e(\Gamma[A',B])\leq |B|\cdot 0.4\eps k,
\]
so that $|A'\setminus A|\leq \left(1+\frac{k}{|B|-k}\right)\cdot 0.4\eps k\leq 0.5\eps k$, where we have used that $|B|\geq 7k$. Thus, each vertex in $B$ has at least $(1-\eps)k+1$ neighbours in $A$. Removing at most one vertex from $A$ to guarantee that $|A|< k$, we then get that $(A,B)$ is an $(\eps,k)$-almost-complete pair, a contradiction.
\end{proof}


\subsection{Proof of Lemma~\ref{lemma:many edges outside clusters}}
We now prove Lemma~\ref{lemma:many edges outside clusters}, which shows that if our maximal collection of edge-disjoint clusters leaves many vertices in $G$ uncovered, then there is a dense subgraph of $G$ which is edge-disjoint from all the clusters.

\begin{proof}[Proof of Lemma~\ref{lemma:many edges outside clusters}] If $\cH=\emptyset$, then we may take $\Gamma=G$ as Lemma~\ref{lem:few_vertices} implies $n\geq k^{1.02}$, so that, as $1/k\ll \eps$,
\[
d(\Gamma)\geq \frac{2}{n}\cdot (k-1)(n-k+1)= \left(1-\frac{1}{k}\right)\cdot \left(1-\frac{k-1}{n}\right)\cdot 2k\geq \left(1-0.005\eps\right)^22k\geq (1-0.01\eps)2k.
\]

Thus, we may assume $\cH\neq\emptyset$. Without loss of generality, we can assume that no union of two $(\eps,k)$-clusters in $\cH$ is itself a $(\eps,k)$-cluster (for otherwise uniting them will give a family of clusters with the same $\sum_{H\in \cH}e(H)$ and smaller $|\mathcal H|$). 
Let $\mathcal{P}$ be a maximal collection of vertex-disjoint paths in $G$ which each have length at most 2 and such that, for each $P\in \cP$ with endvertices $x_P$ and $y_P$, there are distinct clusters $H_{P,1},H_{P,2}\in \cH$ with $x_P\in V(H_{P,1})$, $y_P\in V(H_{P,2})$, and $V(P)\setminus \{x_P,y_P\}\subset V(G)\setminus \bigcup_{H\in \cH}V(H)$. 

\begin{claim}\label{claim:few connecting paths}
$|\cP|<|\cH|$.
\end{claim}
\claimproofstart[Proof of Claim~\ref{claim:few connecting paths}] Let $L$ be the auxiliary multigraph with vertex set $\cH$ and an edge $H_{P,1}H_{P,2}$ for each $P\in \cP$. By Lemma~\ref{lem:cluster properties}\,\ref{prop:cluster:nocycle}, as $\sum_{\ell\in \cC(G)}\frac{1}{\ell}\leq \sum_{\ell=2}^k\frac{1}{2\ell}$, $L$ has no cycles, for otherwise $G$ would have a corresponding cluster cycle. Thus, $|\cP|=e(L)<|L|=|\cH|$, as required.
\end{proof}

Let $X=\bigcup_{H\in \cH} (V(H)\setminus V(\cP))$, $Y=V(G)\backslash X$, and $\Gamma=G[Y]-\bigcup_{H\in \cH}E(H)$. 
Each vertex $y\in Y\backslash V(\cP)\subseteq V(G)\setminus \bigcup_{H\in \cH}V(H)$ has at most $1$ neighbour in each cluster $H\in\cH$ (otherwise, $G[V(H)\cup \{y\}]$ is a cluster which contradicts the maximality of $\cH$) and, furthermore, it has a neighbour in $V(H)\setminus V(\mathcal{P})$
for at most one cluster $H\in \cH$  (by the maximality of $\cP$). Thus, each vertex of $Y\backslash V(\cP)$ has at most one neighbour in $X$ in $G$. Hence,
\begin{equation}\label{eq:eGXYnoP}
e(G[X,Y\backslash V(\cP)])\leq |Y\backslash V(\cP)|.
\end{equation}

By \eqref{eq:doesnt mostly cover V(G)} and Claim~\ref{claim:few connecting paths}, we have that $|Y|> 2000|\cH|/\eps$ and thus $|X\cup V(\cP)|\leq n-|Y|+|V(\cP)|\leq n+3|\cH|-|Y|<n$.
If there is some $H\in \mathcal{H}$, then $|X\cup V(\cP)|\geq |H|\geq (3-\eps)k\geq 2k$.
Furthermore, $\sum_{\ell\in \cC(G[X\cup V(\cP)])}\frac{1}{\ell}\leq \sum_{\ell\in \cC(G)}\frac{1}{\ell}$. Thus, if \textbf{a)} holds in the statement of the lemma (i.e., if $\sum_{\ell\in \cC(G)}\frac{1}{\ell}<\sum_{\ell=2}^{k}\frac{1}{2\ell}$), then, by the minimality of $n$,  $e(G[X\cup V(\cP)])\leq (k-1)(|X\cup V(\cP)|-k+1)$ and, hence,
 \begin{align}\label{eq:fromminimal:1}
e(\Gamma)&\geq e(G)-e(G[X\cup V(\cP)])-e(G[X,Y\backslash V(\cP)])\nonumber \\
&\geq (k-1)(n-k+1) - (k-1) (|X\cup V(\cP)|-k+1) - e(G[X,Y\backslash V(\cP)])\nonumber\\
&= (k-1) (n-|X\cup V(\cP)|) - e(G[X,Y\backslash V(\cP)]).
 \end{align}
On the other hand, if \textbf{b)} holds in the statement of the lemma, then we have $e(G)\geq k(n-k)$ and, by the minimality of $n$, that $e(G[X\cup V(\cP)])\leq k(|X\cup V(\cP)|-k)$, so that
\begin{align}\label{eq:fromminimal:2}
e(\Gamma)&\geq e(G)-e(G[X\cup V(\cP)])-e(G[X,Y\backslash V(\cP)])\nonumber\\
&\geq k(n-k) - k(|X\cup V(\cP)|-k) - e(G[X,Y\backslash V(\cP)])\nonumber\\
&= k(n-|X\cup V(\cP)|) - e(G[X,Y\backslash V(\cP)]).
\end{align}
Therefore, whether \textbf{a)} or \textbf{b)} holds in the statement of the lemma,
\begin{align}
e(\Gamma)&\overset{\eqref{eq:fromminimal:1},\eqref{eq:fromminimal:2}}{\geq}(k-1) (n-|X\cup V(\cP)|) - e(G[X,Y\backslash V(\cP)])\\
&\;\;\overset{\eqref{eq:eGXYnoP}}{\geq} (k-1) (n-|X\cup V(\cP)|) -  |Y\backslash V(\cP)|\nonumber\\
&\;\;= (k-1) |Y\backslash V(\cP)|-|Y\backslash V(\cP)|= (k-2)|Y\backslash V(\cP)|\geq (k-2)(|Y|-3|\cH|),\label{eq:Egammabound}
\end{align}
where we have used that $|V(P)\cap Y|\leq 3$ for each $P\in\cP$ and hence $|V(\cP)\cap Y|\leq 3|\cH|$. Recalling that $|Y|\geq 2000|\cH|/\eps$, we have
\[
d(\Gamma)\overset{\eqref{eq:Egammabound}}> \frac{2(k-2)(|Y|-3|\cH|)}{|\Gamma|}= \left(1-\frac{2}{k}\right)\cdot \left(1-\frac{3|\cH|}{|Y|}\right)\cdot 2k\geq \left(1-0.005\eps\right)^22k\geq (1-0.01\eps)2k,\]
where we have also used that $1/k\ll \eps$. Thus, $\Gamma$ satisfies our requirements.
\end{proof}


\section{Proof of Theorem~\ref{thm:main}}\label{sec:main_thm_proof}

In this section, we prove our main theorem, Theorem~\ref{thm:main}. We divide this proof into two parts, depending on how many edge-disjoint clusters are in our maximal collection. We prove the case where there is a single cluster as Lemma~\ref{lemma:t=1} in Section~\ref{sec:one_cluster}. We prove the case where there are at least two clusters in our maximal collection as Lemma~\ref{lem:not_many_clusters} in Section~\ref{sec:case:t>1}. We put this together to prove Theorem~\ref{thm:main} in Section~\ref{sec:lastproofbit}.


\subsection{The case of a single cluster}\label{sec:one_cluster}
We will now rule out the case where our maximal edge-disjoint collection of clusters in a minimal counterexample $G$ contains exactly one cluster. If there were exactly one cluster, then the cluster contains an almost-complete bipartite graph on vertex classes $A$ and $B$, and within $G[A,B]$ we can find most of the canonical short even cycles for comparison to $\mathcal{C}(K_{k,n-k})$. The restriction is the size of $A$ -- it may not be large enough to contain even cycles with length up to $2k$. When this happens, vertices outside of $A\cup B$ with two neighbours in $B$ could be used to create longer paths. More generally, if we can find many disjoint cherries resembling this (see Figure~\ref{fig:path_through_cherries}) then we will be able to use them to find longer cycles (see Claim~\ref{claim:cycle through cherries}). Where there are not enough such cherries, then we will be able to show there are enough edges in $G$ away from $A$, $B$, and a maximal disjoint collection of such cherries, to find a cycle with a suitable length (see Claim~\ref{clm:goodcycle}) that we can connect into $G[A,B]$ to find an interval of even, or odd, cycle lengths at a new scale, to then contradict that we have a minimal counterexample.

\begin{lemma}\label{lemma:t=1}
Let $1/k_0\ll \eps \ll 1$. Let $n$ be the least integer for which there is an integer $k_0\leq k\leq n/2$ and an $n$-vertex graph $G$ with more than $(k-1)(n-k+1)$ edges for which either \textup{\textbf{a)}} $\sum_{\ell\in \cC(G)}\frac{1}{\ell}<\sum_{\ell=2}^k\frac{1}{2\ell}$ or \textup{\textbf{b)}} $\sum_{\ell\in \cC(G)}\frac{1}{\ell}= \sum_{\ell=2}^k\frac{1}{2\ell}$, $e(G)\geq k(n-k)$ and $G$ is not a copy of $K_{k,n-k}$.
Let $G$ be such a graph and let $\mathcal{H}$ be a collection of edge-disjoint $(\eps,k)$-clusters in $G$ maximising $\sum_{H\in \cH}e(H)$. Then, $|\cH|\neq 1$.
\end{lemma}

\begin{proof} Suppose for contradiction that $\mathcal{H}$ has only one cluster, denoted by $H$, and let $(A, B)$ be the $(\eps, k)$-almost-complete pair in $H$. Let $B'=\{v\in V(G)\setminus A: |N(v)\cap A|\ge 3\eps k\}$, noting $B\subseteq B'$. Let $x_1y_1z_1, \dots, x_ty_tz_t$ be a maximal collection of disjoint cherries with leaves $x_i, z_i$ in $B'$ and with centers $y_i$ outside of $A\cup B'$ (see Figure~\ref{fig:path_through_cherries}). Let $Y=\{y_1, \dots, y_t\}$ and $A'=A\cup \{y_1, \dots, y_t\}$. Let $\Gamma=G-(A'\cup B')$.
\begin{claim}\label{claim:cycle through cherries}
    For every  $q =0, \dots, \min(1.1\eps k,t)$, there is a cycle of length $2(|A|+q)$.
\end{claim}
\noindent
\begin{minipage}[t]{0.56\textwidth}
\vspace{0pt}
\claimproofstart
The case $q=0$ follows from Corollary~\ref{cor:cycles in almost complete pairs}, so suppose $q\geq 1$. Since $2q\leq 2.2\eps k+2<3\eps k$, we may choose distinct vertices $a_i,b_i\in A$, $i\in[q]$, such that $a_ix_i,z_ib_i\in E(G)$. Using Lemma~\ref{lemma:path lengths in almost complete pairs} greedily, join $b_i$ to $a_{i+1}$, for each $i\in[q-1]$, by pairwise internally vertex-disjoint paths of length four in $G[A,B]$, avoiding all the other vertices already chosen. This is possible since throughout fewer than $k/10$ vertices are forbidden. Together with the cherries, these paths form an $a_1,b_q$-path $R_q$ of length $8q-4$. 
Let $F=(V(R_q)\cap(A\cup B))\setminus\{a_1,b_q\}$.
Then $|F|\leq 7q\leq k/10$ and $|A\setminus F|=|A|-3q+3$. Applying Lemma~\ref{lemma:path lengths in almost complete pairs} once more gives a path of length $2(|A|-3q+2)$  between $a_1$ and $b_q$ disjoint from $F$.  Joining this to $R_q$ gives a cycle of length $8q-4+2(|A|-3q+2)=2(|A|+q)$ as required. 
\claimproofend
\end{minipage}\hfill
\begin{minipage}[t]{0.41\textwidth}
\vspace{0pt}
\centering
\resizebox{\linewidth}{!}{
\begin{tikzpicture}[
    x=0.9cm,
    y=0.8cm,
    point/.style={
        circle,
        draw=black,
        fill=white,
        inner sep=0,
        minimum size=4.5pt
    },
    edge/.style={
        draw=black,
        line width=0.8pt
    },
    setbox/.style={
        rounded corners=4pt,
        line width=0.6pt
    },
    scale=1.2
]
    \draw[setbox] (-3.35,0.65) rectangle (1.35,1.75);
    \node[anchor=east] at (-3.45,1.2) {$B'$};

\def\shiftup{0.35cm}

\begin{scope}[shift={(0,\shiftup)}]
    \draw[setbox] (-3.05,-1.35) rectangle (1.05,-0.255);
    \node[anchor=east] at (-3.15,-0.6) {$A$};
\end{scope}

    \node[point,label=above:$y_1$] (y1) at (-2.5,2.25) {};
    \node[point,label=above:$y_2$] (y2) at (-1,2.25) {};
    \node[point,label=above:$y_3$] (y3) at (0.5,2.25) {};

    \node[point,label={[label distance=-3pt]left:$x_1$}] (x1) at (-2.75,1.2) {};
    \node[point,label={[label distance=-3pt]right:$z_1$}] (z1) at (-2.25,1.2) {};

    \node[point,label={[label distance=-3pt]left:$x_2$}] (x2) at (-1.25,1.2) {};
    \node[point,label={[label distance=-3pt]right:$z_2$}] (z2) at (-0.75,1.2) {};

    \node[point,label={[label distance=-3pt]left:$x_3$}] (x3) at (0.25,1.2) {};
    \node[point,label={[label distance=-3pt]right:$z_3$}] (z3) at (0.75,1.2) {};

\begin{scope}[shift={(0,\shiftup)}]
    \node[point,label={[label distance=2pt]below:$a_1$}] (a1) at (-2.75,-0.6) {};
    \node[point,label={below:$b_1$}] (b1) at (-2.25,-0.6) {};

    \node[point,label={[label distance=2pt]below:$a_2$}] (a2) at (-1.25,-0.6) {};
    \node[point,label={ below:$b_2$}] (b2) at (-0.75,-0.6) {};

    \node[point,label={[label distance=2pt]below:$a_3$}] (a3) at (0.25,-0.6) {};
    \node[point,label={below:$b_3$}] (b3) at (0.75,-0.6) {};
\end{scope}

    \coordinate (p11) at (-2,0.9) {};

\begin{scope}[shift={(0,\shiftup)}]
        \coordinate (p12) at (-1.75,-0.5) {};
    \end{scope}
    \coordinate (p13) at (-1.5,0.9) {};

    \coordinate (p21) at (-0.5,0.9) {};
\begin{scope}[shift={(0,\shiftup)}]
 
    \coordinate (p22) at (-0.25,-0.5) {};
\end{scope}
    \coordinate (p23) at (0,0.9) {};

    \draw[edge] (x1)--(y1)--(z1);
    \draw[edge] (x2)--(y2)--(z2);
    \draw[edge] (x3)--(y3)--(z3);

    \draw[edge] (x1)--(a1);
    \draw[edge] (z1)--(b1);
    \draw[edge] (x2)--(a2);
    \draw[edge] (z2)--(b2);
    \draw[edge] (x3)--(a3);
    \draw[edge] (z3)--(b3);

    \draw (b1)--(p11)--(p12)--(p13)--(a2);
    \draw (b2)--(p21)--(p22)--(p23)--(a3);
\end{tikzpicture}
}

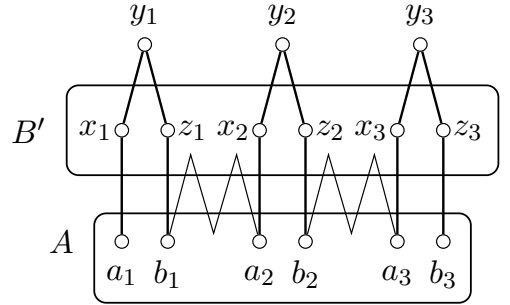
\captionof{figure}{Constructing the path $R_q$ in the proof of Claim~\ref{claim:cycle through cherries}.}
\label{fig:path_through_cherries}
\end{minipage}

Now, if $|A'|\geq k+1$, then by Claim~\ref{claim:cycle through cherries} we have all the even cycle lengths in $[4,2k+2]$, contradicting $\sum_{\ell\in \cC(G)}\frac{1}{\ell}\le \sum_{\ell=2}^k\frac{1}{2\ell}$. If $|A'|=k$, then by Claim~\ref{claim:cycle through cherries} we have all the even cycle lengths in $[4,2k]$, showing that \textbf{a)} does not hold. Therefore, we can assume that $|A'|\le k$, and that if $|A'|= k$, then $e(G) \ge k(n-k)$. Since $|A'|\le k$, we have $|Y|=t\le \eps k$. 

We will now show that this implies $\Gamma$ has a cycle with a useful length, as follows.

\begin{claim} $\Gamma$ contains a cycle whose length is between $1.5k$ and $2k$.\label{clm:goodcycle}
\end{claim}
\claimproofstart[Proof of Claim~\ref{clm:goodcycle}.]
Note that vertices $v\in V(\Gamma)$ can have at most $3\eps k$ neighbours in $A$ (as $v\notin B'$), at most $|Y|\le \eps k$ neighbours in $Y$, and at most $4\eps k$ neighbours in $B'$ (otherwise two of these together with $v$ would give a new cherry disjoint from the others). Thus, $e_G(V(\Gamma), A'\cup B')\le 8\eps k |\Gamma|$ and $\delta(\Gamma)\geq (1-8\eps)k$.

Note that $G[A'\cup B']$ has diameter $\le 8$ --- vertices in $A$ have paths of length $4$ to each other by Lemma~\ref{lemma:path lengths in almost complete pairs}, while vertices in $B'$ and $A'$ have neighbours in $A$ and $B'$ respectively showing that all vertices are within distance $2$ of $A$.  
Now, note that if $G[A'\cup B']$ is non-bipartite, then  the shortest odd cycle $C$ must have length $\le 2\cdot \textrm{Diam}(G[A'\cup B'])+1\le 17$. When this occurs,
\begin{equation}\label{eq:usefultwice}
\sum_{\ell\in \cC(G)}\frac{1}{\ell}\geq \sum_{\ell=2}^{(1-\eps)k} \frac{1}{2\ell}+\frac{1}{17}
\geq \sum_{\ell=2}^{k} \frac{1}{2\ell}+\frac{1}{17}-(\eps k+1)\cdot \frac{1}{2(1-\eps )k}> \sum_{\ell=2}^k \frac{1}{2\ell}
\end{equation}
gives a contradiction. Thus, we have that $G[A'\cup B']$ is bipartite. We have that $|A'|\leq k$, and so, if $\Gamma=\emptyset$, then, as $e(G)>(k-1)(n-k+1)$, we have $|A'|=k$, and, as we deduced $e(G)\geq k(n-k)$ in this case we have that $G$ must be a copy of $K_{k,n-k}$, a contradiction. Thus, we can assume $\Gamma$ is non-empty.

We have that the parts of the bipartition of $G[A'\cup B']$ are $A', B'$ (vertices in $A$ are in the same part since they have length $4$ paths to each other, while vertices in $B'$ are in the opposite part since they have edges to $A$, and finally vertices in $A'$ are in the same part as $A$ since they have edges to $B'$).  Now, if  $|A'|\leq k-1$ then since $|B'|\geq |B|\geq (2-\eps)k$, we have $e(G[A'\cup B'])\leq (k-1)(|A'\cup B'|-k+1)$ and so 
\begin{align*}
e(\Gamma)&\ge e(G)-e(G[A'\cup B'])-e_G(V(\Gamma), A'\cup B')\\
&\ge (k-1)(|\Gamma|+|B'|+|A'|-k+1)-(k-1)(|A'|+|B'|-k+1)-8\eps k|\Gamma|\\
&= (k-1-8\eps k)|\Gamma|\geq (1-9\eps)k|\Gamma|,
\end{align*}
whereas if $|A'|=k$ then we have \textbf{b)} and so 
\begin{align*}
e(\Gamma)&\ge e(G)-e(G[A'\cup B'])-e_G(V(\Gamma), A'\cup B')\\
&\ge k(|\Gamma|+|B'|+|A'|-k)-k|B'|-8\eps k |\Gamma|\geq (1-9\eps)k|\Gamma|.
\end{align*}
Thus, in either case $d(\Gamma)\geq (1-9\eps)2k$.

Using Theorem~\ref{thm-expander}, Lemma~\ref{lem:smalldensenoalmostcomplete} and Lemma~\ref{lemma:dense patches in expanders}, $\Gamma$ contains an $(100\eps,k)$-almost-complete pair, and hence a cycle whose length is between $1.5k$ and $2k$.
\claimproofend

Using Claim~\ref{clm:goodcycle}, let $S$ be a cycle in $\Gamma$ with length between $1.5k$ and $2k$. 
By Lemma~\ref{lem:few uncovered vertices}, we have that $|V(G)\setminus V(H)|\le 2000/\eps<1.5k-2$ and hence $S$ has at least two vertices in $H$. Using Lemma~\ref{lem:cluster properties} \textbf{b)}, we can find vertex-disjoint paths $P_1, P_2$ of total length $\le 6\eps k$ from $S$ to $A\cup B$. By considering such paths with minimum total length, we additionally get $|P_i\cap (A\cup B)|= |P_i\cap S|=1$ for each $i$. We extend $P_1, P_2$ by at most one edge to obtain paths whose one endpoint is in $A$, and we concatenate them with the longer arc of $S$ between their endpoints. In this way we obtain a path, $P$ say, whose length $\ell$ is between $3k/4$ and $3k$ with internal vertices in $V(G)\setminus A$, such that $P$ starts and ends in $A$ and contains at most $2$ vertices of $B$. 

Finally, using Lemma~\ref{lemma:path lengths in almost complete pairs} with $F=V(P)\cap B$, we can close $P$ into a cycle with paths of any even length between $5k/4$ and  $2|A\backslash F|-2$. Beyond the cycle lengths contained in the almost-complete pair $(A, B)$, this gives us at least $k/4$ new cycle lengths somewhere in the interval $[2k, 5k]$. Thus, by a similar calculation as  \eqref{eq:usefultwice}, we have 
\[
\sum_{\ell\in \cC(G)}\frac{1}{\ell}\geq \sum_{\ell=2}^{(1-\eps)k} \frac{1}{2\ell}+\frac{k}{4}\cdot \frac{1}{5k}= \sum_{\ell=2}^{(1-\eps)k} \frac{1}{2\ell}+\frac{1}{20}> \sum_{\ell=2}^k \frac{1}{2\ell},
\]
a contradiction. Thus, $|\mathcal{H}|\neq 1$.
\end{proof}

\subsection{The case of many clusters}\label{sec:many_clusters}\label{sec:case:t>1}

In this section, we will show that the minimal counterexample $G$ to Theorem~\ref{thm:main} cannot have more than one cluster in the collection $\cH$. We begin by giving a high-level overview of the proof. From Lemma~\ref{lemma:minimum degree} \textbf{iii)}, we know that it is impossible to separate a single cluster from the rest of the graph by deleting fewer than $k$ vertices. Hence, by K\H{o}nig's theorem, for each $H\in \cH$, there is a matching $M_H$ of size $k$ between $H$ and the remaining vertices. 

The main idea of the proof will be to construct an auxiliary digraph $D$, whose vertices will be the clusters of $\cH$ and the remaining vertices of $G$ not included in any of the clusters, that is $V_0=V(G)\backslash \bigcup_{H\in \cH} V(H)$. From each edge in the matching $M_H$, we add an edge to $D$, directed from $H$ to either another cluster $H'$ where this edge ends, or to the vertex in $V_0$ where it ends. The resulting digraph has large average degree, and hence it contains an antidirected cycle. By examining this antidirected cycle closely, we will be able to turn it into a cluster cycle, in the sense of the Definition~\ref{def:cluster cycle}, which is impossible due to Lemma~\ref{lem:cluster properties} \textbf{c)}. However, we must do this carefully, in order to avoid unwanted vertex overlaps, and thus we perform a cleaning procedure on the matchings $M_H$ before defining $D$. Let us now give the details.

\begin{lemma}\label{lem:not_many_clusters}
Let $1/k_0\ll \eps \ll 1$. Let $n$ be the least integer for which there is an integer $k_0\leq k\leq n/2$ and an $n$-vertex graph $G$ with more than $(k-1)(n-k+1)$ edges for which either \textup{\textbf{a)}} $\sum_{\ell\in \cC(G)}\frac{1}{\ell}<\sum_{\ell=2}^k\frac{1}{2\ell}$ or \textup{\textbf{b)}} $\sum_{\ell\in \cC(G)}\frac{1}{\ell}= \sum_{\ell=2}^k\frac{1}{2\ell}$, $e(G)\geq k(n-k)$ and $G$ is not a copy of $K_{k,n-k}$. Let $G$ be such a graph and let $\mathcal{H}$ be a collection of edge-disjoint $(\eps,k)$-clusters in $G$ maximising $\sum_{H\in \cH}e(H)$ and, subject to that, minimising $|\mathcal H|$. Then, $|\cH|\leq 1$.
\end{lemma}
\begin{proof} Suppose, for contradiction, that $|\cH|\geq 2$. Without loss of generality, we can assume that no union of two $(\eps,k)$-clusters in $\cH$ is itself a $(\eps,k)$-cluster (otherwise uniting them will give a family of clusters with the same $\sum_{H\in \cH}e(H)$ and smaller $|\mathcal H|$). We start with the following claim.

\begin{claim}\label{clm:matchings} 
For each $H\in \cH$, there is a matching $M_H$ of size $k$ in $G$ between $V(H)$ and $V(G)\backslash V(H)$.
\end{claim}
\claimproofstart[Proof of Claim~\ref{clm:matchings}.] Let $H\in \cH$. Suppose, for contradiction, that no such matching $M_H$ exists. Let $G'$ be the bipartite subgraph of $G$ formed by the edges of $G$ between $V(H)$ and $V(G)\backslash V(H)$.
By K\"onig's theorem, the size of a largest matching in $G'$ is equal to the size of a smallest vertex cover, and therefore there is a vertex cover $S$ of $G'$ with size at most $k-1$.

Now, as $H$ is a cluster, and so contains an $(\eps,k)$-almost-complete pair, we have $|V(H)|\geq |A(H)|+|B(H)|\geq (1-\eps)k+2k\geq 2k+2$. As $|\cH|\geq 2$, we can choose a cluster $H'\in \cH\setminus \{H\}$. If $|V(H)\cap V(H')|\geq 2$, then, picking two shared vertices of $H$ and $H'$ as paths of length 0, with $H$ and $H'$ we have a cluster cycle, contradicting Lemma~\ref{lem:cluster properties}\,\ref{prop:cluster:nocycle}. Therefore, we have $|V(H)\cap V(H')|\leq 1$, so that $|V(G)\setminus V(H)|\geq |A(H')|+|B(H')|-1\geq (1-\eps)k+2k-1\geq 2k+2$. Letting $A=V(H)\setminus S$ and $B=V(G)\setminus (V(H)\cup S)$, we thus have a partition $V(G)=A\cup S\cup B$ such that $|A\cup S|,|B\cup S|\geq 2k+2$, $|S|\leq k-1$ and $S$ separates $A$ and $B$ in $G$. This contradicts Lemma~\ref{lemma:minimum degree} \textbf{iii)}, completing the proof of the claim.
\claimproofend

Let $V_0=V(G)\backslash\big(\bigcup_{H\in \cH} V(H)\big)$ and let $V_{\ge 2}$ be the set of vertices that are in at least two different clusters of $\mathcal H$. For each $v\in V(G)\setminus V_0$, fix some cluster $H(v)$ containing $v$. Note that for any cluster $H\in \mathcal H$ and $v\in V(H)\setminus V_{\ge 2}$, we have $H(v)=H$ (since otherwise $H$ and $H(v)$ would be distinct clusters containing $v$, contradicting $v\notin V_{\ge 2}$). 

Further, for each $H\in \cH$, let $M_H$ be a matching of size $k$ between $V(H)$ and $V(G)\setminus V(H)$, which exists due to Claim~\ref{clm:matchings}. We think of the edges of $M_H$ as being directed from $H$ to $V(G)\setminus V(H)$ (so if we have some edge $uv\in E(G)$ with $u\in H$, $v\in H'$ which is in both the matchings $M_H$ and $M_{H'}$, then it is directed $uv$ in $M_H$ and directed $vu$ in $M_{H'}$). For each $H$, let $M_H'$ be the submatching of $M_H$ consisting of edges $uv$ with $u\not\in V_{\ge 2}$.

Finally, let us define the auxiliary digraph $D$ on the vertex set $V(D)=\cH \cup V_0$. For every edge $uv\in \bigcup_{H\in \mathcal H} M_H'$  we define an edge $f_{uv}$ in $D$ as follows:
\begin{itemize}
    \item If $uv\in M_H'$ and $v\in V_0$, then $f_{uv}$ goes from $H$ to $v$.
    \item If $uv\in M_H'$ and $v\not\in V_0$, then $f_{uv}$ goes from $H$ to $H(v)$.
\end{itemize}
Note that, since each such directed $uv$ is in only one of the matchings $M_H'$, $H\in \cH$, the edge $f_{uv}$ is well defined. We now discuss the properties of $D$.

\begin{claim}\label{claim:proper_digraph}
The digraph $D$ is simple, i.e. it has no multi-edges.
\end{claim}
\claimproofstart
Suppose to the contrary that there are two different edges $f_{u_1v_1}$ and $f_{u_2v_2}$ in $D$ with the same start and endvertices. By definition there are no edges starting in $V_0$, so the start of $f_{u_1v_1}$ and $f_{u_2v_2}$ must be some $H\in \mathcal H$. We must have $u_1, u_2\in V(H)$ and $u_1v_1, u_2v_2\in M_H'$. Since $M_H'$ is a matching, we have that $u_1, u_2, v_1, v_2$ are distinct. 
If $v_1$ is in $V_0$, then $f_{u_1v_1}$ is an edge directed to $v_1$, and $f_{u_2v_2}$ cannot be directed to $v_1$ as $v_2\neq v_1$. This, and the symmetric argument, means that $v_1,v_2\notin V_0$, and so we must have $H(v_1)=H(v_2)$ as the endvertex of $f_{u_1v_1}$ and $f_{u_2v_2}$. Let $H=H(u_1)=H(u_2)$ and $H'=H(v_1)=H(v_2)$. Then, $H$ and $H'$ with the paths of length 1 given by $u_1v_1$ and $u_2v_2$ form a cluster cycle of length two, which is a contradiction by Lemma~\ref{lem:cluster properties}\,\ref{prop:cluster:nocycle}. Thus, there are no two edges $f_{u_1v_1}$ and $f_{u_2v_2}$ in $D$ with the same start and endvertices, i.e.\ $D$ is a simple digraph.
\claimproofend

\begin{claim}\label{claim:edge_count_D}
We have $e(D)\ge (k-2)|\mathcal H|$.
\end{claim}
\claimproofstart
Since $D$ is a simple digraph whose edges are in one-to-one correspondence with the edges of $\bigcup_{H\in \mathcal H} M_H'$, it suffices to show that $\sum_{H\in \mathcal H} e(M_H')\ge (k-2)|\mathcal H|$.

As $M_H'$ consists of directed edges $uv$ with $u\not\in V_{\ge 2}$, in the directed graph $\bigcup_{H\in \mathcal H} M_{H}'$, every vertex $u$ has outdegree $d^+(u)\le 1$ (if $uv$ and $uw$ are edges from $u$ in $\bigcup_{H\in \mathcal H} M_{H}'$ then these must be in different matchings $M_H'$ since each one is a matching --- say $uv \in M_{H_1}'$, $uw \in M_{H_2}'$. Since, for each $i\in[2]$, the matching $M_{H_i}$ is directed from $H_i$, this tells us that $u\in V(H_1)\cap V(H_2)$. Hence, $u\in V_{\ge 2}$, contradicting the choice of $M_{H_1}'$ or $M_{H_2}'$). For the same reason each directed edge $uv$ is in at most one matching $M_H'$, $H\in \mathcal H$. Hence, for each vertex of $V_\geq 2$ in $H$, at most one edge is lost in $M_H'$.

Let us now bound the number of pairs $(v, H)$ where $v\in V_{\geq 2}\cap H$. Let $F$ be an auxiliary bipartite graph with parts $\mathcal H$ and $V_{\ge 2}$ with $Hv$ an edge whenever $v\in V(H)$. Note that $e(F)=\sum_{H\in \cH}|V(H)\cap V_{\geq 2}|$. We claim that $F$ is a forest. Indeed, if $F$ had a cycle $H_1v_1H_2v_2\dots H_{\ell} v_{\ell} H_1$ then we have a cluster cycle with clusters $H_1, \dots, H_{\ell}$ and length $0$ paths $P_1=\{v_1\}, \dots, P_{\ell}=\{v_{\ell}\}$, contradicting Lemma~\ref{lem:cluster properties}\,\ref{prop:cluster:nocycle}. Thus $e(F)\le |V_{\ge 2}|+|\cH|-1$. 
Since all vertices in $V_{\ge 2}$ have degree $\ge 2$ in $F$, we get $2|V_{\ge 2}|\le e(F)$. Combining the two inequalities gives $e(F)\le  2|\cH|-2$.
So, we have 
\begin{equation*}
\sum_{H\in \mathcal H} e(M_H')\ge \sum_{H\in \mathcal H} (e(M_H)-|V(H)\cap V_{\ge 2}|)=\Big(\sum_{H\in \mathcal H} e(M_H)\Big)- e(F)\ge k|\mathcal H|- 2|\mathcal H|=(k-2)|\mathcal H|.\qedhere
\end{equation*}  
\claimproofend
 
\begin{claim}\label{claim:antidirected_cycle}
The digraph $D$ contains an antidirected cycle.  
\end{claim}
\claimproofstart
We find a partition $\cH=\cH_L\cup \cH_R$ and consider the spanning subdigraph $D'$ of $D$ with edges only those directed from $\cH_L$ to $\cH_R\cup V_0$. To do this, form $\cH_L\subset \cH$ by including each cluster in $\cH$ independently at random with probability $1/2$, and let $\cH_R=\cH\setminus \cH_L$. Let $D'\subset D$ have vertex set $V(D)=\cH\cup V_0$ and have all of the edges of $E(D)$ directed from $\cH_L$ to $V_0\cup \cH_R$. For each $e\in E(D)$, we have $\mathbb{P}(e\in E(D'))\geq (1/2)^2$, so we can take the partition $\cH=\cH_L\cup \cH_R$ so that $e(D')\geq e(D)/4\geq (k-2)|\cH|/4\ge k|\cH|/8$.

By Lemma~\ref{lem:few uncovered vertices} and $\eps\le 1$, we have
\begin{equation}\label{eq:Dsizebound}
|D'|=|D|=|\cH|+\bigg|V(G)\backslash\bigcup_{H\in \cH} V(H)\bigg|\le |\cH|+\frac{2000|\cH|}{\eps}\leq \frac{2001 |\cH|}{\eps}.
\end{equation}
Thus, $D'$ has average out-degree at least $(k|\cH|/8)\cdot (2001|\cH|/\eps)^{-1}\geq \eps k/10^5$, which is large as $1/k\ll \eps$. Therefore, the underlying undirected graph for $D'$ contains a cycle, and hence $D'$ contains an antidirected cycle. 
\claimproofend

Take the cycle from Claim~\ref{claim:antidirected_cycle} and let its vertices be $L_1R_1L_2R_2\dots L_{\ell} R_{\ell} L_1$ (with $L_1, \dots, L_{\ell}\in \cH_L$ and $R_1, \dots, R_{\ell}\in \cH_R\cup V_0$). 
For each $i\in [\ell]$, we will now define short paths, which are either $P_i$, or $P_i^-$ and $P_i^+$. Essentially, when there is a short path in $G$ from $L_i$ to $L_{i+1}$ through $R_i$ then this path will be $P_i$ and, where there is not, $P_i^-$ will be a short path in $G$ from  $L_i$ to $R_i$ and  $P_i^+$ will be a short path in $G$ from $R_i$ to $L_{i+1}$. 
More precisely, for each $i\in [\ell]$, let $u_i^-, v_i^-, v_i^+$ and $u_i^+$ be such that $L_iR_i=f_{u_i^-, v_i^-}\in E(D')$ and $L_{i+1}R_i=f_{u_i^+, v_i^+}\in E(D')$ and do the following.
\begin{enumerate}[label = (\roman{enumi})]
    \item If $R_i\in V_0$, then note that $v_i^-=v_i^+=R_i$ and let
    $P_i=u_i^-R_iu_i^+$.\label{(i)}

    \item If $R_i\in\mathcal H_R$ and $v_i^-=v_i^+$, then let
    $P_i=u_i^-v_i^-u_i^+$.\label{(ii)}

    \item If $R_i\in\mathcal H_R$ and $v_i^-\neq v_i^+$, then let
    $P_i^-=u_i^-v_i^-$ and $P_i^+=v_i^+u_i^+$.\label{(iii)}
\end{enumerate}
Now we turn $L_1R_1L_2R_2\dots L_{\ell} R_{\ell} L_1$ into a cluster cycle: for each $i\in [\ell]$, if \ref{(i)} or \ref{(ii)} occurs then replace $R_i$ in the sequence by $P_i$. If \ref{(iii)} occurs, then replace $R_i$ in the sequence by $P_i^- R_i P_i^+$. Labelling appropriately, the result is a sequence $H_1Q_1H_2Q_2\dots H_s Q_sH_1$ alternating between clusters and paths of length $\le 2$. To show that it is a cluster cycle we need to show that its paths are vertex disjoint.  To do this, for each $i,j\in [\ell]$ and $\sigma,\tau\in \{-,+\}$, we now prove three statements that certain pairs of vertices are distinct.
\begin{itemize}
    \item The vertices $u_i^\sigma$ and $u_j^\tau$ are distinct if $(i,\sigma)\neq (j,\tau)$. Indeed, if $u_i^\sigma v_i^\sigma\in M'_H$ and $u_j^\tau v_j^\tau\in M_{H'}'$ for some distinct $H,H'\in \cH$, then, using that $u_i^\sigma\in V(H)$, $u_j^\tau\in V(H')$ and $u_i^\sigma,u_j^\tau\notin V_{\ge 2}$, we have that $u_i^\sigma\neq u_j^\tau$.
    On the other hand, if $u_i^\sigma v_i^\sigma,u_j^\tau v_j^\tau\in M'_H$ for some $H\in \cH$, then $u_i^\sigma v_i^\sigma,u_j^\tau v_j^\tau$ correspond to edges in $D'$ directed from $H$ (which must be in $\cH_L$). Relabelling if necessary, we thus have $\sigma=-$, $j=i-1$ and $\tau=+$ (working $\mathrm{mod}\,\ell$ in the indices), and $H=L_i$. As $M'_{L_i}$ is a matching, we thus have
    $u_i^-\neq u_{i-1}^+$ in this case as well.

    \item  If $v_i^\sigma$ and $v_j^\tau$ occur in different paths then they are distinct. Indeed, suppose then that  $v_i^\sigma$ and $v_j^\tau$ occur in different paths and that $v_i^\sigma=v_j^\tau=v$ for some vertex $v$.     If $v\in V_0$, then, as $v$ is the in-vertex of at most two edges in the antidirected cycle $L_1R_1L_2R_2\dots L_{\ell} R_{\ell} L_1$, we must have, relabelling if necessary, that $j=i$, $\sigma=+$ and $\tau=-$.
    If $v\notin V_0$, then as $H(v)$ is the in-vertex of at most two edges in the same cycle, we reach the same conclusion. However, as $v_i^-=v_j^+=v_i^+$, the two corresponding edges
    are combined into the single path $P_i$ at \ref{(i)} or \ref{(ii)}, contradicting that  $v_i^\sigma$ and $v_j^\tau$ occur in different paths.

\item The vertices $u_i^\sigma$ and $v_j^\tau$ are distinct.  
    Suppose, for contradiction, that there is some vertex $x$ with 
    $u_j^\sigma=v_i^\tau=x$.
 Let $L$ be
    the cluster whose matching $M'_L$ contains the edge directed from 
    $u_j^\sigma$; thus $L\in\cH_L$ and $x\in V(L)$. Since this edge lies
    in $M'_L$, we have $x\notin V_{\geq 2}$, so $L$ is the unique cluster
    of $\cH$ containing $x$. As $x\in V(L)$, we have $x\notin V_0$.
Therefore, the definition of the auxiliary edge
 containing $v_i^\tau=x$ gives
    $R_i=H(x)=L.$
  But $R_i\in\cH_R$, whereas $L\in\cH_L$, contradicting $\cH_L\cap\cH_R=\emptyset$.
  \end{itemize}
Each path constructed in \ref{(i)}--\ref{(iii)} has vertex set contained in the 
vertices  $\{u_i^-, v_i^-,u_i^+,v_i^+\}$. The three observations above therefore
show that distinct paths are vertex disjoint. Hence the resulting cyclic sequence is a cluster
cycle, contradicting Lemma~\ref{lem:cluster properties}.
\end{proof}

\subsection{Proof of Theorem~\ref{thm:main}}\label{sec:lastproofbit}
Finally, we can put all our work together to prove Theorem~\ref{thm:main}.
\begin{proof}[Proof of Theorem~\ref{thm:main}.]
Fix $1/k_0\ll \eps\ll 1$ for which Lemmas~\ref{lem:few uncovered vertices}, ~\ref{lemma:t=1}, and~\ref{lem:not_many_clusters}  hold. We will prove the theorem for all  $k\ge k_0$. Supposing for contradiction that the theorem is false, choose $n$ minimum such that there exists an integer $k\ge k_0$ with $n\ge 2k$ and an $n$-vertex graph $G$ with more than $(k-1)(n-k+1)$ edges, such that either \textbf{a)} $\sum_{\ell\in \cC(G)} \frac{1}{\ell}< \sum_{\ell=2}^k \frac{1}{2\ell}$ or \textbf{b)} $e(G)\geq k(n-k)$,  $G$ is not a complete bipartite graph with vertex classes of size $k$ and $n-k$ and $\sum_{\ell\in \cC(G)} \frac{1}{\ell}= \sum_{\ell=2}^k \frac{1}{2\ell}$.

 Let $\mathcal H$ be a collection of edge-disjoint $(\eps, k)$-clusters in $G$ which maximises $\sum_{H\in \cH}e(H)$ and, subject to that, minimises $|\mathcal H|$. Lemma~\ref{lem:few uncovered vertices} applies which gives us $|V(G)\setminus V(\bigcup_{H\in \cH} H)|\leq \frac{2000|\mathcal{H}|}{\eps}$. In particular, this tells us that $|\mathcal{H}|\ne 0$ since otherwise the left hand side would be   $|V(G)|=n\ge 2k_0>0$ and the right hand side would be $0$.
Lemmas~\ref{lemma:t=1},~\ref{lem:not_many_clusters} apply which give us $|\mathcal H|\ne 1$ and $|\mathcal H|\le  1$ respectively. We have a contradiction since there are no more possible values for  $|\mathcal H|$.
\end{proof}

\subsubsection*{Acknowledgement}
The authors would like to thank the Forschungsinstitut für Mathematik and the London School of Economics for their generous  support during the months of May and June 2025 when this research took place.


\appendix

\section{Proof of Theorem~\ref{mainthm-new}}\label{appendix}
In this appendix, we confirm the modifications to the work of Liu and Montgomery~\cite{liu2023solution} required to prove Theorem~\ref{mainthm-new}. Readers new to the techniques may wish to read the detailed proof sketch which can be found in~\cite[Section~2.4]{liu2023solution} before reading this.

In Section~\ref{quotedresults}, we will record several results from~\cite{liu2023solution}, noting where we can additionally remove the bipartiteness condition without changing the proof. In Section~\ref{alteredproof}, we then prove Theorem~\ref{mainthm-new}, which we repeat below for convenience.

\mainthm*

The basic tool used in~\cite{liu2023solution} to adjust the length of a path is an \textit{adjuster}, defined as follows.

\begin{definition}\label{defn-adj}
A \emph{$(D,m,k)$-adjuster} $\cA=(v_1,F_1,v_2,F_2,A)$ in a graph $G$ consists of vertices $v_1,v_2\in V(G)$, graphs $F_1,F_2\subseteq G$ and a vertex set $A\subseteq V(G)$ such that the following hold for some $\ell\in\N$.
  \stepcounter{propcounter}
  \begin{enumerate}[label = {\bfseries \Alph{propcounter}\arabic{enumi}}]
  \item $A$,  $V(F_1)$ and $V(F_2)$ are pairwise disjoint.\label{d-a-1}
  \item $F_1, F_2$ have $D$ vertices each, and each $v\in V(F_i)$ has distance at most $m$ from $v_i$ in $F_i$, for $i=1, 2$. \label{d-a-2}
  \item $|A|\leq 10mk$.\label{d-a-3}
  \item For each $i\in \{0,1,\ldots,k\}$, there is a $v_1,v_2$-path in $G[A\cup \{v_1,v_2\}]$ with length $\ell+2i$.\label{d-a-4}
\end{enumerate}
We call the smallest such $\ell$ for which these properties hold the \emph{length of the adjuster} and denote it $\ell(\cA)$. Note that it immediately follows that $\ell(\cA)\leq |A|+1\leq 10mk+1$. We call a $(D,m,1)$-adjuster a \emph{{simple adjuster}}. We refer to the subgraphs $F_1$ and $F_2$ of an adjuster $\cA=(v_1,F_1,v_2,F_2,A)$ as the \emph{ends} of the adjuster, and let $V(\cA)=V(F_1)\cup V(F_2)\cup A$.
\end{definition}

\noindent
\begin{minipage}[t]{0.55\textwidth}
\vspace{0pt}
\setlength{\parindent}{17pt}
\indent
In simple terms, an adjuster is a gadget in which we can find paths of many different lengths between the two fixed end-vertices $v_1$ and $v_2$ (see Figure~\ref{fig:adjuster}). The purpose of subgraphs $F_1$ and $F_2$ is to make it easier to chain several adjusters together: the strategy will be to connect the subgraphs $F_i$ using sublinear expansion (for example, using a statement like Lemma~\ref{lem-diameter}) and then route the paths to the vertices $v_i$.
\end{minipage}\hfill
\begin{minipage}[t]{0.37\textwidth}
\vspace{0pt}
\centering

\begin{tikzpicture}[
    adjuster vertex/.style={circle, fill=black, inner sep=0, minimum size=4.5pt},
    adjuster edge/.style={draw=black, line width=0.7pt},
    x=1.15cm,
    y=1.15cm
]
\coordinate (v1) at (-1.35,0);
\coordinate (v2) at (1.35,0);
\draw[rounded corners=7pt, fill=black!3, line width=0.6pt]
    (-1.08,-0.9) rectangle (1.08,1.42);

\coordinate (p01) at (-0.38,-0.62);
\coordinate (p02) at (0.38,-0.62);

\coordinate (s) at (-0.9,0.0);
\coordinate (p11) at (-0.48,0.0);
\coordinate (p12) at (0,0.0);
\coordinate (p13) at (0.48,0);
\coordinate (p14) at (0.9,0);

\draw[adjuster edge] (v1)--(p01)--(p02)--(p14)--(v2);

\draw[adjuster edge] (v1)--(s);
\draw[adjuster edge] (s)--(p11)--(p12)--(p13)--(p14)--(v2);

\coordinate (p21) at (-0.72,0.56);
\coordinate (p22) at (-0.38,0.88);
\coordinate (p23) at (0,0.96);
\coordinate (p24) at (0.38,0.88);
\coordinate (p25) at (0.72,0.70);
\coordinate (p26) at (0.98,0.4);
\draw[adjuster edge] (s)--(p21)--(p22)--(p23)--(p24)--(p25)--(p26)--(v2);

\foreach \v in {v1,v2,s,p01,p02,p11,p12,p13,p14,p21,p22,p23,p24,p25,p26} {
    \node[adjuster vertex] at (\v) {};
}

\coordinate (left upper) at (-2,0.62);
\coordinate (left lower) at (-2,-0.62);
\coordinate (right upper) at (2,0.62);
\coordinate (right lower) at (2,-0.62);
\draw[adjuster edge] (v1)--(left upper)--(left lower)--cycle;
\draw[adjuster edge] (v2)--(right upper)--(right lower)--cycle;

\node at (-1.75,0) {$F_1$};
\node at (1.75,0) {$F_2$};
\node at (0.8,-0.7) {$A$};
\node[above] at (p23) {$\ell+4$};
\node[above=3pt] at (0, 0) {$\ell+2$};
\node[above=-2pt] at (0, -0.6) {$\ell$};
\node[below=2pt] at (v1) {$v_1$};
\node[below=2pt] at (v2) {$v_2$};
\end{tikzpicture}
\vspace{-0.3cm}

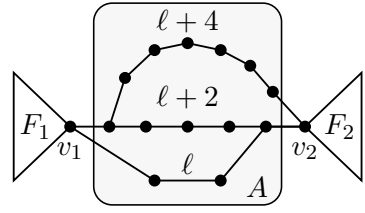
\captionof{figure}{An adjuster.}
\label{fig:adjuster}
\end{minipage}
\vspace{0.1cm}

Having understood why adjusters are useful, the natural question is: how does one build adjusters? Assuming that the property \textbf{C1} does not hold, the first step is to build simple adjusters, which support precisely two distinct path lengths (we show how to do this in Lemma~\ref{lem-robust-adj-new}). Once this is done, if we chain $k$ disjoint $(D, m, 1)$-adjusters, then we will be able to change the length of the resulting path by any $x\in \{2, 4, \dots, 2k\}$, by choosing how many simple adjusters we alter (this is done in Lemma~\ref{lem-adj-to-adj-path-new}).


\subsection{Quoted results}\label{quotedresults}

We will use the following simple and well-known result.

\begin{proposition}\label{bipsub} Within any graph $G$ there is a bipartite subgraph $H$ with $d(H)\geq d(G)/2$.
\end{proposition}

Combining this with Theorem~\ref{thm-expander}, we get the following corollary, which we use in the proof of Theorem~\ref{mainthm-new}.

\begin{corollary}\label{cor-expander}
For every sufficiently small $\eps_1>0$, the following holds for every $\eps_2>0$ and $d\in \N$. Every graph $G$ with $d(G)\geq 8d$ has a bipartite $(\eps_1,\eps_2d)$-expander subgraph $H$ with $\delta(H)\geq d$.\hfill\qed
\end{corollary}

The following lemma gives an upper bound on the diameter of the sublinear expander, and in that sense is very similar to Lemma~\ref{lem-diameter}. The main difference is that it assumes a minimum degree condition, and therefore can allow the sets which need to be connected to be very small. It is quoted verbatim from~\cite[Lemma~3.4]{liu2023solution}.

\begin{lemma}\label{new-connect} For each $0<\eps_1,\eps_2<1$, there exists $d_0=d_0(\eps_1,\eps_2)$ such that the following holds for each $n\geq d\geq d_0$ and $x\geq 1$. Let $G$ be an $n$-vertex $(\ep_1,\ep_2d)$-expander with $\delta(G)\geq d-1$.

Let $A,B\subseteq V(G)$ with $|A|,|B|\geq x$, and  let $W\subseteq V(G)\setminus(A\cup B)$ satisfy $|W|\log^3n\leq 10x$. Then, there is a path from $A$ to $B$ in $G-W$ with length at most $\frac{40}{\eps_1}\log^3n$.
\end{lemma}

We need the following definition, where we use the notation $B_\Gamma^{i}(X)=\{v\in V(\Gamma):\mathrm{dist}_\Gamma(v,X)\leq i\}$ for the ball of radius $i$ around the vertex set $X$ in the graph $\Gamma$.

\begin{definition}\label{defn-limit-contact}
	A vertex set $A$ has \emph{$k$-limited contact} with a vertex set $X$ in a graph $H$ if, for each $i\in \N$,
	$$
	|N_H(B_{H-X}^{i-1}(A))\cap X|\leq k i.
	$$
\end{definition}


The following lemma is not used directly here, but we include it to state clearly the slight modification in its statement (and correspondingly its proof) that we need. It is \cite[Lemma~3.5]{liu2023solution} with the bipartiteness condition removed and where \ref{z3} is a slightly weaker condition compared to that used in~\cite{liu2023solution} (where `$d_G(v,U)\leq d/2$' was used). However, it can be seen that the bipartiteness condition is never used in the proof of \cite[Lemma~3.5]{liu2023solution}, while, in both places that the corresponding version of \ref{z3} is used, the bound actually invoked is that from \ref{z3}. (Moreover, the bound used from \ref{z3} is far from optimised in the proof and thus the version we need holds comfortably.)

\begin{lemma}\label{lem-newbit}
  For any $0<\eps_1,\eps_2<1$, there exists $d_0=d_0(\eps_1,\eps_2)$ such that the following holds for each $n\geq d\geq d_0$. Suppose that $G$ is an  $n$-vertex   $(\eps_1,\eps_2 d)$-expander with $\de(G)\ge d$.

  Let $U\subseteq V(G)$ satisfy $|U|\leq \exp((\log\log n)^2)$, and let $K=G-U$. Let $I$ be any set and $V_i\subseteq V(K)$, $i\in I$, be pairwise disjoint sets such that, for each $i\in I$,
  \stepcounter{propcounter}
  \begin{enumerate}[label = \textup{\bfseries \Alph{propcounter}\arabic{enumi}}]
  	\item $\eps_2 d\le |V_i|\le \exp((\log\log n)^2)$,\label{z1}
  	  	\item $|N_{K}(V_i)|\leq \frac{5|V_i|}{\log^{10} |V_i|}$, and\label{z2}
  	  	  	\item $d_G(v,U)\leq (1-2\eps_2) d$ for each $v\in V_i$.\label{z3}
  \end{enumerate}

Then, $|\cup_{i\in I}V_i|<n^{1/8}$.
\end{lemma}

The following lemma is~\cite[Lemma~3.7]{liu2023solution}, with the bipartiteness condition removed and one of the conditions relaxed (\cite[\textbf{C4}]{liu2023solution}, to get \ref{mouse3}) so that certain vertices can have at most $(1-2\eps_2)d$ neighbours in a relevant set $U$, rather than at most $d/2$. This is possible as, firstly, the bipartiteness condition is used only in invoking~\cite[Lemma~3.5]{liu2023solution}, in whose proof the condition is not used. Secondly, the condition \cite[\textbf{C4}]{liu2023solution} is invoked twice in the proof of \cite[Lemma~3.7]{liu2023solution}. The first time it is used is for an inequality that can easily be checked to hold with the weaker \ref{mouse3}\footnote{More precisely, towards the end of the proof of \cite[Lemma~3.7]{liu2023solution}, it is reasoned `if $|A_i|<\eps_2d$, then \dots $|B_{G-U-B_i-C_i}(A_i)|\ge \delta(G)-|B_i|-4-d/2\ge\eps_2d$', where $\delta(G)\ge d$ and $|B_i|\leq |A_i|/\log^{10}|A_i|$. If \ref{mouse3} is used instead, we get  `if $|A_i|<\eps_2d$, then \dots $|B_{G-U-B_i-C_i}(A_i)|\ge \delta(G)-|B_i|-4-(1-2\eps_2)d\ge\eps_2d$' as $d\ge d_0(\eps_1,\eps_2)$.
}, while the second time it is used is to apply \cite[Lemma~3.5]{liu2023solution} which can be replaced by Lemma~\ref{lem-newbit}.

\begin{lemma}\label{lem-newbit-other-new}
  For each $0<\eps_1<1$, $0<\eps_2<1/5$ and $k\in \N$ there exists $d_0=d_0(\eps_1,\eps_2,k)$ such that the following holds for each $n\geq d\geq d_0$.
   Suppose that $G$ is an $n$-vertex $(\eps_1,\eps_2 d)$-expander with $\de(G)\ge d$.
  Let $U\subseteq V(G)$ satisfy $|U|\leq \exp((\log\log n)^2)$. Let $r=n^{1/8}$ and $\ell_0=(\log\log n)^{20}$. Suppose $(A_i,B_i,C_i)$, $i\in [r]$, are such that the following hold for each $i\in [r]$.
\stepcounter{propcounter}
\begin{enumerate}[label = \textup{\bfseries \Alph{propcounter}\arabic{enumi}}]
\item $|A_i|\geq d_0$.\label{mouse0}
\item $B_i\cup C_i$ and $A_i$ are disjoint sets in $V(G)\setminus U$, with $|B_i|\leq |A_i|/\log^{10}|A_i|$.\label{mouse1}
\item $A_i$ has $4$-limited contact with $C_i$ in $G-U-B_i$.\label{mouse2}
\item Each vertex in $B_{G-U-B_i-C_i}^{\ell_0}(A_i)$ has at most $(1-2\eps_2)d$ neighbours in $U$.\label{mouse3}
\item For each $j\in [r]\setminus\{i\}$, $A_i$ and $A_j$ are at least a distance $2\ell_0$ apart in $G-U-B_i-C_i-B_j-C_j$.\label{mouse4}
\end{enumerate}

Then, for some $i\in [r]$, $|B^{\ell_0}_{G-U-B_i-C_i}(A_i)|\geq \log^kn$.
\end{lemma}

We need the following definition.

\begin{definition}
Given a vertex $v$ in a graph $F$, $F$ is a \emph{$(D,m)$-expansion of $v$} if $|F|=D$ and $v$ is at distance at most $m$ in $F$ from any other vertex of $F$.
\end{definition}

The following result says that a $(D, m)$-expansion can always be trimmed to become a $(D', m)$-expansion, for any $D'\leq D$. This result is proved in \cite[Proposition~3.10]{liu2023solution} by successively removing the leaves from the spanning tree of depth at most $m$ until $D'$ vertices remain.

\begin{proposition}\label{prop-trimming} Let $D,m\in \N$ and $1\leq D'\leq D$. Then, any graph $F$ which is a $(D,m)$-expansion of $v$ contains a subgraph which is a $(D',m)$-expansion of $v$.
\end{proposition}

We will use the following result, which is \cite[Lemma~3.11]{liu2023solution} with the bipartiteness condition removed. This removal is possible as the condition is not used in its proof, which relies on \cite[Lemma~3.2]{liu2023solution}, \cite[Lemma~3.4]{liu2023solution}, and \cite[Proposition~3.10]{liu2023solution}, none of which use a bipartiteness condition.

\begin{lemma}\label{lem-expansion-new}
	For each $k\in \N$ and any $0<\eps_1, \eps_2<1$, there exists $d_0=d_0(\ep_1,\ep_2,k)$ such that the following holds for each $n\geq d\geq d_0$.

  Suppose that $G$ is an $n$-vertex $(\eps_1,\eps_2 d)$-expander with $\delta(G)\geq d-1$.
	Let $m=\frac{40}{\ep_1}\log^3 n$. Let $C$ be a shortest cycle in $G$, and let $x_1,\ldots,x_k$ be distinct vertices in $G$. For each $i,j\in [k]$, let $D_{i,j}\in [1,\log^{5k}n]$.

    Then, there are graphs $F_{i,j}\subseteq G$, $i,j\in [k]$, such that the following hold.
  \begin{itemize}
    \item For each $i,j\in [k]$, $F_{i,j}$ is a $(D_{i,j},5m)$-expansion around $x_i$ which contains no vertices other than $x_i$ in $V(C)\cup \{x_1,\ldots,x_k\}$.
    \item The sets  $V(F_{i,j})\setminus\{x_i\}$, $i,j\in [k]$, are pairwise disjoint.
\end{itemize}
\end{lemma}

We will also use the following result, which is \cite[Lemma~3.12]{liu2023solution} with the bipartiteness condition removed. This removal is possible as the condition is not used in its proof, which relies on no other results in \cite{liu2023solution}.

\begin{lemma}\label{cl-egg-new} For any $0<\eps_1, \eps_2<1$, there exists $d_0=d_0(\eps_1,\eps_2)$ such that the following holds for each $n\geq d\geq d_0$. Suppose that $G$ is an $n$-vertex $(\eps_1,\eps_2 d)$-expander with $\de(G)\ge d$ and let $m=\frac{50}{\eps_1}\log^3n$.

For any set $W\subseteq V(G)$ with $|W|\leq \eps_1 n/(100\log^{2}n)$, there is a set $B\subseteq V(G)\setminus W$ with size at least $n/25$ and diameter at most $2m$, and such that $G[B]$ is a $(D,m)$-expansion around some vertex $v\in B$ for $D=|B|$.
\end{lemma}

Next, we will use the following result, which is \cite[Corollary~3.15]{liu2023solution} with the bipartiteness condition removed. This is possible, as follows. Its proof does not use the bipartiteness condition directly, but invokes~\cite[Lemma~3.4]{liu2023solution}, \cite[Lemma~3.13]{liu2023solution} and \cite[Lemma~3.14]{liu2023solution}, the latter two of which require a bipartiteness condition. However, both \cite[Lemma~3.13]{liu2023solution} and \cite[Lemma~3.14]{liu2023solution} can have their bipartiteness condition removed as their proofs do not use it directly and only invoke \cite[Lemma~3.4]{liu2023solution}, \cite[Proposition~3.10]{liu2023solution} and \cite[Lemma~3.12]{liu2023solution}, only the last of which has a bipartiteness condition, which we have already noted is unneeded while recording it as Lemma~\ref{cl-egg-new}.

\begin{lemma}\label{longconnect4-new}
For any $0<\eps_1, \eps_2<1$, there exists $d_0=d_0(\eps_1,\eps_2)$ such that the following holds for each $n\geq d\geq d_0$. Suppose that $G$ is an $n$-vertex $(\eps_1,\eps_2 d)$-expander with $\de(G)\ge d$.

Let $\log^{10}n\leq D\leq n/\log^{10}n$, $\frac{100}{\ep_1}\log^3n\le m\le \log^4n$ and $\ell\leq n/\log^{12}n$. Let $A\subseteq V(G)$ satisfy $|A|\leq D/\log^3n$. Let $F_1,\ldots,F_4\subseteq G-A$ be vertex-disjoint subgraphs and $v_1,\ldots,v_4$ be vertices such that, for each $i\in [4]$, $F_i$ is a $(D,m)$-expansion of $v_i$.

Then, $G-A$ contains vertex-disjoint paths $P$ and $Q$ with $\ell \leq \ell(P)+\ell(Q)\leq \ell+22m$ such that both $P$ and $Q$ connect $\{v_1,v_2\}$ to $\{v_3,v_4\}$.
\end{lemma}

Within a bipartite expander, we will be able to find a simple adjuster, using the following result, which is \cite[Lemma 4.2]{liu2023solution}. The bipartiteness condition is used in its proof. Though it would not be difficult to prove a very similar result without this condition, in order to minimise our alterations we will use it directly.

\begin{lemma}\label{lem-twin-path}
	For any $0<\eps_1<1$, $0<\eps_2<1/5$ and $k\in \N$, there exists $d_0=d_0(\eps_1,\eps_2,k)$ such that the following is true for each $n\geq d\geq d_0$. Suppose that $G$ is an $n$-vertex bipartite $(\eps_1,\eps_2 d)$-expander with $\de(G)\ge d-1$.

  Let $C$ be a shortest cycle in $G$ and let $x_1,x_2$ be distinct vertices in $V(G)\setminus V(C)$. Let $m=\frac{200}{\eps_1}\log^3n$ and $D\leq \log^{5k}n$.

  Then, $G$ contains a $(D,m,1)$-adjuster $(v_1,F_1,v_2,F_2,A)$ with $v_1=x_1$, $v_2=x_2$ and $V(C)\subseteq A$.
\end{lemma}

\subsection{Proof of Theorem~\ref{mainthm-new}}\label{alteredproof}

We can now prove Theorem~\ref{mainthm-new}. We do so while making minimal adjustments to the proofs in \cite{liu2023solution}. In essence, the proof proceeds by assuming that \textbf{C1} does not hold, and showing that simple adjusters can be built robustly, and then combined to get many distinct path lengths between any two vertices $v_1, v_2\in V(G)$. In what follows, we will denote the negation of the assumption \textbf{C1} by \textbf{H}.

Aside from removing some bipartiteness conditions in Section~\ref{quotedresults}, as described in Section~\ref{sec:largeexpandercase}, instead of  condition \ref{prop:oldEHthm3} in Theorem~\ref{mainthm-old}, we will use \ref{prop:newEHthm3}. Essentially, aside from using the minimum degree bound more closely here, \ref{prop:newEHthm3} was used in \cite{liu2023solution} to imply \ref{prop:oldEHthm3}. Thus, the impact of the alteration is minimal, but requires a small change throughout the proof.
As noted above, the proof of Theorem~\ref{mainthm-new} is discussed without these modifications in Section~2.4 of~\cite{liu2023solution}. The following proof follows closely the proof in \cite[Sections~4.2, 4.3 and 4.5]{liu2023solution}.

In Section~\ref{sec:robadj-new}, we use Lemma~\ref{lem-twin-path} to find such an adjuster despite the removal of any medium-sized vertex set from the expander, giving Lemma~\ref{lem-robust-adj-new}.
In Section~\ref{sec:adjpath-new}, we chain simple adjusters together for Lemma~\ref{lem-adj-to-adj-path-new}, before using this to join vertex expansions by paths with precise lengths for Lemma~\ref{lem-finalconnect-new}. Finally, we prove Theorem~\ref{mainthm-new} in Section~\ref{sec:mainthmfinal-new}.

\subsubsection{Finding simple adjusters robustly}\label{sec:robadj-new}
In this section, we prove Lemma~\ref{lem-robust-adj-new}, a key component of our proof. This finds a simple adjuster robustly in an expander $G$ -- that is, given any subset $U\subseteq V(G)$ with moderate size, we construct an adjuster in $G-U$. Note that the property \textbf{H} assumed by this lemma is precisely the negation of the property~\ref{prop:newEHthm3} required by Theorem~\ref{mainthm-new}.

\begin{lemma}\label{lem-robust-adj-new}
 For every sufficiently small $\eps_1>0$ and every $0<\eps_2<1/5$, there exists $d_0=d_0(\eps_1,\eps_2)$ such that the following is true for each $n\geq d\geq d_0$. Suppose that $G$ is an $n$-vertex $(\eps_1,\eps_2 d)$-expander with $\de(G)\ge d$ for which the following property holds.
\stepcounter{propcounter}
\begin{enumerate}[label = \textup{\textbf{\Alph{propcounter}}}]
\item\label{newpropP} There are no disjoint sets $U_0,W_0\subset V(G)$ such that $|U_0|\leq \log^{20}n$, $|W_0|\geq \log^{80}n$ and every vertex in $W_0$ has at least $(1-2\eps_2)d$ neighbours in $U_0$.
\end{enumerate}

Let $m=\frac{400}{\eps_1}\log^3n$ and $D=\log^{14} n$. For every subset $U\subseteq V(G)$ with $|U|\leq 10D$, $G-U$ contains a $(D,2m,1)$-adjuster.
\end{lemma}
\begin{proof} Let $0<\eps_1<1$ be small enough that the property in Corollary~\ref{cor-expander} holds. Suppose, for contradiction, that $G-U$ contains no $(D,2m,1)$-adjuster. Let $\Delta=200mD$, $L=\{v\in V(G):d_G(v)\geq \Delta\}$ and $G'=G-L$, so that $\Delta(G')\leq \Delta$.

Set $\ell_0=(\log\log n)^{20}$. Let $Z_0=\{v\in V(G)\setminus U:d_G(v,U)\geq (1-2\eps_2)d\}$. Then, by \ref{newpropP}, as $|U|\leq 10D\le \log^{20}n$, we have $|Z_0|< \log^{80}n$. Hence, as $\delta(G)\geq d$ and $n\geq d_0(\eps_1,\eps_2)$ is large, $G-U$ contains at least $(n-|U|-|Z_0|)\cdot 2\eps_2 d/2\geq \eps_2 nd/2$ edges. Let $U_1=U\cup Z_0$, so that $|U_1|\le 2\log^{80}n$.

Take a maximal collection $\mathbf{A}_0$ of adjusters in $G-U$, such that the following hold.
\stepcounter{propcounter}
\begin{enumerate}[label = {\bfseries \Alph{propcounter}\arabic{enumi}}]
\item The sets $V(F_{1}\cup F_2)$, $(v_1,F_1,v_2,F_2,A)\in \mathbf{A}_0$, are subsets of $V(G')$ and are all at least a distance $10\ell_0$ apart from each other and from $U_1\setminus L$ in $G'$.\label{hot1}
\item For each $\cA\in \mathbf{A}_0$, for some $m_{\cA}$ with $\log^3d_0\leq m_{\cA}\leq m$, $\cA$ is an $(m_\cA^{2},m_\cA,1)$-adjuster.\label{hot2}
\end{enumerate}

  \begin{claim}\label{ping}
$|\mathbf{A}_0|\geq n^{1/4}$.
  \end{claim}
  \claimproofstart Suppose, for contradiction, that $|\mathbf{A}_0|< n^{1/4}$. Let $W=(U_1\cup (\bigcup_{\cA\in \mathbf{A}_0}V(\cA)))\setminus L$. For each $\cA=(v_1,F_1,v_2,F_2,A)\in \mathbf{A}_0$, $|V(\cA)|=|F_1|+|F_2|+|A|\leq 2m_\cA^2+10m_\cA\leq 3m^2$, and therefore $|W|\leq n^{1/4}\cdot 3m^3+2\log^{80}n\leq n^{1/3}$. Let $W'=B_{G'}^{10\ell_0}(W)$, so, as $\Delta(G')\le \Delta$, we have that $|W'|\leq 2|W|\cdot\Delta^{10\ell_0}\leq n^{1/2}$.

Now, there are at most $|W'|\Delta\leq \Delta n^{1/2}\leq \eps_2 nd/4$ edges in $G$ with some vertex in $W'$. Let $\bar{d}=\eps_2 d/16$. As  $G-U$ contains at least $\eps_2 nd/2$ edges, $G-U-W'$ contains at least $\eps_2 nd/4$ edges, so that $d(G-U-W')\geq \eps_2 d/2=8\bar{d}$. Then, by Corollary~\ref{cor-expander}, $G-U-W'$ contains an $(\eps_1,\eps_2\bar{d})$-expander $H$ with $\delta(H)\geq \bar{d}$.
Let $C$ be a shortest cycle in $H$. We will consider two cases, depending on how many vertices of $L$ there are in $V(H)\setminus V(C)$.

\medskip

\noindent\textbf{Case I: $|(V(H)\setminus V(C))\cap L|\leq 1$.} Let $H'=H-(V(H)\setminus V(C))\cap L$, so that $\delta(H')\geq \bar{d}-1$. Note that, for each $X\subseteq V(H')$ with $\eps_2\bar{d}/2\leq |X|\leq |H'|/2\le |H|/2$, we have
\begin{align*}
|N_{H'}(X)|&\geq |N_H(X)|-1\geq |X|\cdot \rho(|X|,\eps_1,\eps_2\bar{d})-1 \\
&\geq \frac{1}{2}|X|\cdot \rho(|X|,\eps_1,\eps_2\bar{d})+\frac{\eps_2\bar{d}}{4}\cdot \rho(\eps_2\bar{d}/2,\eps_1,\eps_2\bar{d})-1
\\
&\geq |X|\cdot \rho(|X|,\eps_1/2,\eps_2\bar{d})+\frac{\eps_2\bar{d}}{4}\cdot \frac{\eps_1}{\log^2(15/2)}-1\geq |X|\cdot \rho(|X|,\eps_1/2,\eps_2\bar{d}),
\end{align*}
where the last inequality follows as $\bar{d}\geq 2\eps_2 \cdot d_0(\eps_1,\eps_2)/32$ is large. Therefore, $H'$ is a $(\eps_1/2,\eps_2\bar{d})$-expander with $\delta(H')\geq \bar{d}-1$. Note that $C$ is a shortest cycle in $H'$.

Let $m_{H'}=400\log^3|H'|/\eps_1\le m$, and note that, as $|H'|\geq \delta(H')+1\geq \bar{d}\geq \eps_2 d_0/16$, and $d_0=d_0(\ep_1,\ep_2)$ is large, $m_{H'}\geq \log^3d_0$. Since $C$ is induced and $\delta(H')\ge \bar d-1>3$, there are at least two vertices outside $C$. Picking arbitrary distinct vertices $x_1,x_2\in V(H')\setminus V(C)$ and noting that $\bar{d}\geq 2\eps_2 \cdot d_0(\eps_1,\eps_2)/32$ is large, by Lemma~\ref{lem-twin-path} with $(k,D)_{\ref{lem-twin-path}}=(10,m_{H'}^2)$,
 $H'$ contains an $(m_{H'}^{2},m_{H'},1)$-adjuster  $(v_1,F_1,v_2,F_2,A)$ with $V(C)\subseteq A$. As $A$ is disjoint from $V(F_1\cup F_2)$, $V(C)\subseteq A$ and $(V(H')\setminus V(C))\cap L=\varnothing$, we have that $V(F_1\cup F_2)$ is disjoint from $L$, and hence lies in $V(G')$.
 Together with $V(F_1\cup F_2)\subseteq V(H')$ being disjoint from $W'$ and so $10\ell_0$-far in $G'$ from the ends of the adjusters in $\mathbf{A}_0$ and from $U_1\setminus L$, this violates the maximality of $\mathbf{A}_0$, a contradiction.

\medskip

\noindent\textbf{Case II: $|(V(H)\setminus V(C))\cap L|\geq 2$.} Let $x_1,x_2\in (V(H)\setminus V(C))\cap L$ be distinct and let $m_{H}=200\log^3|H|/\eps_1\le m$. By Lemma~\ref{lem-twin-path} with $(k,D)_{\ref{lem-twin-path}}=(1,1)$, $H$ contains a $(1,m_{H},1)$-adjuster $(v_1,F_1,v_2,F_2,A)$ with $v_1=x_1$ and $v_2=x_2$. Using that $|A|\leq 10m_{H}\leq 10m$, $|U|\leq 10D$, and $d_G(x_1),d_G(x_2)\geq \Delta=200mD$,
pick disjoint sets $X_1\subseteq N_G(x_1)\setminus (U\cup A\cup \{x_2\})$ and $X_2\subseteq N_G(x_2)\setminus (U\cup A\cup \{x_1\})$ with $|X_1|=|X_2|=D-1$. Letting $F_i'=G[\{x_i\}\cup X_i]$ for each $i\in [2]$, and noting $|A|\leq 20m$, we have that $(x_1,F'_1,x_2,F'_2,A)$ is a $(D,2m,1)$-adjuster in $G-U$, a contradiction.
  \claimproofend

Now, let $\mathbf{A}_1\subseteq \mathbf{A}_0$ be the set of adjusters $(v_1,F_1,v_2,F_2,A)\in \mathbf{A}_0$ for which there is no path with length at most $\ell_0$ from $V(F_1)\cup V(F_2)$ to $L\setminus U$ in $G-U-A$.

  \begin{claim}\label{ping2}
$|\mathbf{A}_1|\geq n^{1/4}/2$.
  \end{claim}
  \claimproofstart Let $r= n^{1/8}$.
Suppose, for contradiction, that we can label distinct $\cA_1,\ldots,\cA_r\in \mathbf{A}_0\setminus \mathbf{A}_1$. Say, for each $i\in [r]$, that $\cA_i=(v_{i,1},F_{i,1},v_{i,2},F_{i,2},\bar{A}_i)$ and let $P_i'$ be a shortest path with length at most $\ell_0$ from $V(F_{i,1})\cup V(F_{i,2})$ to $L\setminus U$ in $G-U-\bar{A}_i$. Relabelling, if necessary, for each $i\in [r]$ suppose the endvertex of $P_i'$ in $V(F_{i,1}\cup F_{i,2})$ is in $V(F_{i,1})$, and let $Q_i$ be a path from this endvertex of $P_i'$ to $v_{i,1}$ in $F_{i,1}$ with length at most $m_{\cA_i}$.

For each $i\in [r]$, let $x_i$ be the endpoint of $P_i'$ in $L\setminus U$, and let $P_i=P_i'-x_i$. We shall apply Lemma~\ref{lem-newbit-other-new} by setting, for each $i\in [r]$,  $A_i=V(F_{i,2})$, $B_i=\bar{A}_i\cup V(Q_i)\cup \{x_i\}$ and $C_i=V(P_i)$. Firstly, as $|A_i|=m_{\cA_i}^2\geq \log^6d_0$ by \ref{hot2} and $d_0=d_0(\ep_1,\ep_2)$ is large, we have that $|A_i|\geq d_0^{\ref{lem-newbit-other-new}}$, where $d_0^{\ref{lem-newbit-other-new}}=d_0^{\ref{lem-newbit-other-new}}(\eps_1,\eps_2,100)$ is the function in Lemma~\ref{lem-newbit-other-new}, so that \ref{mouse0} holds.

 As $V(F_{i,2})\subseteq V(G')=V(G)\setminus L$ by \ref{hot1}, and $V(F_{i,2})$ is disjoint from $V(F_{i,1})$ and $\bar{A}_i$ by \ref{d-a-1}, we have that $A_i$ and $B_i\cup C_i$ are disjoint. Furthermore, $|B_i|\leq |\bar{A}_i|+|Q_i|+1\leq 20m_{\cA_i}\leq m_{\cA_i}^2/\log^{10}(m_{\cA_i}^2)$ as $m_{\cA_i}\geq \log^3 d_0$ is large, and thus \ref{mouse1} holds.

Now, as $P_i'$ is a shortest path from $V(F_{i,1})\cup V(F_{i,2})$ to $L\setminus U$ in $G-U-\bar{A}_i$, which has an endvertex in $V(F_{i,1})$, and $A_i=V(F_{i,2})$, we have, for each $\ell\in \N$, that $B^\ell_{G-U-\bar{A_i}}(A_i)$
has at most $\ell+1$ vertices in $P_i'$, and hence $P_i$. Therefore, $A_i$ has 4-limited contact with $C_i$ in $G-U-\bar{A_i}$, and hence in $G-U-B_i$, and thus \ref{mouse2} holds.

Suppose there is a path, $R_i$ say, with length at most $10\ell_0$ from $A_i$ to $L\setminus (U\cup\{x_i\})$ in $G-U-B_i-C_i$. Then, there is a path $R'_i\subseteq R_i\cup F_{i,2}$ from $v_{i,2}$ to some vertex $y_i\in L\setminus (U\cup\{x_i\})$ with length at most $10\ell_0+m_{\cA_i}\leq 2m-1$, and the path $Q_i\cup P'_i$ is a path from $v_{i,1}$ to $x_i$ with length at most $m_{\cA_i}+\ell_0\leq 2m-1$ in $G-U-\bar{A}_i$ with vertices in $B_i\cup C_i$.
Then, as $|U\cup \bar A_i\cup V(R_i')\cup V(Q_i\cup P_i')|\leq 10D+10m_{\cA_i}+4m\leq 10D+15m$, as $x_i,y_i\in L$ both have degree at least $\Delta=200mD$, we can choose $X_i\subseteq N_G(x_i)$ and $Y_i\subseteq N_G(y_i)$ which are disjoint from each other and from $U\cup \bar A_i\cup V(R_i')\cup V(Q_i\cup P_i')$ and have size $D-|P_i'\cup Q_i|$ and $D-|R'_i|$ respectively. Then, $(v_{i,1}, G[X_i\cup V(P_i')\cup V(Q_i)],v_{i,2},G[Y_i\cup V(R_i')],\bar A_i)$
is a $(D,2m,1)$-adjuster in $G-U$, a contradiction.
Therefore, there is no such path $R_i$. Consequently, recalling that $A_i=V(F_{i,2})$, we have
\[
B_{G-U-B_i-C_i}^{\ell_0}(A_i)=B_{G'-U-B_i-C_i}^{\ell_0}(A_i),
\]
which, by~\ref{hot1}, is disjoint from $U_1$. By the choice of $Z_0\subseteq U_1$, we have that \ref{mouse3} holds.

Now, similarly, for any $j\in [r]\setminus\{i\}$, we have that $B_{G-U-B_j-C_j}^{\ell_0}(A_j)=B_{G'-U-B_j-C_j}^{\ell_0}(A_j)$, so that, by \ref{hot1}, $B_{G-U-B_j-C_j}^{\ell_0}(A_j)$ and $B_{G-U-B_i-C_i}^{\ell_0}(A_i)$ are disjoint. In particular, $A_i$ and $A_j$ are a distance at least $2\ell_0$ apart in $G-U-B_i-C_i-B_j-C_j$, and therefore $\ref{mouse4}$ holds.

Thus, by Lemma~\ref{lem-newbit-other-new}, there is some $j\in [r]$ for which $|B^{\ell_0}_{G-U-B_j-C_j}(A_j)|\geq \log^{100}n\geq D$. As $F_{j,2}$ is an $(m_{\cA_j}^{2},m_{\cA_j})$-expansion of $v_{j,2}$  in $G'-U-B_j-C_j$, $m_{\cA_j}\leq m$ and $A_j=V(F_{j,2})$, we have that $|B^{2m}_{G-U-B_j-C_j}(v_{j,2})|\geq D$ as $\ell_0\ll m$. Therefore, by Proposition~\ref{prop-trimming}, we can pick a $(D,2m)$-expansion, $F'_{j,2}$ say, of $v_{j,2}$ in $G-U-B_j-C_j$.

As $x_j\in L$, we can then pick a set $U'$ of neighbours of $x_{j}$ disjoint from $U\cup V(F'_{j,2})\cup \bar{A}_{j}\cup  V(Q_j)\cup V(P_j')$ with $|U'|=D-|V(P_j'\cup Q_j)|$. Let $F'_{j,1}=G[U'\cup V(P_j')\cup V(Q_j)]$. Note that $F'_{j,1}$ is then a $(D,2m)$-expansion of $v_{j,1}$ as $Q_j\cup P'_j$ is a $v_{j,1},x_j$-path with length at most $m_{\cA_j}+\ell_0\leq 2m-1$. Finally, note that $(v_{j,1},F'_{j,1},v_{j,2},F'_{j,2},\bar{A}_j)$ is a $(D,2m,1)$-adjuster in $G-U$, a contradiction. Therefore, $ |\mathbf{A}_0 \setminus \mathbf{A}_1| <  r=n^{1/8}$, and so by Claim~\ref{ping}, we have $|\mathbf{A}_1|> n^{1/4}-r\geq n^{1/4}/2$.
\claimproofend

Let $\mathbf{A}_1'\subseteq \mathbf{A}_1$ satisfy $|\mathbf{A}_1'|=n^{1/4}/2$. Then, $|\cup_{\cA\in \mathbf{A}_1'}V(\cA)|\leq n^{1/4}\cdot 3m^2\leq n^{1/3}$ by \ref{hot2}. 
Set $W=U\cup B_{G'}^{\ell_0}(\cup_{\cA\in \mathbf{A}_1'}(V(\cA)\setminus L))\cup \bigcup_{\mathcal A\in \mathbf A_1'}V(\mathcal A)$, noting that $|W|\le  10D+n^{1/3}\cdot 2\Delta^{\ell_0}+n^{1/3}\leq n^{1/2}.$ 
Thus, by Lemma~\ref{cl-egg-new} and Proposition~\ref{prop-trimming}, there is a set $Z\subseteq V(G)\setminus W$ with $|Z|=10m^2D$ so that $G[Z]$ has diameter at most $m/2$ and $Z$ is a distance greater than $\ell_0$ in $G'$ from $V(\cA)\setminus L$ for each $\cA\in \mathbf{A}_1'$. Note that since $Z\subseteq V(G)\setminus W$, we have $Z\cap V(\mathcal A)=\emptyset$ for all $\mathcal A\in \mathbf A_1'$.

Let $\mathbf{A}_2\subseteq \mathbf{A}_1'$ be the set of adjusters $(v_1,F_1,v_2,F_2,A)\in \mathbf{A}_1'$ for which there is no path with length at most $m/2$ from $V(F_1)\cup V(F_2)$ to $Z$ in $G-U-A$.

  \begin{claim}\label{ping3}
$|\mathbf{A}_2|\geq n^{1/4}/4$.
  \end{claim}
  \claimproofstart Let $r=n^{1/8}$.
Suppose, for contradiction, we can label distinct $\cA_1,\ldots,\cA_r\in \mathbf{A}_1'\setminus \mathbf{A}_2$. Say, for each $i\in [r]$, that $\cA_i=(v_{i,1},F_{i,1},v_{i,2},F_{i,2},\bar{A}_i)$ and let $P_i$ be a shortest path with length at most $m/2$ from $V(F_{i,1})\cup V(F_{i,2})$ to $Z$ in $G-U-\bar{A}_i$. Relabelling, if necessary, for each $i\in [r]$ suppose the endvertex of $P_i$ in $V(F_{i,1}\cup F_{i,2})$ is in $V(F_{i,1})$, and let $Q_i$ be a path from this endvertex of $P_i$ to $v_{i,1}$ in $F_{i,1}$ with length at most $m_{\cA_i}$.

We will apply Lemma~\ref{lem-newbit-other-new} to $A_i=V(F_{i,2})$, $B_i=\bar{A}_i\cup V(Q_i)$ and $C_i=V(P_i)$, for each $i\in [r]$. For each $i\in [r]$, similarly to the proof of Claim~\ref{ping2}, we have that \ref{mouse0}--\ref{mouse2} hold.
By the choice of $\mathbf{A}_1$, for each $i\in[r]$, there is no path of length at most $\ell_0$ from $A_i$ to $L\setminus U$ in $G-U-B_i-C_i$.
Therefore, the sets $B_{G-U-B_i-C_i}^{\ell_0}(A_i)$ and $B_{G'-U-B_i-C_i}^{\ell_0}(A_i)$ are the same set, and thus, by \ref{hot1}, this set is disjoint from $U_1$. Thus, \ref{mouse3} holds by the definition of $Z_0$.
It similarly follows that $B_{G-U-B_i-C_i}^{\ell_0}(A_i)$ and $B_{G-U-B_j-C_j}^{\ell_0}(A_j)$ are vertex-disjoint for each $j\in [r]\setminus \{i\}$, and thus \ref{mouse4} holds.

Thus, by Lemma~\ref{lem-newbit-other-new}, there is some $j\in [r]$ for which $|B^{\ell_0}_{G'-U-B_j-C_j}(A_j)|=|B^{\ell_0}_{G-U-B_j-C_j}(A_j)|\geq D$. Thus, as $F_{j,2}$ is an $(m_{\cA_j}^2,m_{\cA_j})$-expansion of $v_{j,2}$ in $G'-U-B_j-C_j$ by \ref{hot1} and \ref{hot2}, and $A_j=V(F_{j,2})$, by Proposition~\ref{prop-trimming}, there is a $(D,2m)$-expansion, $F'_{j,2}$ say, of $v_{j,2}$ in $B^{\ell_0}_{G'-U-B_j-C_j}(V(F_{j,2}))$. As $Z$ was chosen to have a distance greater than $\ell_0$ in $G'$ from $V(\cA_j)\setminus L$, we have that $V(F'_{j,2})$ is disjoint from $Z$.

Now, as $Z$ has diameter at most $m/2$ in $G$, $Q_j\cup P_j\cup G[Z]$ is an expansion of $v_{j,1}$ with radius at most $\ell(Q_j)+\ell(P_j)+m/2\leq 2m$ and size at least $D$. Therefore, by Proposition~\ref{prop-trimming}, we can find within $Q_j\cup P_j\cup G[Z]$ a $(D,2m)$-expansion, $F'_{j,1}$ say, of $v_{j,1}$,
which then must be vertex-disjoint from $\bar{A}_j$ and from $V(F'_{j,2})\subseteq B^{\ell_0}_{G'-U-B_j-C_j}(V(F_{j,2}))$.
Thus,  we have that $(v_{j,1},F'_{j,1},v_{j,2},F'_{j,2},\bar{A}_j)$
is a $(D,2m,1)$-adjuster in $G-U$, a contradiction. Thus, $|\mathbf{A}_2|\geq |\mathbf{A}_1'|-r\geq n^{1/4}/4$, by Claim~\ref{ping2}.
\claimproofend

Let $r=n^{1/8}$. Using Claim~\ref{ping3}, label distinct $\cA_1,\ldots,\cA_r\in \mathbf{A}_2$, and say, for each $i\in [r]$, that $\cA_i=(v_{i,1},F_{i,1},v_{i,2},F_{i,2},\bar{A}_i)$.
We shall apply Lemma~\ref{lem-newbit-other-new}
to $A_i=V(F_{i,1}\cup F_{i,2})$, $B_i=\bar{A}_i$ and $C_i=\varnothing$.
Similarly to the proof of Claim~\ref{ping3}, the only difference being that \ref{mouse2} holds trivially as $C_i=\varnothing$ and $A_i$ is slightly larger, we have that~\ref{mouse0}--\ref{mouse4} hold.

Thus, applying  Lemma~\ref{lem-newbit-other-new} with $k=100$, there is some $j\in [r]$ with $|B^{\ell_0}_{G-U-B_j-C_j}(A_j)|=|B^{\ell_0}_{G-U-B_j}(A_j)|\geq 10m^2D\ge 10\log^3n |U\cup B_j|$. Therefore, by Lemma~\ref{new-connect}, there is a path in $G-U-B_j$
 from $B^{\ell_0}_{G-U-B_j}(A_j)$ to $Z$ with length at most $m/4$. Then, as $A_j=V(F_{j,1}\cup F_{j,2})$ and $B_j=\bar{A}_j$, there is a path in $G-U-\bar{A}_j$ from $V(F_{j,1}\cup F_{j,2})$ to $Z$ with length at most $m/2$, contradicting $\cA_j\in \mathbf{A}_2$, and completing the proof.
\end{proof}


\subsubsection{Connecting simple adjusters for paths with specific lengths}\label{sec:adjpath-new}
Using Lemma~\ref{lem-robust-adj-new}, we can find many vertex-disjoint simple adjusters. We now connect them together into a larger adjuster, for Lemma~\ref{lem-adj-to-adj-path-new}, before using these to construct paths with specific lengths for Lemma~\ref{lem-finalconnect-new}.

\begin{lemma}\label{lem-adj-to-adj-path-new}
	 For every sufficiently small $\eps_1>0$ and any $0<\eps_2<1/5$, there exists $d_0=d_0(\eps_1,\eps_2)$ such that the following holds for each $n\geq d\geq d_0$. Suppose that $G$ is an $n$-vertex $(\eps_1,\eps_2 d)$-expander with $\de(G)\ge d$ and the following property.

\stepcounter{propcounter}
\begin{enumerate}[label = \emph{\textbf{H}}]
  \item There are no disjoint sets $U_0,W_0\subset V(G)$ such that $|U_0|\leq \log^{20}n$, $|W_0|\geq \log^{80}n$ and every vertex in $W_0$ has at least $(1-2\eps_2)d$ neighbours in $U_0$.\label{newpropP-1}
  \end{enumerate}

\noindent	Let $m=\frac{1600}{\ep_1}\log^3n$ and $D=\log^{10} n$.
	Suppose $1\le r\le 30m$ and $U\subseteq V(G)$ with $|U|\leq 2D$.

Then, there is a $(D,m,r)$-adjuster in $G-U$.
\end{lemma}
\begin{proof} Let $\eps_1>0$ be sufficiently small that the property in Lemma~\ref{lem-robust-adj-new} holds, and set $D':=\log^{14}n$. By that lemma, as $d\geq d_0(\ep_1,\ep_2)$ is large, for every set $V\subseteq V(G)$ with $|V|\leq 10D'$, the graph $G-V$ contains a $(D',m/2,1)$-adjuster. By Lemma~\ref{new-connect}, for any sets $X$ and $Y$ with size at least $2D'$, and any set $V\subseteq V(G)\setminus(X\cup Y)$ with size at most $20D'/\log^3n$, there is a path from $X$ to $Y$ in $G-V$ with length at most $m/40$.

We prove by induction on $r$ that $G-U$ contains a $(D',m,r)$-adjuster. For $r=1$, this follows from Lemma~\ref{lem-robust-adj-new}, since $|U|\leq 2D\leq 10D'$. Suppose that $1\leq r<30m$ and that $G-U$ contains a $(D',m,r)$-adjuster $(v_1,F_1,v_2,F_2,A_1)$. Let
 $U'=U\cup A_1\cup V(F_1)\cup V(F_2).$ Since $|A_1|\leq 10mr\leq 300m^2$, we have $|U'|\leq 2D+300m^2+2D'\leq 5D'$ for sufficiently large $n$. Hence $G-U'$ contains a $(D',m/2,1)$-adjuster $(v_3,F_3,v_4,F_4,A_2)$.

Moreover, $|U\cup A_1\cup A_2|\leq 2D+300m^2+5m\leq \frac{20D'}{\log^3n}$ for sufficiently large $n$. Applying Lemma~\ref{new-connect}, and then taking a minimal subpath between the two indicated unions, gives a path $P$ of length at most $m/40$ from $V(F_1)\cup V(F_2)$ to $V(F_3)\cup V(F_4)$ in $G-(U\cup A_1\cup A_2)$ whose internal vertices avoid all four end expansions. Relabelling if necessary, assume that $P$ runs from $V(F_1)$ to $V(F_3)$. Since $F_1$ is a $(D',m)$-expansion of $v_1$ and $F_3$ is a
$(D',m/2)$-expansion of $v_3$, there is a $v_1,v_3$-path
$Q\subseteq F_1\cup P\cup F_3$ with $\ell(Q)\leq m+\frac{m}{40}+\frac{m}{2}<2m.$ Note that $|A_1\cup A_2\cup V(Q)|
 \leq 10mr+5m+(2m+1)
 \leq 10(r+1)m$.

Then $(v_2,F_2,v_4,F_4,A_1\cup A_2\cup V(Q))$ is a $(D',m,r+1)$-adjuster. Indeed, \ref{d-a-1} and \ref{d-a-2} follow from the choice of $P$ and $Q$, and $|A_1\cup A_2\cup V(Q)|\leq 10mr+5m+2m\leq 10(r+1)m,$ 
so \ref{d-a-3} holds. Finally, let $\ell_i$ be the length of the $i$th adjuster and put $\ell=\ell_1+\ell_2+\ell(Q)$. For every $i\in\{0,1,\ldots,r+1\}$ choose $i_1\in\{0,1,\ldots,r\}$ and $i_2\in\{0,1\}$ with $i=i_1+i_2$. Concatenating a $v_2,v_1$-path of length $\ell_1+2i_1$ in $A_1\cup\{v_1,v_2\}$, the path $Q$, and a $v_3,v_4$-path of length $\ell_2+2i_2$ in $A_2\cup\{v_3,v_4\}$ proves \ref{d-a-4}.

Thus $G-U$ contains a $(D',m,r)$-adjuster. Applying Proposition~\ref{prop-trimming} to each of its two ends, and shortening them from order $D'$ to order $D$, gives the required $(D,m,r)$-adjuster in $G-U$.
\end{proof}

Combining Lemma~\ref{lem-adj-to-adj-path-new} with Lemma~\ref{longconnect4-new}, we can finally find paths with exactly some desired length, as follows.

\begin{lemma}\label{lem-finalconnect-new}
 For every sufficiently small $\eps_1>0$ and  any $0<\eps_2<1/5$, there exists $d_0=d_0(\eps_1,\eps_2)$ such that the following holds for each $n\geq d\geq d_0$. Suppose that $G$ is an $n$-vertex $(\eps_1,\eps_2 d)$-expander with $\de(G)\ge d$ and the following property.

\stepcounter{propcounter}
\begin{enumerate}[label = \textup{\textbf{H}}]
  \item There are no disjoint sets $U_0,W_0\subset V(G)$ such that $|U_0|\leq \log^{20}n$, $|W_0|\geq \log^{80}n$ and every vertex in $W_0$ has at least $(1-2\eps_2)d$ neighbours in $U_0$.\label{newpropP-2}
  \end{enumerate}

Let $D=\log^{10} n$ and $m=\frac{1600}{\eps_1}\log^3n$. Suppose $F_1,F_2\subseteq G$ are vertex disjoint such that $F_i$ is a $(D,m)$-expansion of $v_i$, for each $i\in[2]$. Then, for each $\ell_0$ with $\log^{7}n\leq \ell_0\leq n/\log^{12}n$, there is some $j\in [2]$ such that, for each $\ell\equiv j\pmod 2$ with
$\ell_0\leq \ell\leq \ell_0+2m$ there is a $v_1,v_2$-path with length $\ell$ in $G$.
\end{lemma}
\begin{proof} Let $\eps_1>0$ be sufficiently small that, for every $0<\eps_2<1/5$, we can take $d_0=d_0(\eps_1,\eps_2)$ so that the property in Lemma~\ref{lem-adj-to-adj-path-new} holds. Take, then,  $n\geq d\geq d_0$ and let $G$ be an $n$-vertex $(\eps_1,\eps_2 d)$-expander with $\de(G)\ge d$ and the property \ref{newpropP-2}.
By the property from Lemma~\ref{lem-adj-to-adj-path-new}, there is a $(D,m,24m)$-adjuster, $\cA=(v_3,F_3,v_4,F_4,A)$ say, in $G-V(F_1\cup F_2)$ with length $\ell(\cA)\le |A|+1\leq 500m^2$. Let $\bar{\ell}=\ell_0-22m-\ell(\cA)$,
so that $0\leq \bar{\ell}\le n/\log^{12}n$. As $|A|\leq 500m^2\leq D/\log^3n$, by Lemma~\ref{longconnect4-new}, there are paths $P$ and $Q$ in $G-A$ which are vertex-disjoint, both connect $\{v_1,v_2\}$ to $\{v_3,v_4\}$ and so that $\bar{\ell}\leq \ell(P)+\ell(Q)\leq \bar{\ell}+22m$.
Note that we can assume, without loss of generality, that $P$ is a $v_1,v_3$-path and $Q$ is a $v_2,v_4$-path.

Now, $0\leq \ell_0-\ell(P)-\ell(Q)-\ell(\cA)\leq 22m$. As $\cA$ is a $(D,m,24m)$-adjuster there is a $v_3,v_4$-path in $G[A\cup \{v_3,v_4\}]$ with length $\ell(\cA)$. Let $j\in [2]$ be such that $j\equiv \ell(P)+\ell(Q)+\ell(\cA) \pmod 2$.
Then, for each $\ell\equiv j\pmod 2$ with
$\ell_0\leq \ell\leq \ell_0+2m$, let $i$ be a nonnegative integer such that $2i=\ell-\ell(P)-\ell(Q)-\ell(\cA)$, where $i\leq 13m$. Therefore, by the property of the adjuster, there is a $v_3,v_4$-path, $R$ say, with length $\ell(\cA)+2i=\ell-\ell(P)-\ell(Q)$ in $G[A\cup\{v_3,v_4\}]$. Then, $P\cup R\cup Q$ is a $v_1,v_2$-path with length $\ell$ in~$G$.
\end{proof}


\subsubsection{Proof of Theorem~\ref{mainthm-new}}\label{sec:mainthmfinal-new}
Finally, we combine Lemmas~\ref{lem-expansion-new} and~\ref{lem-finalconnect-new} to prove Theorem~\ref{mainthm-new}.

\begin{proof}[Proof of Theorem~\ref{mainthm-new}]
Let $\eps_1>0$ be such that the property in Lemma~\ref{lem-finalconnect-new} holds.  Let $0<\eps_2<1/5$. Let $d_0=d_0(\ep_1,\ep_2)$ be large and let $n\geq d\geq d_0$. Suppose then that $\Gamma$ is an $n$-vertex $(\eps_1,\eps_2d)$-expander
with $d(\Gamma)\geq 2d$ and $\delta(\Gamma)\geq d$. 
Assuming that \ref{prop:newEHthm3} does not hold, we will now show that \ref{prop:newEHthm2} holds, completing the proof of the theorem.

For \ref{prop:newEHthm2}, let $x,y\in V(\Gamma)$ be distinct and let $\ell\in [\log^{7}n,n/(4\log^{12}n)]$.
Let $m=\frac{1600}{\eps_1}\log^3n$ and $D=\log^{10}n$. Then, by Lemma~\ref{lem-expansion-new} (applied with $C$ taken to be an arbitrary shortest cycle in $\Gamma$), there are vertex-disjoint graphs $F_x,F_y\subseteq \Gamma$ so that $F_x$ is a $(D,m)$-expansion of $x$ and $F_y$ is a $(D,m)$-expansion of $y$. As \ref{prop:newEHthm3} does not hold, we have that \ref{newpropP-2} holds.
By Lemma~\ref{lem-finalconnect-new}, there thus is some $r\in [2]$ such that  (taking $\ell_0=2\ell-1$) there is an $x,y$-path in $\Gamma$ with the unique length in $\{2\ell,2\ell+1\}$ congruent to $r$ modulo $2$, namely $\ell_0+3-r$, as required.
\end{proof}
\end{document}